\documentclass[11pt,a4paper]{article}
\usepackage[utf8]{inputenc}
\usepackage{amsmath}
\usepackage{amsfonts}
\usepackage{amsthm}
\usepackage{amssymb}
\usepackage{graphicx}
\usepackage[margin=1in]{geometry}
\usepackage{hyperref}

\usepackage{mathtools}
\usepackage{xcolor}
\usepackage{enumitem}
\usepackage{mathabx}
\usepackage{cases}
\usepackage{stmaryrd}
\usepackage[framemethod=tikz]{mdframed}
\usepackage{rotating}
\usepackage{tabularx}
\usepackage{pifont}

\newtheorem{thm}{Theorem}
\newtheorem{lem}[thm]{Lemma}
\newtheorem{prop}[thm]{Proposition}

\newtheorem{algo}{Algorithm}

\theoremstyle{definition}

\newtheorem{rmk}[thm]{Remark}
\newtheorem{ex}[thm]{Example}

\setlist[enumerate]{label=$\rm{(\roman*)}$,leftmargin=\parindent}
\numberwithin{equation}{section}
\numberwithin{thm}{section}
\numberwithin{table}{section}
\numberwithin{figure}{section}

\newcommand{\sR}{\mathbb{R}}
\newcommand{\sS}{\mathbb{S}}
\newcommand{\sH}{\mathcal{H}}
\newcommand{\sG}{\mathcal{G}}
\newcommand{\sK}{\mathcal{K}}
\newcommand{\sX}{\mathcal{X}}
\newcommand{\sY}{\mathcal{Y}}

\newcommand{\mysum}{\displaystyle \sum\limits}

\newcommand{\dom}{\mathrm{dom}}
\newcommand{\Id}{\mathrm{Id}}
\newcommand{\bO}{\mathcal{O}}

\newcommand{\E}{\mathcal{E}}

\newcommand{\Q}{\mathcal{Q}}

\newcommand{\W}{\mathcal{W}}
\newcommand{\Lag}{\mathcal{L}}
\newcommand{\Lb}{\mathcal{L}_{\beta}}
\newcommand{\TL}{\mathcal{T}_{\Lag}}

\newcommand{\sol}{\mathcal{S}}
\newcommand{\Fea}{\mathcal{F}}
\newcommand{\sB}{\mathbb{B}}

\newcommand{\tx}{\widetilde{x}}
\newcommand{\bx}{\widebar{x}}
\newcommand{\tlambda}{\widetilde{\lambda}}
\newcommand{\blambda}{\widebar{\lambda}}
\newcommand{\tnu}{\widetilde{\nu}}

\newcommand{\CAstar}{C_{0}}
\newcommand{\Csup}{C_{1}}
\newcommand{\CLag}{C_{2}}
\newcommand{\Cite}{C_{3}}
\newcommand{\Clog}{C_{4}}
\newcommand{\Ctel}{C_{5}}
\newcommand{\Cbnd}{C_{6}}
\newcommand{\Cnu}{C_{7}}

\title{Fast Augmented Lagrangian Method in the convex regime with convergence guarantees for the iterates}
\author{Radu Ioan Bo\c{t}\footnote{Faculty of Mathematics, University of Vienna, Oskar-Morgenstern-Platz 1, 1090 Vienna, Austria, e-mail: \url{radu.bot@univie.ac.at}. Research partially supported by FWF (Austrian Science Fund), project W 1260.}
\and Ern\"{o} Robert Csetnek\footnote{Faculty of Mathematics, University of Vienna, Oskar-Morgenstern-Platz 1, 1090 Vienna, Austria, e-mail: \url{robert.csetnek@univie.ac.at}. Research partially supported by FWF (Austrian Science Fund), project P 29809-N32.}
\and Dang-Khoa Nguyen\footnote{Faculty of Mathematics, University of Vienna, Oskar-Morgenstern-Platz 1, 1090 Vienna, Austria, e-mail: \url{dang-khoa.nguyen@univie.ac.at}. Research supported by FWF (Austrian Science Fund), project P 29809-N32.}}

\begin{document}
	
\maketitle	

\begin{abstract}
This work aims to minimize a continuously differentiable convex function with Lipschitz continuous gradient under linear equality constraints. The proposed inertial algorithm results from the discretization of the second-order primal-dual dynamical system with asymptotically vanishing damping term addressed by Bo\c t and Nguyen in \cite{Bot-Nguyen}, and it is formulated in terms of the Augmented Lagrangian associated with the minimization problem. The general setting we consider for the inertial parameters covers the three classical rules by Nesterov, Chambolle-Dossal and Attouch-Cabot used in the literature to formulate fast gradient methods. For these rules, we obtain in the convex regime convergence rates of order $\bO \left( 1/k^{2} \right)$ for the primal-dual gap, the feasibility measure, and the objective function value.  In addition, we prove that the generated sequence of primal-dual iterates converges to a primal-dual solution in a general setting that covers the two latter rules. This is the first result which provides the convergence of the sequence of iterates generated by a fast algorithm for linearly constrained convex optimization problems without additional assumptions such as strong convexity. We also emphasize that all convergence results of this paper are compatible with the ones obtained in \cite{Bot-Nguyen} in the continuous setting.
\end{abstract}	

\vspace{1ex}

	
\section{Introduction}	

\subsection{Problem formulation and motivation}

Consider the optimization problem
\begin{align}
\label{intro:pb}
\begin{array}{rl}
\min & f \left( x \right), \\
\textrm{subject to} 	& Ax = b
\end{array}
\end{align}
where $\sH, \sG$ are real Hilbert spaces, $f \colon \sH \to \sR$ is a convex and Fr\'echet differentiable function with $L-$Lipschitz continuous gradient, for $L >0$, $A \colon \sH \to \sG$ is a continuous linear operator and $b \in \sG$. We assume that the set $\sol$ of primal-dual optimal solutions of \eqref{intro:pb} (see Section \ref{augLagrangian} for a precise definition) is nonempty. 

Optimization problems of type \eqref{intro:pb} arise in many applications in areas like image recovery \cite{Goldstein-ODonoghue-Setzer-Baraniuk,Ouyang-Chen-Lan-PasiliaoJr, Tran-Dinh-Fercoq-Cevher, Xu}, machine learning \cite{Boyd-et.al,Lin-Li-Fang}, and network optimization \cite{Zeng-Lei-Chen}.

Other than in the unconstrained case, for which fast continuous and discrete time approaches have been intensively investigated in the last years, the study of solution methods with fast convergence rates for linearly constrained convex optimization problems of the form \eqref{intro:pb} is in an incipient stage.

Zeng, Lei, and Chen (in \cite{Zeng-Lei-Chen}) and He, Hu, and Fang (in \cite{He-Hu-Fang:ds}) have investigated a dynamical system with asymptotic vanishing damping  attached to \eqref{intro:pb}, and have shown a convergence rate of order $\bO \left( 1/t^{2} \right)$ for the primal-dual gap, while Attouch, Chbani, Fadili and Riahi have considered in \cite{Attouch-Chbani-Fadili-Riahi} a more general dynamical system with time rescaling. More recently, for a primal-dual dynamical system formulated in the spirit of \cite{Attouch-Chbani-Fadili-Riahi, He-Hu-Fang:ds, Zeng-Lei-Chen}, Bo\c t and Nguyen have obtained in \cite{Bot-Nguyen} fast convergence rates for the primal-dual gap, the feasibility measure and the objective function value along the generated trajectory, and, additionally, have proved asymptotic convergence guarantees for the primal-dual trajectory to a primal-dual optimal solution.

Fast numerical methods for solving \eqref{intro:pb} have been mainly considered in the literature under additional assumptions such as strong convexity, and in several cases the convergence rate results have been formulated in terms of ergodic sequences. In the merely convex regime no convergence result for the iterates has been provided so far for fast convergence algorithms. To the works addressing fast converging methods for linearly constrained convex optimization problems belong \cite{Chambolle-Pock, Chen-Lan-Ouyang, Goldstein-ODonoghue-Setzer-Baraniuk, He-Hu-Fang, Lou:1,Lou:2, Ouyang-Chen-Lan-PasiliaoJr, Sabach-Teboulle, Tran-Dinh-Fercoq-Cevher, Tseng,  Xu, Yan-He}, at which we will take a closer look in Section \ref{subsec:related}. 

The aim of this paper is to propose a numerical algorithm for solving \eqref{intro:pb}, which results from the discretization of the dynamical system in \cite{Bot-Nguyen}, exhibits fast convergence rates for the primal-dual gap, the feasibility measure, and the objective function value as well as convergence of the sequence of iterates without additional assumptions such as strong convexity. Although there is an obvious interplay between continuous time dissipative dynamical systems and their discrete counterparts, one cannot directly and straightforwardly transfer asymptotic results from the continuous setting to numerical algorithms, thus, a separate analysis is needed for the latter. In this paper we will also comment on the similarities and the differences between the continuous and discrete time approaches.

\subsection{Augmented Lagrangian formulation}\label{augLagrangian}

Consider the saddle point problem associated to problem \eqref{intro:pb}
\begin{equation*}
\min_{x \in \sH} \max_{\lambda \in \sG} \Lag \left( x , \lambda \right),
\end{equation*}
where $\Lag \colon \sH \times \sG \to \sR$ denotes the Lagrangian function
\begin{equation*}
\Lag \left( x , \lambda \right) := f \left( x \right) + \left\langle \lambda , Ax - b \right\rangle .
\end{equation*}
Since $f$ is a convex function, $\Lag$ is convex with respect to $x \in \sH$ and affine with respect to $\lambda \in \sG$. A pair $\left( x_{*} , \lambda_{*} \right) \in \sH \times \sG$ is said to be a saddle point of the Lagrangian function $\Lag$ if for every $\left( x , \lambda \right) \in \sH \times \sG$
\begin{equation*}
\Lag \left( x_{*} , \lambda \right) \leq \Lag \left( x_{*} , \lambda_{*} \right) \leq \Lag \left( x , \lambda_{*} \right).
\end{equation*}
If $\left( x_{*} , \lambda_{*} \right) \in \sH \times \sG$ is a saddle point of $\Lag$, then $x_{*} \in \sH$ is an optimal solution of \eqref{intro:pb} and $\lambda_{*} \in \sG$ is an optimal solution of its Lagrange dual problem. If $x_{*} \in \sH$ is an optimal solution of \eqref{intro:pb} and a suitable constraint qualification is fulfilled (see, for instance, \cite{Bauschke-Combettes:book,Bot:book}), then there exists an optimal solution $\lambda_{*} \in \sG$ of the Lagrange dual problem of \eqref{intro:pb} such that $\left( x_{*} , \lambda_{*} \right) \in \sH \times \sG$ is a saddle point of $\Lag$. 

The set of saddle points of $\Lag$, called also set of primal-dual optimal solutions of \eqref{intro:pb}, will be denoted by $\sol$ and, as stated above, throughout this paper it will be assumed to be nonempty.  The set of feasible points of \eqref{intro:pb} will be denoted by $\Fea := \left\lbrace x \in \sH \colon Ax = b \right\rbrace$ and the optimal objective value of \eqref{intro:pb}  by $f_{*}$.

The system of primal-dual optimality conditions for \eqref{intro:pb} reads
\begin{equation}
\label{intro:opt-Lag}
\left( x_{*} , \lambda_{*} \right) \in \sol
\Leftrightarrow \begin{cases}
\nabla_{x} \Lag \left( x_{*} , \lambda_{*} \right) 			& = 0 \\
\nabla_{\lambda} \Lag \left( x_{*} , \lambda_{*} \right) 	& = 0
\end{cases} \Leftrightarrow \begin{cases}
\nabla f \left( x_{*} \right) + A^{*} \lambda_{*} 	& = 0 \\
Ax_{*} - b 													& = 0
\end{cases},
\end{equation}
where $A^{*} : \sG \rightarrow \sH$ denotes the adjoint operator of $A$.
This optimality system can be equivalently written as
\begin{equation*}
\TL \left( x_{*} , \lambda_{*} \right) = 0 ,
\end{equation*}
where
\begin{equation*}
\TL \colon \sH \times \sG \to \sH \times \sG, \quad \TL \left( x , \lambda \right)
= \begin{pmatrix}
\nabla_{x} \Lag \left( x , \lambda \right) \\ - \nabla_{\lambda} \Lag \left( x , \lambda \right)
\end{pmatrix}
= \begin{pmatrix}
\nabla f \left( x \right) + A^{*} \lambda \\ b-Ax
\end{pmatrix},
\end{equation*}
is the maximally monotone operator associated to the convex-concave function $\Lag$. Indeed, it is immediate to verify that $\TL$ is monotone. Since it is also continuous, it is maximally monotone (see, for instance, \cite[Corollary 20.28]{Bauschke-Combettes:book}).
Therefore $\sol$ can be interpreted as the set of zeros of the maximally monotone operator $\TL$, which means that it is a closed convex subset of $\sH \times \sG$ (see, for instance, \cite[Proposition 23.39]{Bauschke-Combettes:book}).

For $\beta \geq 0$, we also consider the augmented Lagrangian $\Lb \colon \sH \times \sG \to \sR$ associated with \eqref{intro:pb}
\begin{equation*}
\Lb \left( x , \lambda \right) := \Lag \left( x , \lambda \right) + \dfrac{\beta}{2} \left\lVert Ax - b \right\rVert ^{2} = f \left( x \right) + \left\langle \lambda , Ax - b \right\rangle + \dfrac{\beta}{2} \left\lVert Ax - b \right\rVert ^{2} .
\end{equation*}
For every $(x, \lambda) \in \Fea \times \sG$ it holds
\begin{equation}
\label{intro:Fea:eq}
f \left( x \right) = \Lb \left( x , \lambda \right) = \Lag \left( x , \lambda \right) .
\end{equation}
If $\left( x_{*} , \lambda_{*} \right) \in \sol$, then for every $\left( x , \lambda \right) \in \sH \times \sG$ we have
\begin{equation}
\label{intro:saddle}
\Lag \left( x_{*} , \lambda \right) = \Lb \left( x_{*} , \lambda \right) = \Lag \left( x_{*} , \lambda_{*} \right) = \Lb \left( x_{*} , \lambda_{*} \right)
\leq \Lag \left( x , \lambda_{*} \right) \leq \Lb \left( x , \lambda_{*} \right).
\end{equation}
In addition,
\begin{equation*}
\left( x_{*} , \lambda_{*} \right) \in \sol
\Leftrightarrow \begin{cases}
\nabla_{x} \Lb \left( x_{*} , \lambda_{*} \right) 			& = 0 \\
\nabla_{\lambda} \Lb \left( x_{*} , \lambda_{*} \right) 	& = 0
\end{cases} \Leftrightarrow \begin{cases}
\nabla f \left( x_{*} \right) + A^{*} \lambda_{*} 	& = 0 \\
Ax_{*} - b 													& = 0
\end{cases} .
\end{equation*}

\subsection{Related works}
\label{subsec:related}

In this section we will recall the most significant fast primal-dual numerical approaches for linearly constrained convex optimization problems and for convex optimization problems involving compositions with continuous linear operators.

In  \cite{Chambolle-Pock},  Chambolle and Pock have studied in a finite-dimensional setting the convergence rates of their celebrated primal-dual algorithm for solving the minimax problem
\begin{equation}
\label{intro:pb-minmax}
\min_{x \in \sH} \max_{\lambda \in \sG} {\cal L} (x,\lambda) := f \left( x \right) + \left\langle Ax , \lambda \right\rangle - g^{*} \left( \lambda \right),
\end{equation}
which is naturally attached to the convex optimization problem
\begin{equation}\label{intro:composed}
\min_{x \in \sH} f(x) + g(Ax), 
\end{equation}
with $f \colon \sH \to \sR \cup \left\lbrace + \infty \right\rbrace$ and $g \colon \sG \to \sR \cup \left\lbrace + \infty \right\rbrace$ proper, convex and lower semicontinuous functions and $g^{*} \colon \sG \to \sR \cup \left\lbrace + \infty \right\rbrace$ the Fenchel conjugate of $g$. The problem \eqref{intro:composed} becomes  \eqref{intro:pb} for $g$ the indicator function of the set $\{b\}$.  For the primal-dual sequence of iterates  $\left\lbrace (x_{k} , \lambda_{k}) \right\rbrace _{k \geq 0}$ the corresponding ergodic sequence $\left\lbrace (\bx_{k} , \blambda_{k}) \right\rbrace _{k \geq 0}$ is defined for every $k \geq 0$ as
\begin{equation*}
\bx_{k} := \dfrac{1}{\sum_{i = 0}^{k} \sigma_{i}} \mysum_{i = 0}^{k} \sigma_{i} x_{i}
\qquad \textrm{ and } \qquad
\blambda_{k} := \dfrac{1}{\sum_{i = 0}^{k} \sigma_{i}} \mysum_{i = 0}^{k} \sigma_{i} \lambda_{i},
\end{equation*}
where $\left\lbrace \sigma_{k} \right\rbrace _{k \geq 0}$ is a sequence of properly chosen positive step sizes.  The Chambolle-Pock primal-dual algorithm exhibits for the restricted primal-dual gap an ergodic convergence rate of
\begin{equation*}
\sup \limits_{\left( x , \lambda \right) \in \sX \times \sY} \left( \Lag \left(\bx_{k} , \lambda \right) - \Lag \left( x , \blambda_{k} \right) \right) = \bO \left( \dfrac{1}{k} \right) \ \mbox{as} \ k \rightarrow +\infty,
\end{equation*}
where $\sX \subseteq \sH$ and $\sY \subseteq \sG$ are bounded sets.  If $f$ is strongly convex, then the accelerated variant of this primal-dual algorithm exhibits for the same restricted primal-dual gap an ergodic convergence rate of 
$$\sup \limits_{\left( x , \lambda \right) \in \sX \times \sY} \left( \Lag \left(\bx_{k} , \lambda \right) - \Lag \left( x , \blambda_{k}\right) \right) = \bO \left( \dfrac{1}{k^2} \right) \ \mbox{as} \ k \rightarrow +\infty$$
whereas, if both $f$ and $g^*$ are strongly convex, then even linear convergence can be achieved.

In \cite{Chen-Lan-Ouyang}, Chen, Lan and Ouyang have considered the same minimax problem \eqref{intro:pb-minmax}, but for $f \colon \sH \to \sR$ a convex and Fr\'echet differentiable function with $L$-Lipschitz continuous gradient, for $L>0$, and have proposed a primal-dual algorithm that  exhibits for the restricted primal-dual gap an ergodic convergence rate of
\begin{equation}\label{mixed}
\sup\limits_{\left( x , \lambda \right) \in \sX \times \sY} \left( \Lag \left(\bx_{k} , \lambda \right) - \Lag \left( x , \blambda_{k} \right) \right) = \bO \left( \dfrac{L}{k^{2}} + \dfrac{\left\lVert A \right\rVert}{k} \right) \ \mbox{as} \ k \rightarrow +\infty.
\end{equation}
A stochastic counterpart of the primal-dual algorithm along with corresponding convergence rate results and, for both the deterministic and the stochastic setting, convergence rates when either $\sX$ or $\sY$ is unbounded have been also provided.

Later on, Ouyang, Chen, Lan and Pasiliao Jr.  have developed in  \cite{Ouyang-Chen-Lan-PasiliaoJr} an accelerated ADMM algorithm for the optimization problem \eqref{intro:composed} with $f$ assumed to be Fr\'echet differentiable with $L$-Lipschitz continuous gradient, for $L>0$,  on its effective domain.  In the case when $f$ and $g^*$ have bounded domains this method has been proved to exhibit an ergodic convergence rate for the objective function value of type \eqref{mixed},  with the coefficient of $1/{k^2}$ depending on $L$ and the diameter of $\dom f$ and the coefficient of $1/k$ depends on $\|A\|$ and of the diameter of $\dom g^*$.  On the other hand, without assuming boundedness for the domains of $f$ and $g^*$, the accelerated ADMM algorithms has been proved to exhibit ergodic convergence rates for the feasibility measure and the objective function value of $\bO \left( 1/k \right)$ as $k \rightarrow +\infty$.

By using a smoothing approach, Tran-Dinh, Fercoq and Cevher have designed in \cite{Tran-Dinh-Fercoq-Cevher} a  primal-dual algorithm for solving \eqref{intro:composed} and its particular formulation \eqref{intro:pb} that exhibits  last iterates convergence rates for the objective function value and the feasibility measure in the convex regime of $\bO \left( 1/k \right)$, and in the strongly convex regime of $\bO \left( 1/{k^2} \right)$ as $k \rightarrow +\infty$.

Goldstein, O'Donoghue, Setzer and Baraniuk have studied in \cite{Goldstein-ODonoghue-Setzer-Baraniuk} the two-block separable optimization problem with linear constraints
\begin{align}
\label{intro:pb-2b}
\begin{array}{rl}
\min & f \left( x \right) + h \left( y \right), \\
\textrm{subject to} 	& Ax+By = b
\end{array}
\end{align}
where $\sK$ is another real Hilbert space, $f \colon \sH \to  \sR \cup \left\lbrace + \infty \right\rbrace$ and  $h \colon \sK \to \sR \cup \left\lbrace + \infty \right\rbrace$ are proper, convex and lower semicontinuous functions, $A \colon \sH \to \sG$ and  $B \colon \sK \to \sG$ are continuous linear operators and $b \in \sG$.  It is obvious that \eqref{intro:pb} can be reformulated as  \eqref{intro:pb-2b} and vice versa. In \cite{Goldstein-ODonoghue-Setzer-Baraniuk} a numerical algorithm for solving \eqref{intro:pb-2b} has been proposed that exhibits, when $f$ and $h$ are strongly convex, convergence rates for the dual objective function of $\bO \left( 1/k^{2} \right)$ and for the feasibility measure of $\bO \left( 1/k \right)$ as $k \rightarrow +\infty$.  For a fast version of the Alternating Minimization Algorithm (see \cite{Tseng}) a convergence rate for the dual objective function of $\bO \left( 1/k^{2} \right)$ as $k \rightarrow +\infty$ has been also proved.

Xu has proposed in \cite{Xu} a linearized Augmented Lagrangian Method for the optimization problem \eqref{intro:pb} for which he has shown that it exhibits for constant step sizes  ergodic convergence rates of $\bO \left( 1/k \right)$ as $k \rightarrow +\infty$ for the feasibility measure and the objective function value, whereas the sequence of primal-dual iterates has been shown to converge to a primal-dual solution.  He has also proved that for appropriately chosen variable step sizes, in particular when allowing the dual step sizes to be unbounded,  the convergence rates of the feasibility measure and the objective function value can be improved to $\bO \left( 1/{k^2} \right)$ as $k \rightarrow +\infty$,  without saying anything about the convergence of the primal-dual iterates in this setting.  In addition,  a linearized Alternating Direction Method of Multipliers for \eqref{intro:pb-2b} has been proposed in \cite{Xu}, for which similar statements as for the linearized Augmented Lagrangian Method have been proved, whereby the fast convergence rates have been obtained by assuming that one of the summands in the objective function is strongly convex.

In \cite{He-Yuan},  He and Yuan have enhanced the Augmented Lagrangian Method for the linearly constrained convex optimization problem \eqref{intro:pb} with a Nesterov's momentum update rule for the sequence of dual iterates.  They have proved that the expression $\Lag(x_*,\lambda_*) - \Lag \left(x_{k} , \lambda_k \right)$ has an upper bound of order $1/k^2$, where $\left(x_{k} , \lambda_k \right)_{k \geq 0}$ denotes the generated sequence of primal-dual iterates and $(x_*,\lambda_*)$ is an arbitrary optimal solution of the Wolfe dual problem of  \eqref{intro:pb}.

In \cite{Yan-He}, Yan and He have proposed for optimization problems of type \eqref{intro:pb}, with a proper, convex and lower semicontinuous objective function, a numerical algorithm which combines the Augmented Lagrangian Method with a Bregman proximal evaluation of the objective. When choosing the sequence of proximal parameter to fulfil $\eta_{k} := \eta \left( k + 1 \right) ^{p}$ for every $k \geq 0$, where $\eta >0$ and $p \geq 0$, ergodic convergence rates of 
$$\sup \limits_{\left( x , \lambda \right) \in \sX \times \sY} \left( \Lag \left(\bx_{k} , \lambda \right) - \Lag \left( x , \blambda_{k}\right) \right) = \bO \left( \dfrac{1}{k^{p+2}} \right)  \ \mbox{as} \ k  \rightarrow +\infty,$$
$$\left\lVert A \bx_{k} - b \right\rVert = \bO \left( \dfrac{\log \left( k \right)}{k^{p+2}} \right) \mbox{ and } \left\lvert f \left( \bx_{k} \right) - f_{*} \right\rvert = \bO \left( \dfrac{\log \left( k \right)}{k^{p+2}} \right)  \ \mbox{as} \ k  \rightarrow +\infty$$
have been obtained.

In \cite{Sabach-Teboulle}, Sabach and Teboulle have considered a unified algorithmic framework for proving faster convergence rates for various Lagrangian-based methods designed to solve optimization problems of type \eqref{intro:pb} with a proper, convex and lower semicontinuous objective function. In the convex regime these methods exhibit a non-ergodic rate of convergence of $\bO \left( 1/k \right)$ as $k \rightarrow +\infty$ for the feasibility measure and the objective function value, namely,
\begin{equation*}
f \left( x_{k} \right) - f_{*} \ \mbox{has an upper bound of order} \  \bO \left( \dfrac{1}{k} \right) \ \mbox{and} \ \left\lVert Ax_{k} - b \right\rVert =  \bO \left( \dfrac{1}{k} \right) \ \mbox{as} \ k  \rightarrow +\infty.
\end{equation*}
In the strongly convex regime the convergence rates can be improved to $\bO \left( 1/k^{2} \right)$  as $k \rightarrow +\infty$.

For the same class of optimization problems, He, Hu, and Fang have proposed in \cite{He-Hu-Fang} an accelerated primal-dual Lagrangian-based method, with inertial parameters following the choice of Chambolle-Dossal, that achieves a convergence rate of $\bO \left( 1/k^{2} \right)$ as $k \rightarrow +\infty$ for the feasibility measure and the objective function value without any strong convexity assumption. 

Recently, in  \cite{Lou:1}, Lou have introduced in the same context an unifying algorithmic scheme which covers both the convex and the strongly convex setting. In the convex regime a convergence rate of $\bO \left( 1/k \right)$ as $k \rightarrow +\infty$ is obtained for the primal-dual gap, the feasibility measure, and the objective function value, while in the  strongly convex regime these rates are improved to $\bO \left( 1/k^{2} \right)$ as $k \rightarrow +\infty$. These results have been extended to optimization problems of type \eqref{intro:pb-2b} in \cite{Lou:2}, where it has been shown that, in order to achieve a convergence rate of $\bO \left( 1/k^{2} \right)$ as $k \rightarrow +\infty$, it is enough to assume that only one of the functions in the objective is strongly convex.

Noticeably none of theses works has addressed to convergence of the sequences of primal-dual iterates, with very few exceptions in the strongly convex regime. This phenomenon could be noticed for unconstrained convex optimization problems, too. The convergence of the sequences of iterates generated by fast numerical methods has been proved much later (by Chambolle and Dossal in \cite{Chambolle-Dossal} and by Attouch and Peypouquet in \cite{Attouch-Peypouquet}) after the derivation of the convergence rates for Nesterov's accelerated gradient method (\cite{Nesterov:83}) and FISTA (\cite{FISTA}). One explanation for this is that the analysis of the first is much more involved.

\subsection{Our contributions}

We consider as starting point a second-order dynamical system with asymptotic vanishing damping term associated with the optimization problem \eqref{intro:pb}.  This dynamical system is formulated in terms of the augmented Lagrangian and it has been studied in \cite{Bot-Nguyen}. By an appropriate time discretization this system gives rise to an inertial primal-dual numerical algorithm, which allows a flexible choice of the inertial parameters.  This choice covers the three classical inertial parameters rules by Nesterov (\cite{FISTA, Nesterov:83}), Chambolle-Dossal (\cite{Chambolle-Dossal}) and Attouch-Cabot (\cite{Attouch-Cabot:18}) used in the literature to formulate fast gradient methods. We show that for these rules the resulting algorithm exhibits in the convex regime convergence rates of order $\bO \left( 1/k^{2} \right)$ for the primal-dual gap, the feasibility measure, and the objective function value. In addition, we prove that the generated sequence of primal-dual iterates converges weakly to a primal-dual solution of the underlying problem, which is nothing else than a saddle-point of the Lagrangian. The convergence of the iterates is stated in a general setting that covers the inertial parameters rules by Chambolle-Dossal and Attouch-Cabot. This is the first result which provides the convergence of the sequence of iterates generated by a fast algorithm for linearly constrained convex optimization problems without additional assumptions such as strong convexity.  All convergence and convergence rate results of this paper are compatible with the ones obtained in \cite{Bot-Nguyen} in the continuous setting.

The proposed Fast Augmented Lagrangian Method and all convergence results can be easily extended by using the product space approach to two-block separable linearly constrained optimization problems of the form \eqref{intro:pb-2b} with $f$ and $h$ convex and Fr\'echet differentiable functions with Lipschitz continuous gradients.

\subsection{Notations and preliminaries}

We denote by $\sB \left( x ; \varepsilon \right) := \left\lbrace y \in \sH \colon \left\lVert x - y \right\rVert \leq \varepsilon \right\rbrace$ the closed ball centered at $x \in \sH$ with radius $\varepsilon > 0$.

Let $x, y \in \sH$. We have
\begin{equation}
\label{pre:id-eq}
\left\lVert x + y \right\rVert ^{2} = \left\lVert x \right\rVert ^{2} + \left\lVert y \right\rVert ^{2} + 2 \left\langle x , y \right\rangle .
\end{equation}
For every $s, t \in \sR$ such that $s+t=1$ it holds (\cite[Corollary 2.15]{Bauschke-Combettes:book})
\begin{equation}
\label{pre:sum-1}
\left\lVert sx + ty \right\rVert ^{2}
= s \left\lVert x \right\rVert ^{2} + t \left\lVert y \right\rVert ^{2} - st \left\lVert x - y \right\rVert ^{2}.
\end{equation}
From here one can easily deduce that for $s, t \in \sR$ such that $s+t \neq 0$ it holds
\begin{equation}
\label{pre:sum-n0}
\dfrac{1}{s+t} \left\lVert sx + ty \right\rVert ^{2}
= s \left\lVert x \right\rVert ^{2} + t \left\lVert y \right\rVert ^{2} - \dfrac{st}{s+t} \left\lVert x - y \right\rVert ^{2} .
\end{equation}

We denote by $\sS_{+} \left( \sH \right)$ the family of self-adjoint and positive semidefinite continuous linear operators $\W \colon \sH \to \sH$. Every $\W \in \sS_{+} \left( \sH \right)$ induces on $\sH$ a semi-norm defined by
\begin{equation*}
\left\lVert x \right\rVert _{\W}^{2} = \left\langle x , x \right\rangle _{\W} := \left\langle \W x , x \right\rangle \qquad \forall x \in \sH .
\end{equation*}
The Loewner partial ordering on $\sS_{+} \left( \sH \right)$ is defined for $\W, \W' \in \sS_{+} \left( \sH \right)$ as
\begin{equation*}
\W \succcurlyeq \W' \Leftrightarrow \left\lVert x \right\rVert _{\W}^{2} \geq \left\lVert x \right\rVert _{\W'}^{2} \qquad \forall x \in \sH .
\end{equation*}
Thus $\W \in \sS_{+} \left( \sH \right)$ is nothing else than $\W \succcurlyeq 0$.
If there exists $\alpha > 0$ such that  $\W \succcurlyeq \alpha \Id$ then the semi-norm $\left\lVert \cdot \right\rVert _{\W}$ becomes a norm.

In the spirit of \eqref{pre:id-eq} and \eqref{pre:sum-n0}, respectively, for every $x, y \in \sH$ it holds
\begin{equation}
\label{pre:vm:id-eq}
\left\lVert x + y \right\rVert _{\W}^{2} = \left\lVert x \right\rVert _{\W}^{2} + \left\lVert y \right\rVert _{\W}^{2} + 2 \left\langle x, y \right\rangle _{\W},
\end{equation}
and for every real numbers $s, t$ such that $s+t \neq 0$
\begin{equation}
\label{pre:vm:sum-n0}
\dfrac{1}{s+t} \left\lVert sx + ty \right\rVert _{\W}^{2}
= s \left\lVert x \right\rVert _{\W}^{2} + t \left\lVert y \right\rVert _{\W}^{2} - \dfrac{st}{s+t} \left\lVert x - y \right\rVert _{\W}^{2} .
\end{equation}

Let $f \colon \sH \to \sR$ be a continuously differentiable and convex function such that $\nabla f$ is $L-$Lipschitz continuous, for $L >0$. For every $x, y \in \sH$ it holds (see \cite[Theorem 2.1.5]{Nesterov:book} or \cite[Theorem 18.15]{Bauschke-Combettes:book})
\begin{equation}
\label{pre:f-bound}
0 \leq \dfrac{1}{2L} \left\lVert \nabla f \left( x \right) - \nabla f \left( y \right) \right\rVert ^{2} \leq f \left( x \right) - f \left( y \right) - \left\langle \nabla f \left( y \right) , x - y \right\rangle \leq \dfrac{L}{2} \left\lVert x - y \right\rVert ^{2} .
\end{equation}
The second inequality is also known as the Descent Lemma.

The following result is a particular instance of \cite[Lemma 5.31]{Bauschke-Combettes:book} and will be used several times in this paper.
\begin{lem}
	\label{lem:quasi-Fej}
	Let $\left\lbrace a_{k} \right\rbrace _{k \geq 1}$, $\left\lbrace b_{k} \right\rbrace _{k \geq 1}$ and $\left\lbrace d_{k} \right\rbrace _{k \geq 1}$ be sequences of real numbers. Assume that $\left\lbrace a_{k} \right\rbrace _{k \geq 1}$ is bounded from below, and $\left\lbrace b_{k} \right\rbrace _{k \geq 1}$ and $\left\lbrace d_{k} \right\rbrace _{k \geq 1}$ are nonnegative such that $\sum_{k \geq 1} d_{k} < + \infty$. Suppose further that for every $k \geq 1$ it holds
	\begin{equation}
	\label{quasi-Fej:inq}
	a_{k+1} \leq a_{k} - b_{k} + d_{k} .
	\end{equation}
	Then the following statements are true
	\begin{enumerate}
		\item the sequence $\left\lbrace b_{k} \right\rbrace _{k \geq 1}$ is summable, namely $\sum_{k \geq 1} b_{k} < + \infty$;
		\item the sequence $\left\lbrace a_{k} \right\rbrace _{k \geq 1}$ is convergent.
	\end{enumerate}
\end{lem}

In order to establish the weak convergence of the iterates, we will use Opial’s Lemma in discrete form (see, for instance, \cite[Theorem 5.5]{Bauschke-Combettes:book}), which we recall as follows.
\begin{lem}
	\label{lem:Opial}
	Let ${\cal C}$ be a nonempty subset of ${\sH}$ and $\left\lbrace x_{k} \right\rbrace _{k \geq 1}$ a sequence in $\sH$. Assume that
	\begin{enumerate}
		\item for every $x_{*} \in {\cal C}$, $\lim\limits_{k \to + \infty} \left\lVert x_{k} - x_{*} \right\rVert$ exists;
		\item every weak sequential cluster point of $\left\lbrace x_{k} \right\rbrace _{k \geq 1}$ belongs to ${\cal C}$.
	\end{enumerate}
	Then the sequence $\left\lbrace x_{k} \right\rbrace _{k \geq 1}$ converges weakly to an element in ${\cal C}$ as $k \to + \infty$.
\end{lem}

\section{Continuous time approaches and their discrete counterparts}
\label{sec:ds}

In this section we want to derive by time discretization a primal-dual numerical algorithm from the second-order dynamical system investigated in \cite{Bot-Nguyen}. The employed discretization technique replicates the one used when relating fast gradient algorithms with the second-order dynamical system proposed by Su, Boyd and Cand\`es in \cite{Su-Boyd-Candes} in the unconstrained case.

\subsection{The primal-dual dynamical approach with vanishing damping}

The second-order primal-dual dynamical system with asymptotically vanishing damping term associated in \cite{Bot-Nguyen} with the augmented Lagrangian formulation of \eqref{intro:pb} reads
\begin{mdframed}	
	\begin{equation}
	\tag{$\mathrm{PD}$-$\mathrm{AVD}$}
	\label{ds:PD-AVD}
	\begin{dcases}
	\ddot{x} \left( t \right) + \dfrac{\alpha}{t} \dot{x} \left( t \right) + \nabla_{x} \Lb \Bigl( x \left( t \right) , \lambda \left( t \right) + \theta t \dot{\lambda} \left( t \right) \Bigr)  		& = 0 \\
	\ddot{\lambda} \left( t \right) + \dfrac{\alpha}{t} \dot{\lambda} \left( t \right) - \nabla_{\lambda} \Lb \Bigl( x \left( t \right) + \theta t \dot{x} \left( t \right) , \lambda \left( t \right) \Bigr) 	& = 0 \\
	\Bigl( x \left( t_{0} \right) , \lambda \left( t_{0} \right) \Bigr) 			= \Bigl( x_{0} , \lambda_{0} \Bigr) \textrm{ and }
	\Bigl( \dot{x} \left( t_{0} \right) , \dot{\lambda} \left( t_{0} \right) \Bigr) = \Bigl( \dot{x}_{0} , \dot{\lambda}_{0} \Bigr)
	\end{dcases},
	\end{equation}
\end{mdframed}
where $t_{0} >0, \alpha \geq 3, \beta \geq 0, \theta >0$ and $\left( x_{0}, \lambda_{0} \right), \left( \dot{x}_{0} , \dot{\lambda}_{0} \right) \in \sH \times \sG$.

Plugging the expressions of the partial gradients of $\Lb$ into the system leads to the following formulation for \eqref{ds:PD-AVD}
\begin{equation}
\label{ds:Cauchy}
\begin{dcases}
\ddot{x} \left( t \right) + \dfrac{\alpha}{t} \dot{x} \left( t \right) + \nabla f \left( x \left( t \right) \right) + A^{*} \left( \lambda \left( t \right) + \theta t \dot{\lambda} \left( t \right) \right) + \beta A^{*} \Bigl( Ax \left( t \right) - b \Bigr)		& = 0 \\
\ddot{\lambda} \left( t \right) + \dfrac{\alpha}{t} \dot{\lambda} \left( t \right) - \Bigl( A \bigl( x \left( t \right) + \theta t \dot{x} \left( t \right) \bigr) - b \Bigr) 	& = 0 \\
\Bigl( x \left( t_{0} \right) , \lambda \left( t_{0} \right) \Bigr) 			= \Bigl( x_{0} , \lambda_{0} \Bigr) \textrm{ and }
\Bigl( \dot{x} \left( t_{0} \right) , \dot{\lambda} \left( t_{0} \right) \Bigr) = \Bigl( \dot{x}_{0} , \dot{\lambda}_{0} \Bigr)
\end{dcases}.
\end{equation}
In \cite{Bot-Nguyen} it has been shown that, supposing that \begin{equation*}
	\alpha \geq 3 , \quad {\beta \geq 0}
	\quad \textrm{ and } \quad
	\dfrac{1}{2} \geq \theta \geq \dfrac{1}{\alpha - 1},
	\end{equation*}
for a solution $\left( x , \lambda \right) \colon \left[ t_{0} , + \infty \right) \to \sH \times \sG$  of \eqref{ds:PD-AVD} and $\left( x_{*} , \lambda_{*} \right) \in \sol$ it holds for every $t \geq t_0$ 
\begin{equation*}
				0 \leq \Lag \left( x \left( t \right) , \lambda_{*} \right) - \Lag \left( x_{*} , \lambda \left( t \right) \right) + \left\lVert A x \left( t \right) - b \right\rVert \leq \dfrac{\widehat{C}}{\theta^{2} t^{2}}
		\end{equation*}
and
		\begin{equation*}
		- \dfrac{\left\lVert \lambda_{*} \right\rVert \widehat{C}}{\theta^{2} t^{2}} \leq f \left( x \left( t \right) \right) - f_{*} \leq \dfrac{\left( 1 + \left\lVert \lambda_{*} \right\rVert \right) \widehat{C}}{\theta^{2} t^{2}},
		\end{equation*}
where $\widehat{C} >0$.

If, in addition, $\nabla f$ is $L-$Lipschitz continuous, $\alpha >3$ and $\dfrac{1}{2} > \theta > \dfrac{1}{\alpha - 1}$, then it holds
	\begin{equation}\label{ratesgrad}
	\left\lVert A^{*} \left( \lambda \left( t \right) - \lambda_{*} \right) \right\rVert
	= o \left( \dfrac{1}{\sqrt{t}} \right) \quad \mbox{and} \quad \left\lVert \nabla f \left( x \left( t \right) \right) - \nabla f \left( x_{*} \right) \right\rVert
	= o \left( \dfrac{1}{\sqrt{t}} \right) \textrm{ as } t \to + \infty
	\end{equation}
and, consequently,
	\begin{align*}
	\left\lVert \nabla_{x} \Lag \bigl( x \left( t \right) , \lambda \left( t \right) \bigr) \right\rVert
	& = \left\lVert \nabla f \left( x \left( t \right) \right) + A^{*} \lambda \left( t \right) \right\rVert
	= o \left( \dfrac{1}{\sqrt{t}} \right)
	& \textrm{ as } t \to + \infty,
	\end{align*}
whereas
	\begin{align*}
	\left\lVert \nabla_{\lambda} \Lag \bigl( x \left( t \right) , \lambda \left( t \right) \bigr) \right\rVert
	& = \left\lVert Ax \left( t \right) - b \right\rVert = \bO \left( \dfrac{1}{t^{2}} \right)
	& \textrm{ as } t \to + \infty.
	\end{align*}
By additionally requiring that $\beta >0$, it has been also proved in \cite{Bot-Nguyen}  that the trajectory $\bigl( x \left( t \right) , \lambda \left( t \right) \bigr)$ converges weakly to a primal-dual optimal solution of \eqref{intro:pb} as $t \to + \infty$.

\subsection{Fast gradient scheme: from continuous to discrete time}

We recall in this section for reader's convenience the connection between the second-order dynamical system by Su, Boyd and Cand\`{e}s (\cite{Su-Boyd-Candes}) and the fast gradient numerical methods formulated in \cite{Chambolle-Dossal, Attouch-Cabot:18} in the spirit of Nesterov's accelerated gradient algorithm (\cite{Nesterov:83}). To this end we consider the unconstrained optimization problem
\begin{equation}\label{unconstrained:pb}
\min\limits_{x \in \sH} f \left( x \right),
\end{equation}
where $f : \sH \rightarrow {\mathbb R}$ is a convex and Fr\'echet differentiable function with $L$-Lipschitz continuous gradient, for $L>0$.

The continuous time approach proposed in \cite{Su-Boyd-Candes} in connection with this optimization problem reads
\begin{equation}
\tag{$\mathrm{AVD}$}
\label{ds:AVD}
\ddot{x} \left( t \right) + \dfrac{\alpha}{t} \dot{x} \left( t \right) + \nabla f \left( x \left( t \right) \right)   = 0 ,
\end{equation}
where $t_{0} > 0$ and $\alpha \geq 3$. One can easily notice that for $A=0$ and $b=0$ the optimization problem \eqref{intro:pb} becomes \eqref{unconstrained:pb}, while \eqref{ds:PD-AVD} reduces to \eqref{ds:AVD}.

For every $t \geq t_{0}$, we define
\begin{equation*}
z \left( t \right) := x \left( t \right) + \dfrac{t}{\alpha - 1} \dot{x} \left( t \right) .
\end{equation*}
This leads to
\begin{align*}
\dot{z} \left( t \right)
& = \dot{x} \left( t \right) + \dfrac{1}{\alpha - 1} \dot{x} \left( t \right) + \dfrac{t}{\alpha - 1} \ddot{x} \left( t \right)
= \dfrac{t}{\alpha - 1} \ddot{x} \left( t \right) + \dfrac{\alpha}{\alpha - 1} \dot{x} \left( t \right) = - \dfrac{t}{\alpha - 1} \nabla f \left( x \left( t \right) \right)
\end{align*}
and \eqref{ds:AVD} can be written as a first-order ordinary differential equation
\begin{equation}
\label{ds:AVD:fo}
\begin{dcases}
\dot{z} \left( t \right) 	& = - \dfrac{t}{\alpha - 1} \nabla f \left( x \left( t \right) \right) \\
z \left( t \right) 			& = x \left( t \right) + \dfrac{t}{\alpha - 1} \dot{x} \left( t \right).
\end{dcases}
\end{equation}

Let $\sigma > 0$. For every $k \geq 1$ we take as time step
\begin{equation*}
\sigma_{k} := \sigma \left( 1 + \dfrac{\alpha - 1}{k} \right) > 0 ,
\end{equation*}
and set $\tau_{k} := \sqrt{\sigma_{k}} k = \sqrt{\sigma k \left( k + \alpha - 1 \right)} \approx \sqrt{\sigma} \left( k + 1 \right)$, $x \left( \tau_{k} \right) \approx x_{k+1}$ and $z \left( \tau_{k} \right) \approx z_{k+1}$. We ``approximate'' $\tau_k$ with $\sqrt{\sigma} \left( k + 1 \right)$ since it is closer to this value than to $\sqrt{\sigma} k$. This also explains why we consider
$x \left( \tau_{k} \right) \approx x_{k+1}$ and $z \left( \tau_{k} \right) \approx z_{k+1}$ instead of the seemingly more natural choices $x \left( \tau_{k} \right) \approx x_{k}$ and $z \left( \tau_{k} \right) \approx z_{k}$, respectively.

The implicit finite-difference scheme for \eqref{ds:AVD:fo} at time $t := \tau_{k}$ gives
\begin{equation*}
\begin{dcases}
\dfrac{z_{k+1} - z_{k}}{\sqrt{\sigma_{k}}} 	& = - \dfrac{\sqrt{\sigma_{k}} k}{\alpha - 1} \nabla f \left( y_{k} \right) \\
z_{k+1} 					& = x_{k+1} + \dfrac{\sqrt{\sigma_{k}} k}{\alpha - 1} \dfrac{x_{k+1} - x_{k}}{\sqrt{\sigma_{k}}}
\end{dcases}
\end{equation*}
or, equivalently,
\begin{equation}
\label{ds:im-fd}
\begin{dcases}
z_{k+1} - z_{k} 	& = - \sigma \left( 1 + \dfrac{\alpha - 1}{k} \right) \dfrac{k}{\alpha - 1} \nabla f \left( y_{k} \right) \\
z_{k+1} 					& = x_{k+1} + \dfrac{k}{\alpha - 1} \left( x_{k+1} - x_{k} \right),
\end{dcases}
\end{equation}
where the gradient $\nabla f$ is evaluated at the point $y_{k}$, which is to be determined as a suitable convex combination of $x_{k}$ and $z_{k}$ such that $x_{k+1} - y_{k} \to 0$ as $k \to + \infty$. Notice that, since $\nabla f$ is $L-$Lipschitz continuous, this implies that $\nabla f \left( x_{k+1} \right) - \nabla f \left( y_{k} \right) \to 0$ as $k \to + \infty$.

The second equation in \eqref{ds:im-fd} is equivalent with
\begin{equation*}
x_{k+1} = \dfrac{\alpha - 1}{k + \alpha - 1} z_{k+1} + \dfrac{k}{k + \alpha - 1} x_{k}
\end{equation*}
and consequently suggests the following choice for $y_k$
\begin{equation}
\label{dis:y:com}
y_{k} = \dfrac{\alpha - 1}{k + \alpha - 1} z_{k} + \dfrac{k}{k + \alpha - 1} x_{k} .
\end{equation}
From the second equation in \eqref{ds:im-fd} we further obtain
\begin{align*}
y_{k} & = \dfrac{\alpha - 1}{k + \alpha - 1} z_{k} + \dfrac{k}{k + \alpha - 1} x_{k} = \dfrac{\alpha - 1}{k + \alpha - 1} \left( x_{k} + \dfrac{k-1}{\alpha - 1} \left( x_{k} - x_{k-1} \right) \right) + \dfrac{k}{k + \alpha - 1} x_{k} \nonumber \\
& = x_{k} + \dfrac{k - 1}{k + \alpha - 1} \left( x_{k} - x_{k-1} \right) . \label{dis:y:x}
\end{align*}
In addition,
\begin{equation*}
z_{k+1} - z_{k} = \dfrac{k + \alpha - 1}{\alpha - 1} \left( x_{k+1} - y_{k} \right) = \left( 1 + \dfrac{k}{\alpha - 1} \right) \left( x_{k+1} - y_{k} \right).
\end{equation*}
Consequently, \eqref{ds:im-fd} can be equivalently written as
\begin{equation}
\label{ite:C-D}
 \begin{dcases}
y_{k} 		& := x_{k} + \dfrac{k - 1}{k + \alpha - 1} \left( x_{k} - x_{k-1} \right) \\
x_{k+1} 	& := y_{k} - \sigma \nabla f \left( y_{k} \right).
\end{dcases}
\end{equation}
This is nothing else than the algorithm considered by Chambolle and Dossal in \cite{Chambolle-Dossal} (see also \cite{Attouch-Peypouquet}).

Furthermore, if we write for every $k \geq 1$
\begin{equation}
\label{dis:t-k:C-D}
t_{k} := 1 + \dfrac{k-1}{\alpha - 1} = \dfrac{k + \alpha - 2}{\alpha - 1} ,
\end{equation}
so that
\begin{equation*}
t_{k+1} - 1 = \dfrac{k}{\alpha - 1}
\qquad \textrm{ and } \qquad
\dfrac{t_{k}-1}{t_{k+1}} = \dfrac{k-1}{k + \alpha - 1} ,
\end{equation*}
then \eqref{ite:C-D} becomes
\begin{equation}
\label{ite:Nes}
(\forall k \geq 1) \quad  \begin{dcases}
y_{k} 		& := x_{k} + \dfrac{t_{k}-1}{t_{k+1}} \left( x_{k} - x_{k-1} \right) \\
x_{k+1} 	& := y_{k} - \sigma \nabla f \left( y_{k} \right).
\end{dcases}
\end{equation}

Modifications of the sequence $\left\lbrace t_{k} \right\rbrace _{k \geq 1}$ which preserve its asymptotic behaviour lead to various acceleration schemes from the literature.

For instance, the classical Nesterov's accelerated gradient method (\cite{Nesterov:83}) is precisely \eqref{ite:Nes}, where the sequence $\left\lbrace t_{k} \right\rbrace _{k \geq 1}$ satisfies the recurrence rule
\begin{equation}
\label{algo:FISTA-Nesterov}
t_{1} := 1
\qquad \textrm{ and } \qquad
t_{k+1} := \dfrac{1 + \sqrt{1 + 4t_{k}^{2}}}{2} \quad \forall k \geq 1.
\end{equation}

Another example is the algorithm proposed by Attouch and Cabot in \cite{Attouch-Cabot:18} that corresponds to \eqref{ite:Nes} with the choice
\begin{equation}
\label{algo:Attouch-Cabot}
t_{k} := \dfrac{k-1}{\alpha - 1}  \quad \forall k \geq 1.
\end{equation}
It can also be interpreted as a discretization of \eqref{ds:AVD:fo} with time step
\begin{equation*}
\sigma_{k} := \dfrac{\sigma k}{k - \alpha + 1}  \quad \forall k \geq 1,
\end{equation*}
and by setting $\tau_{k} := \sqrt{\sigma_{k}} \left( k - \alpha + 1 \right) = \sqrt{\sigma k \left( k - \alpha + 1 \right)} \approx \sqrt{\sigma} \left( k + 1 \right), x \left( \tau_{k} \right) \approx x_{k+1}$ and $z \left( \tau_{k} \right) \approx z_{k+1}$.

\subsection{The time discretization of (\ref{ds:PD-AVD})}

In order to provide a useful time discretization of the dynamical system \eqref{ds:PD-AVD} we follow the approach of the previous section and define for every $t \geq t_{0}$
\begin{equation}
\label{dis:z-nu}
z \left( t \right) := x \left( t \right) + \dfrac{t}{\alpha - 1} \dot{x} \left( t \right)
\qquad \textrm{ and } \qquad
\nu \left( t \right) := \lambda \left( t \right) + \dfrac{t}{\alpha - 1} \dot{\lambda} \left( t \right) .
\end{equation}
Further, we set
\begin{equation}
\label{ds:gamma}
\gamma := \dfrac{1}{\theta \left( \alpha - 1 \right)} \in \left[ \dfrac{2}{\alpha - 1} , 1 \right].
\end{equation}
The parameter $\gamma$ will play an essential role in our analysis. For every $t \geq t_0$ we define
\begin{subequations}
	\begin{align}
	z^{\gamma} \left( t \right)
	& := \gamma \left( x \left( t \right) + \theta t \dot{x} \left( t \right) \right)
	= \gamma x \left( t \right) + \dfrac{t}{\alpha - 1} \dot{x} \left( t \right)
	= z \left( t \right) + \left( \gamma - 1 \right) x \left( t \right) , \\
	\nu^{\gamma} \left( t \right)
	& := \gamma \left( \lambda \left( t \right) + \theta t \dot{\lambda} \left( t \right) \right)
	= \gamma \lambda \left( t \right) + \dfrac{t}{\alpha - 1} \dot{\lambda} \left( t \right)
	= \nu \left( t \right) + \left( \gamma - 1 \right) \lambda \left( t \right) .
	\end{align}
\end{subequations}
Using these notations, the system \eqref{ds:PD-AVD}  (see also its equivalent formulation \eqref{ds:Cauchy}) can be written as
\begin{equation}
\label{ds:Cauchy-gamma}
\begin{dcases}
\ddot{x} \left( t \right) + \dfrac{\alpha}{t} \dot{x} \left( t \right) + \nabla f \left( x \left( t \right) \right) + \dfrac{1}{\gamma} A^{*} \nu^{\gamma} \left( t \right) + \beta A^{*} \Bigl( Ax \left( t \right) - b \Bigr)	& = 0 \\
\ddot{\lambda} \left( t \right) + \dfrac{\alpha}{t} \dot{\lambda} \left( t \right) - \dfrac{1}{\gamma} \Bigl( Az^{\gamma} \left( t \right) - \gamma b \Bigr) 	& = 0 \\
\Bigl( x \left( t_{0} \right) , \lambda \left( t_{0} \right) \Bigr) 			= \Bigl( x_{0} , \lambda_{0} \Bigr) \textrm{ and }
\Bigl( \dot{x} \left( t_{0} \right) , \dot{\lambda} \left( t_{0} \right) \Bigr) = \Bigl( \dot{x}_{0} , \dot{\lambda}_{0} \Bigr)
\end{dcases} .
\end{equation}
Using that for every $t \geq t_0$
\begin{align*}
\dot{z} \left( t \right) = \dfrac{t}{\alpha - 1} \left( \ddot{x} \left( x \right) + \dfrac{\alpha}{t} \dot{x} \left( t \right) \right) \ \mbox{and} \
\dot{\nu} \left( t \right) = \dfrac{t}{\alpha - 1} \left( \ddot{\lambda} \left( x \right) + \dfrac{\alpha}{t} \dot{\lambda} \left( t \right) \right) ,
\end{align*}
the first two lines in \eqref{ds:Cauchy-gamma} can be equivalently written as
\begin{equation}
\label{ds:Cauchy-gamma:fo}
\begin{dcases}
& \dot{z} \left( t \right)		= - \dfrac{t}{\alpha - 1} \nabla f \left( x \left( t \right) \right) - \dfrac{t}{\alpha - 1} \dfrac{1}{\gamma} A^{*} \nu^{\gamma} \left( t \right) - \dfrac{t}{\alpha - 1} \beta A^{*} \Bigl( Ax \left( t \right) - b \Bigr) \\
& \dot{\nu} \left( t \right) 	= \dfrac{1}{\gamma} \dfrac{t}{\alpha - 1} \Bigl( Az^{\gamma} \left( t \right) - \gamma b \Bigr) \\
& z \left( t \right) 			= x \left( t \right) + \dfrac{t}{\alpha - 1} \dot{x} \left( t \right)\\
& z^{\gamma} \left( t \right) 	= \gamma x \left( t \right) + \dfrac{t}{\alpha - 1} \dot{x} \left( t \right) \\
& \nu \left( t \right) 			= \lambda \left( t \right) + \dfrac{t}{\alpha - 1} \dot{\lambda} \left( t \right)\\
& \nu^{\gamma} \left( t \right) 	= \gamma \lambda \left( t \right) + \dfrac{t}{\alpha - 1} \dot{\lambda} \left( t \right).
\end{dcases}
\end{equation}

Let $\sigma, \rho > 0$. For every $k \geq 1$ we take for $x$ and $\lambda$ two different time steps
\begin{equation*}
\sigma_{k} := \sigma \left( 1 + \dfrac{\alpha - 1}{k} \right) >0
\qquad \textrm{ and } \qquad
\rho_{k} := \rho \left( 1 + \dfrac{\alpha - 1}{k} \right) >0,
\end{equation*}
respectively, and set $\tau_{k} := \sqrt{\sigma_{k}} k \approx \sqrt{\sigma} \left( k + 1 \right), x \left( \tau_{k} \right) \approx x_{k+1},z \left( \tau_{k} \right) \approx z_{k+1}, z^{\gamma} \left( \tau_{k} \right) \approx z_{k+1}^{\gamma}$, and $\tau_{k}' := \sqrt{\rho_{k}} k \approx \sqrt{\rho} \left( k + 1 \right), \lambda \left( \tau_{k}' \right) \approx \lambda_{k+1}, \nu \left( \tau_{k}' \right) \approx \nu_{k+1}$ and $\nu^{\gamma} \left( \tau_{k}' \right) \approx \nu_{k+1}^{\gamma}$.

The implicit finite-difference scheme for \eqref{ds:Cauchy-gamma:fo} at time $t := \tau_{k}$ for $x$ and time $t := \tau_{k}'$ for $\lambda$ gives
\begin{equation}
\label{dis:Cauchy-gamma:fd}
\begin{dcases}
\dfrac{z_{k+1} - z_{k}}{\sqrt{\sigma_{k}}} 	& = - \dfrac{\sqrt{\sigma_{k}} k}{\alpha - 1} \nabla f \left( y_{k} \right) - \dfrac{\sqrt{\sigma_{k}} k}{\alpha - 1} A^{*} \dfrac{1}{\gamma} \tnu_{k+1} - \dfrac{\sqrt{\sigma_{k}} k}{\alpha - 1} \beta A^{*} \left( Ay_{k} - b \right) \\
\dfrac{\nu_{k+1} - \nu_{k}}{\sqrt{\rho_{k}}} 	& = \dfrac{1}{\gamma} \dfrac{\sqrt{\rho_{k}} k}{\alpha - 1} \left( Az_{k+1}^{\gamma} - \gamma b \right) \\
z_{k+1} 					& = x_{k+1} + \dfrac{\sqrt{\sigma_{k}} k}{\alpha - 1} \dfrac{x_{k+1} - x_{k}}{\sqrt{\sigma_{k}}} \\
z_{k+1}^{\gamma} 			& = \gamma x_{k+1} + \dfrac{\sqrt{\sigma_{k}} k}{\alpha - 1} \dfrac{x_{k+1} - x_{k}}{\sqrt{\sigma_{k}}} \\
\nu_{k+1} 					& = \lambda_{k+1} + \dfrac{\sqrt{\sigma_{k}} k}{\alpha - 1} \dfrac{\lambda_{k+1} - \lambda_{k}}{\sqrt{\sigma_{k}}} \\
\nu_{k+1}^{\gamma} 			& = \gamma \lambda_{k+1} + \dfrac{\sqrt{\sigma_{k}} k}{\alpha - 1} \dfrac{\lambda_{k+1} - \lambda_{k}}{\sqrt{\sigma_{k}}},
\end{dcases}
\end{equation}
where $y_{k}$ and $\tnu_{k+1}$ will be chosen appropriately to obtain an easily implementable iterative scheme.
Notice that $\tnu_{k+1}$ must be an approximation of $\nu_{k+1}^{\gamma}$.

Once again we take as in the previous section (see \eqref{dis:y:com})
\begin{equation*}
y_{k} = \dfrac{\alpha - 1}{k + \alpha - 1} z_{k} + \dfrac{k}{k + \alpha - 1} x_{k} = x_k + \dfrac{k - 1}{k + \alpha - 1}(x_k - x_{k-1}),
\end{equation*}
which, by using the third equation in \eqref{dis:Cauchy-gamma:fd}, gives
\begin{equation*}
z_{k+1} - z_{k} = \dfrac{k + \alpha - 1}{\alpha - 1} \left( x_{k+1} - y_{k} \right) = \left( 1 + \dfrac{k}{\alpha - 1} \right) \left( x_{k+1} - y_{k} \right).
\end{equation*}
Following \eqref{dis:y:com} we set also for the sequence of dual variables
\begin{equation*}
\mu_{k}  = \dfrac{\alpha - 1}{k + \alpha - 1} \nu_{k} + \dfrac{k}{k + \alpha - 1} \lambda_{k} = \lambda_{k} + \dfrac{k - 1}{k + \alpha - 1} \left( \lambda_{k} - \lambda_{k-1} \right),
\end{equation*}
which, by using the fifth equation in \eqref{dis:Cauchy-gamma:fd}, gives
\begin{equation}
\label{eu1}
\nu_{k+1} - \nu_{k} = \dfrac{k + \alpha - 1}{\alpha - 1} \left( \lambda_{k+1} - \mu_{k} \right) = t_{k+1}\left( \lambda_{k+1} - \mu_{k} \right).
\end{equation}
For these choices, and by taking into consideration the definition of $\left\lbrace t_{k} \right\rbrace _{k \geq 1}$ in \eqref{dis:t-k:C-D},  \eqref{dis:Cauchy-gamma:fd} becomes
\begin{equation}
\begin{cases}
y_{k} 			& = x_{k} + \dfrac{t_{k} - 1}{t_{k+1}} \left( x_{k} - x_{k-1} \right) \\
x_{k+1} 		& = y_{k} - \sigma \nabla f \left( y_{k} \right) - \dfrac{\sigma}{\gamma} A^{*} \tnu_{k+1} - \sigma \beta A^{*} \left( Ay_{k} - b \right) \\
\mu_{k} 		& = \lambda_{k} + \dfrac{t_{k} - 1}{t_{k+1}} \left( \lambda_{k} - \lambda_{k-1} \right) \\
z_{k+1}^{\gamma} 	& = \gamma x_{k+1} + \left( t_{k+1} - 1 \right) \left( x_{k+1} - x_{k} \right) \\
\lambda_{k+1} 	& = \mu_{k} + \dfrac{\rho}{\gamma} \left( Az_{k+1}^{\gamma} - \gamma b \right)\\
\nu_{k+1}^{\gamma} 	& = \gamma \lambda_{k+1} + \left( t_{k+1} - 1 \right) \left(\lambda_{k+1} - \lambda_{k} \right), 
\end{cases} \label{dis:tnu}
\end{equation}
where $\tnu_{k+1}$ is still to be chosen such that $\tnu_{k+1} - \nu_{k+1}^{\gamma} \to 0$ as $k \to + \infty$. We will not opt for $\tnu_{k+1} = \nu_{k+1}^{\gamma}$ in order to avoid an implicit iterative scheme, but choose instead (see also \eqref{eu1})
\begin{align*}
\tnu_{k+1}
:= & \ \nu_{k+1}^{\gamma} + \left( 1 - \gamma \right) \left( \lambda_{k+1} - \lambda_{k} \right) =  \gamma \lambda_{k+1} + \left( t_{k+1} - \gamma \right) \left(\lambda_{k+1} - \lambda_{k} \right) =  \gamma \lambda_{k} + t_{k+1} \left(\lambda_{k+1} - \lambda_{k} \right)\\
= & \ \gamma \lambda_{k} + \left( t_{k} - 1 \right) \left(\lambda_{k} - \lambda_{k-1} \right) + t_{k+1} \left(\lambda_{k+1} - \lambda_{k} - \frac{t_k-1}{t_{k+1}} \left(\lambda_{k} - \lambda_{k-1} \right)  \right)\\
 = & \ \nu_{k}^{\gamma} + t_{k+1}\left( \lambda_{k+1} - \mu_{k} \right) = \nu_{k}^{\gamma} + \nu_{k+1} - \nu_{k} = \nu_{k}^{\gamma} + \dfrac{\rho}{\gamma} t_{k+1} \left( Az_{k+1}^{\gamma} - \gamma b \right) \nonumber \\
 = & \ \nu_{k}^{\gamma} + \dfrac{\rho}{\gamma} t_{k+1} \left( \left( t_{k+1} - 1 + \gamma \right) Ax_{k+1} - \left( t_{k+1} - 1 \right) Ax_{k} - \gamma b \right) \nonumber \\
 = & \ \nu_{k}^{\gamma} + \dfrac{\rho}{\gamma} t_{k+1} \left( t_{k+1} - 1 + \gamma \right) \left( Ax_{k+1} - \dfrac{1}{t_{k+1} - 1 + \gamma} \left( \left( t_{k+1} - 1 \right) Ax_{k} + \gamma b \right) \right) ,
\end{align*}
Such a choice is reasonable as long as $\lambda_{k+1} - \lambda_{k} \to 0$ as $k \to + \infty$, which will then imply that $\tnu_{k+1} - \nu_{k+1}^{\gamma} \to 0$ as $k \to + \infty$. By setting
\begin{align*}
s_{k+1} 	:= \dfrac{\rho}{\gamma} t_{k+1} \left( t_{k+1} - 1 + \gamma \right) \ \mbox{and} \
\eta_{k} 	:= \dfrac{1}{t_{k+1} - 1 + \gamma} \left( \left( t_{k+1} - 1 \right) Ax_{k} + \gamma b \right),
\end{align*}
the second line in \eqref{dis:tnu} becomes
\begin{equation*}
x_{k+1} = y_{k} - \sigma \nabla f \left( y_{k} \right) - \dfrac{\sigma}{\gamma} A^{*} \left( \nu_{k}^{\gamma} + s_{k+1} \left( Ax_{k+1} - \eta_{k} \right) \right) - \sigma \beta A^{*} \left( Ay_{k} - b \right) 
\end{equation*}
or, equivalently,
\begin{align*}
x_{k+1} + \dfrac{\sigma}{\gamma} s_{k+1} A^{*} Ax_{k+1}
& = \left( \Id +  \dfrac{\sigma}{\gamma} s_{k+1} A^{*} A \right) x_{k+1} \nonumber \\
& = y_{k} - \sigma \nabla f \left( y_{k} \right) - \dfrac{\sigma}{\gamma} A^{*} \nu_{k}^{\gamma} + \dfrac{\sigma}{\gamma} s_{k+1} A^{*} \eta_{k} - \sigma \beta A^{*} \left( Ay_{k} - b \right).
\end{align*}
After rearranging the order in which the sequences are updated,  \eqref{dis:tnu} leads to the fast Augmented Lagrangian Method which we propose in this paper,  and also investigate from the point of view of its convergence properties.

\section{Fast Augmented Lagrangian Method}	

In this section we will give a precise formulation of the Augmented Lagrangian Method  for solving \eqref{intro:pb} and prove that it exhibits convergence rates of order $\bO \left( 1/k^{2} \right)$ for the primal-dual gap, the feasibility measure, and the objective function value.

\subsection{The algorithm}	

\begin{mdframed}
	\begin{algo}
		\label{algo:fast}
		Let $\beta \geq 0$ and $m , \gamma, \rho, \sigma > 0$ be such that
		\begin{equation}
		\label{assume:sigma}
		0 < m \leq \gamma \leq 1
		\qquad \textrm{ and } \qquad
		0 < \sigma \leq \dfrac{\gamma}{L + \gamma \beta \left\lVert A \right\rVert ^{2}} .
		\end{equation}
		Let $\left\lbrace t_{k} \right\rbrace _{k \geq 1}$ be a nondecreasing sequence such that
		\begin{equation}
		\label{assume:t-k+}
		t_{1} := 1 \qquad \textrm{ and } \qquad t_{k+1}^{2} - mt_{k+1} \leq t_{k}^{2} \quad \forall k \geq 1 .
		\end{equation}
		Given $x_{0} = x_{1} \in \sH$ and $\lambda_{0} = \lambda_{1} \in \sG$, for every $k \geq 1$ we set
		\begin{subequations}
			\label{algo:all}
			\begin{align}
			y_{k}				& := x_{k} + \dfrac{t_{k}-1}{t_{k+1}} \left( x_{k} - x_{k-1} \right) , \label{algo:y} \\
			\mu_{k}				& := \lambda_{k} + \dfrac{t_{k}-1}{t_{k+1}} \left( \lambda_{k} - \lambda_{k-1} \right) , \label{algo:mu} \\
			\eta_{k}			& := Ax_k + \dfrac{\gamma}{t_{k+1} - 1 + \gamma} (b-Ax_k), \label{algo:eta} \\			
			\nu_{k}^{\gamma}	& := \gamma \lambda_{k} + \left( t_{k} - 1 \right) \left( \lambda_{k} - \lambda_{k-1} \right) , \label{algo:nu-gamma} \\
			s_{k+1}				& := \dfrac{\rho}{\gamma} t_{k+1} \left( t_{k+1} - 1 + \gamma \right) , \label{algo:s} \\
			x_{k+1}				& := \arg\min_{x \in \sH} \left\lbrace \left\langle \nabla f \left( y_{k} \right) + \beta A^{*} \left( Ay_{k} - b \right) , x - y_{k} \right\rangle + \dfrac{1}{\gamma} \left\langle \nu_{k}^{\gamma} , Ax - b \right\rangle \right. \nonumber \\
			& \qquad\qquad\qquad\qquad\qquad\qquad \left. + \dfrac{1}{2 \gamma} s_{k+1} \left\lVert Ax - \eta_{k} \right\rVert ^{2} + \dfrac{1}{2 \sigma} \left\lVert x - y_{k} \right\rVert ^{2} \right\rbrace , \label{algo:x} \\
			z_{k+1}^{\gamma}	& := \gamma x_{k+1} + \left( t_{k+1} - 1 \right) \left( x_{k+1} - x_{k} \right) , \label{algo:z-gamma} \\
			\lambda_{k+1}		& := \mu_{k} + \dfrac{\rho}{\gamma} \left( Az_{k+1}^{\gamma} - \gamma b \right) . \label{algo:lambda}
			\end{align}
		\end{subequations}
	\end{algo}
\end{mdframed}

One can notice that Algorithm \ref{algo:fast} can be written in a concise way only in terms of the sequences of primal-dual iterates $\left\lbrace \left( x_{k} , \lambda_{k} \right) \right\rbrace _{k \geq 0}$, however, this elaborated formulation using auxiliary sequences is more convenient for its analysis.

Even though the choice $\gamma=1$ would give a simplified version of Algorithm \ref{algo:fast}, without affecting its fast convergence properties, we will see that in order to guarantee the convergence of $\left\lbrace \left( x_{k} , \lambda_{k} \right) \right\rbrace _{k \geq 0}$ to a primal-dual optimal solution it will be crucial to choose $\gamma \in (0,1)$. A similar phenomenon is known from the continuous and discrete schemes in the unconstrained case, where fast convergence rates have been obtained for $\alpha \geq 3$, while the convergence of the sequence of iterates/trajectory could be shown only for $\alpha >3$. In view of \eqref{ds:gamma}, in order to be allowed to choose $\gamma \in (0,1)$, one must have $\alpha >3$.

\begin{ex} {\bf (The case $A=0$ and $b=0$)}
If $A=0$ and $b=0$, then \eqref{intro:pb} becomes the unconstrained optimization problems \eqref{unconstrained:pb} and Algorithm \ref{algo:fast} reduces to the well-known accelerated gradient scheme which, given  $0 < \sigma \leq \dfrac{1}{L}$, $\{t_k\}_{k \geq 1}$ a nondecreasing sequence fulfilling \eqref{assume:t-k+} and $x_{0} = x_{1} \in \sH$, reads for every $k \geq 1$
	\begin{align*}
	y_{k} 	& := x_{k} + \dfrac{t_{k}-1}{t_{k+1}} \left( x_{k} - x_{k-1} \right) \nonumber \\
	x_{k+1}	& := y_{k} - \sigma \nabla f \left( y_{k} \right),
	\end{align*}
as the dual sequence $\left\lbrace \lambda_{k} \right\rbrace_{k \geq 0}$ can be neglected.
\end{ex}

\begin{rmk} \label{rmk:z}
By denoting for every $k \geq 1$
	\begin{subequations}
		\begin{align}
		z_{k}
		& := x_{k} + \left( t_{k} - 1 \right) \left( x_{k} - x_{k-1} \right) \label{eq:z-x} \\	
		& = x_{k} + t_{k+1} \left( y_{k} - x_{k} \right) \label{eq:z-x-y} \\
		& = t_{k} x_{k} - \left( t_{k} - 1 \right) x_{k-1}, \label{algo:z}
		\end{align}
	\end{subequations}
it yields	
	\begin{align}
	y_{k} = \left( 1 - \dfrac{1}{t_{k+1}} \right) x_{k} + \dfrac{1}{t_{k+1}} z_{k}= \left( 1 - \dfrac{1}{t_{k+1}} \right) x_{k} + \dfrac{1}{t_{k+1}} \left( x_{k} + \left( t_{k} - 1 \right) \left( x_{k} - x_{k-1} \right) \right). \label{eq:y}
	\end{align}
	On the other hand, \eqref{algo:z} with index $k+1$ reads
	\begin{equation}
	\label{eq:z-1}
	z_{k+1} = t_{k+1} x_{k+1} - \left( t_{k+1} - 1 \right) x_{k} ,
	\end{equation}
	which is equivalent to
	\begin{equation}
	\label{eq:x}
	x_{k+1} = \left( 1 - \dfrac{1}{t_{k+1}} \right) x_{k} + \dfrac{1}{t_{k+1}} z_{k+1} .
	\end{equation}
	Subtracting \eqref{eq:y} from \eqref{eq:x}, we obtain
	\begin{equation}
	\label{eq:x-y-z}
	x_{k+1} - y_{k} = \dfrac{1}{t_{k+1}} \left( z_{k+1} - z_{k} \right) .
	\end{equation}
Furthermore, by the definition of $z_{k}^{\gamma}$ and $z_{k}$ in \eqref{algo:z-gamma} and \eqref{eq:z-x}, it holds
	\begin{equation*}
	z_{k}^{\gamma} = z_{k} + \left( \gamma - 1 \right) x_{k},
	\end{equation*}
which leads to
	\begin{subequations}
		\begin{align}
		z_{k+1}^{\gamma} - z_{k}^{\gamma}
		& = z_{k+1} - z_{k} + \left( \gamma - 1 \right) \left( x_{k+1} - x_{k} \right) \label{eq:dz:z-x} \\
		& = t_{k+1} \left( x_{k+1} - y_{k} \right) + \left( \gamma - 1 \right) \left( x_{k+1} - x_{k} \right) \label{eq:dz:x-x} .
		\end{align}	
	\end{subequations}

By a similar argument, denoting for every $k \geq 1$
\begin{align}
		\nu_{k}
		& := \lambda_{k} + \left( t_{k} - 1 \right) \left(\lambda_{k} - \lambda_{k-1} \right) \label{eq:nu-lambda}, 
		\end{align}
we can derive that
	\begin{equation}
	\label{eq:nu-gamma-lambda}
\nu_{k+1}^{\gamma} - \nu_{k}^{\gamma} = t_{k+1} \left( \lambda_{k+1} - \mu_{k} \right) +  \left( \gamma- 1 \right) \left( \lambda_{k+1} - \lambda_{k} \right) \ \mbox{and} \ \lambda_{k+1} - \mu_{k} = \dfrac{1}{t_{k+1}} \left(\nu_{k+1} - \nu_{k} \right).
	\end{equation}
\end{rmk}

\begin{ex}  {\bf (The choice $\gamma=1$)} In case $\gamma=1$ we denote $z_k:=z_k^1$ and $\nu_k:=\nu_k^1$ for every $k \geq 1$, which is consistent with the notations in the remark above.  Given $0 < \sigma \leq \dfrac{1}{L + \beta\|A\|^2}$, $\{t_k\}_{k \geq 1}$ a nondecreasing sequence fulfilling \eqref{assume:t-k+}  $x_{0} = x_{1} \in \sH$ and $\lambda_{0} = \lambda_{1} \in \sG$, Algorithm \ref{algo:fast} simplifies for every $k \geq 1$ to
	\begin{subequations}
		\begin{align*}
		y_{k}				& := x_{k} + \dfrac{t_{k}-1}{t_{k+1}} \left( x_{k} - x_{k-1} \right) , \\
	         \mu_{k} 		& := \lambda_{k} + \dfrac{t_{k}-1}{t_{k+1}} \left( \lambda_{k} - \lambda_{k-1} \right) , \\
		\eta_{k}			& := \left( 1 - \dfrac{1}{t_{k+1}} \right) Ax_{k} + \dfrac{1}{t_{k+1}} b , \\	
\nu_{k} & := \lambda_{k} + \left( t_{k} - 1 \right) \left( \lambda_{k} - \lambda_{k-1} \right),\\   
		x_{k+1}				& := \arg\min_{x \in \sH} \bigg \{ \left\langle \nabla f \left( y_{k} \right) + \beta A^{*} \left( Ay_{k} - b \right) , x - y_{k} \right\rangle + \left\langle \nu_{k} , Ax - b \right\rangle  \nonumber \\
		& \qquad\qquad\qquad\qquad\qquad\qquad \left. + \dfrac{1}{2} \rho t_{k+1}^{2} \left\lVert Ax - \eta_{k} \right\rVert ^{2}  + \dfrac{1}{2 \sigma} \left\lVert x - y_{k} \right\rVert ^{2} \right\rbrace , \\
       \lambda_{k+1}		& := \mu_{k} + \rho  \left( Ax_{k+1} - b + \left( t_{k+1} - 1 \right) A \left( x_{k+1} - x_{k} \right) \right).
		\end{align*}
	\end{subequations}
The fact that this iterative scheme exhibits fast convergence rates for the primal-dual gap, the feasibility measure, and the objective function value will follow from the analysis we will do for Algorithm \ref{algo:fast}. However,  nothing can be said about the convergence of the primal-dual iterates. To this end we will have to assume that $\gamma \in (0,1)$, which will be a crucial assumption.
\end{ex}

\begin{rmk}
He, Hu and Fang have considered in \cite{He-Hu-Fang} for $\alpha > 3$, $\sigma, \sigma' > 0$ and 
	\begin{equation*}
	t_{k} := 1 + \dfrac{k-2}{\alpha - 1} \quad \forall k \geq 1 
	\end{equation*}
the following iterative scheme which reads for every $k \geq 1$
		\begin{align*}		
		y_{k}			& := x_{k} + \dfrac{t_{k}-1}{t_{k+1}} \left( x_{k} - x_{k-1} \right) , \\
		\mu_{k} 		& := \lambda_{k} + \dfrac{t_{k}-1}{t_{k+1}} \left( \lambda_{k} - \lambda_{k-1} \right) , \\
		\eta_{k}			& := \left( 1 - \dfrac{1}{t_{k+1}} \right) Ax_{k} + \dfrac{1}{t_{k+1}} b , \\	
\nu_{k} & := \lambda_{k} + \left( t_{k} - 1 \right) \left( \lambda_{k} - \lambda_{k-1} \right),\\   
		x_{k+1}			& := \arg\min\limits_{x \in \sH} \left\lbrace \left\langle \nabla f \left( y_{k} \right) , x - y_{k} \right\rangle + \left\langle \nu_{k} , Ax - b \right\rangle + \dfrac{\sigma \left( t_{k+2} - 1 \right) t_{k+1}}{2} \left\lVert Ax - \eta_{k} \right\rVert ^{2} \right. \nonumber \\
		& \qquad\qquad\qquad\qquad\qquad\qquad\qquad \qquad\qquad\qquad \qquad \left.  + \dfrac{\sigma ' t_{k+1}}{\sigma \left( t_{k+2} - 1 \right)} \left\lVert x - y_{k} \right\rVert ^{2} \right\rbrace ,\\
		\lambda_{k+1}	& := \mu_{k} + \dfrac{\sigma \left( t_{k+2} - 1 \right)}{t_{k+1}} \left( Ax_{k+1} - b + \left( t_{k+1} - 1 \right) A \left( x_{k+1} - x_{k} \right) \right). 
		\end{align*}
This algorithm differs from Algorithm \ref{algo:fast} for the choice $\gamma =1$ (as formulated in the above example) in the way the primal-dual iterates $\left\lbrace \left( x_{k} , \lambda_{k} \right) \right\rbrace _{k \geq 0}$ are defined. The formulation of the first allows a more direct derivation of the fast convergence rates for the feasibility measure and the objective function value. The convergence of $\left\lbrace \left( x_{k} , \lambda_{k} \right) \right\rbrace _{k \geq 0}$ has been not addressed in \cite{He-Hu-Fang}, and it is by far not clear whether this sequence converges.
\end{rmk}

The following lemma collects some properties of the sequence $\{t_k\}_{k \geq 1}$ fulfilling \eqref{assume:t-k+}.

\begin{lem}
	\label{lem:lim}
	Let $0 < m \leq 1$ and $\left\lbrace t_{k} \right\rbrace _{k \geq 1}$ a nondecreasing sequence fulfilling
\begin{equation*}
		t_{1} := 1 \qquad \textrm{ and } \qquad t_{k+1}^{2} - mt_{k+1} \leq t_{k}^{2} \quad \forall k \geq 1 .
		\end{equation*}
Then for every $k \geq 1$ it holds
	\begin{equation}
	\label{bound:t-k}
	t_{k+1} - t_{k} \leq \varphi_{m} := \dfrac{m - 2 + \sqrt{m^{2} + 4}}{2} \leq \dfrac{\sqrt{5} - 1}{2} < 1,
	\end{equation}
	\begin{equation}
	\label{bound:k}
	t_{k+1} \leq \left( 1 + \varphi_{m} \right) t_{k}
	\qquad \textrm{ and } \qquad
	t_{k+1} \leq 1 + k \varphi_{m} \leq \left( 1 + \varphi_{m} \right) k .
	\end{equation}
\end{lem}
\begin{proof}	
Let $k \geq 1$. From the assumption we have that
	\begin{equation*}
	1 \leq t_{k+1} \leq \dfrac{m + \sqrt{m^{2} + 4t_{k}^{2}}}{2},
	\end{equation*}
which further gives
	\begin{equation*}
	t_{k+1} - t_{k} 	
	\leq \dfrac{m + \sqrt{m^{2} + 4t_{k}^{2}}}{2} - t_{k} .
	\end{equation*}
We define the function $\psi : [ 1 , + \infty) \to\sR$ as $s \mapsto \dfrac{m + \sqrt{m^{2} + 4s^{2}}}{2} - s$. Since
	\begin{equation*}
	\psi ' \left( s \right) = \dfrac{2s}{\sqrt{m^{2} + 4s^{2}}} - 1 < 0,
	\end{equation*}
$\psi$ is nonincreasing, consequently
	\begin{equation*}
	t_{k+1} - t_{k} 	
	\leq \psi \left( 1 \right) = \dfrac{m + \sqrt{m^{2} + 4}}{2} - 1 = \varphi_{m} \leq \dfrac{\sqrt{5} - 1}{2}.
	\end{equation*}
The statements in \eqref{bound:k} follow from the fact that $t_{k} \geq 1$ for every $k \geq 1$ and $\varphi_m \geq 0$ and by telescoping arguments, respectively.
\end{proof}

\subsection{Some important estimates and an energy function}

In this section we will provide some important estimates which will be useful when proving that the sequence of values of a discrete energy function, which we will associate with Algorithm \ref{algo:fast}, takes at a saddle point is nonincreasing.

\begin{lem}
	\label{lem:smooth}
	Let $\left\lbrace \left( x_{k} , \lambda_{k} \right) \right\rbrace _{k \geq 0}$ be the sequence generated by Algorithm \ref{algo:fast}. Then for every $x \in \sH$ and every $k \geq 1$ the following inequality holds
	\begin{align}
	f \left( x_{k+1} \right)
	 \leq & \ f \left( x \right) - \dfrac{1}{\gamma} \left\langle \nu_{k+1}^{\gamma} , Ax_{k+1} - Ax \right\rangle + \dfrac{1}{\gamma} \left( 1 - \gamma \right) \left\langle \lambda_{k} - \lambda_{k+1} , Ax_{k+1} - Ax \right\rangle \nonumber \\
	& \ + \dfrac{1}{\sigma} \left\langle y_{k} - x_{k+1} , x_{k+1} - x \right\rangle - \beta \left\langle Ay_{k} - b , Ax_{k+1} - Ax \right\rangle \nonumber \\
	& \ + \dfrac{L}{2} \left\lVert x_{k+1} - y_{k} \right\rVert ^{2} - \dfrac{1}{2L} \left\lVert \nabla f \left( y_{k} \right) - \nabla f \left( x \right) \right\rVert ^{2} . \label{smooth:inq}
	\end{align}
\end{lem}
\begin{proof}
Let $x \in \sH$ and $k \geq 1$ be fixed. According to \eqref{algo:x} we have that
	\begin{align}
\nabla f \left( y_{k} \right) + \dfrac{1}{\gamma} A^{*} \nu_{k}^{\gamma}  + \dfrac{1}{\gamma} s_{k+1} A^{*}(Ax_{k+1}-\eta_k) + \dfrac{1}{\sigma}(x_{k+1} - y_k) + \beta A^{*} \left( Ay_{k} - b \right) = 0. \label{Fej:ite-x}
	\end{align}
On the other hand, from \eqref{algo:eta}, \eqref{algo:s} and \eqref{algo:lambda} we have
	\begin{align}
	\dfrac{1}{\gamma} s_{k+1} \left( Ax_{k+1} - \eta_{k} \right)
	& = \dfrac{\rho}{\gamma^{2}} t_{k+1} \left( \left( t_{k+1} - 1 + \gamma \right) Ax_{k+1} - \left( t_{k+1} - 1 \right) Ax_{k} - \gamma b \right) \nonumber \\
	& = \dfrac{\rho}{\gamma^{2}} t_{k+1} \left( Az_{k+1}^{\gamma} - \gamma b \right)
	= \dfrac{1}{\gamma} t_{k+1} \left( \lambda_{k+1} - \mu_{k} \right) \nonumber \\
	& = \dfrac{1}{\gamma} \left( \nu_{k+1}^{\gamma} - \nu_{k}^{\gamma} + \left( 1 - \gamma \right) \left( \lambda_{k+1} - \lambda_{k} \right) \right) , \label{Fej:nu-gamma}
	\end{align}
	where the last equation follows from \eqref{eq:nu-gamma-lambda}.
	Hence, replacing \eqref{Fej:nu-gamma} in \eqref{Fej:ite-x} we have
	\begin{equation}
	\label{Fej:opt-x}
	\nabla f \left( y_{k} \right) = - \dfrac{1}{\gamma} A^{*} \nu_{k+1}^{\gamma} + \dfrac{1}{\gamma} \left( 1 - \gamma \right) A^{*} \left( \lambda_{k} - \lambda_{k+1} \right) + \dfrac{1}{\sigma} \left( y_{k} - x_{k+1} \right) - \beta A^{*} \left( Ay_{k} - b \right) .
	\end{equation}
The Descent Lemma inequality \eqref{pre:f-bound} provides
	\begin{equation*}
	f \left( x_{k+1} \right) \leq f \left( y_{k} \right) + \left\langle \nabla f \left( y_{k} \right) , x_{k+1} - y_{k} \right\rangle + \dfrac{L}{2} \left\lVert x_{k+1} - y_{k} \right\rVert ^{2}
	\end{equation*}
and
	\begin{equation*}
	f \left( y_{k} \right) \leq f \left( x \right) + \left\langle \nabla f \left( y_{k} \right) , y_{k} - x \right\rangle - \dfrac{1}{2L} \left\lVert \nabla f \left( y_{k} \right) - \nabla f \left( x \right) \right\rVert ^{2} .
	\end{equation*}
By summing up these relations it yields
	\begin{align*}
	f \left( x_{k+1} \right)
	\leq & \ f \left( x \right) + \left\langle \nabla f \left( y_{k} \right) , x_{k+1} - x \right\rangle + \dfrac{L}{2} \left\lVert x_{k+1} - y_{k} \right\rVert ^{2} - \dfrac{1}{2L} \left\lVert \nabla f \left( y_{k} \right) - \nabla f \left( x \right) \right\rVert ^{2} \nonumber \\
	= & \ f \left( x \right) - \dfrac{1}{\gamma} \left\langle \nu_{k+1}^{\gamma} , Ax_{k+1} - Ax \right\rangle + \dfrac{1}{\gamma} \left( 1 - \gamma \right) \left\langle \lambda_{k} - \lambda_{k+1} , Ax_{k+1} - Ax \right\rangle \nonumber \\
	& + \dfrac{1}{\sigma} \left\langle y_{k} - x_{k+1} , x_{k+1} - x \right\rangle - \beta \left\langle Ay_{k} - b , Ax_{k+1} - Ax \right\rangle \nonumber \\
	& + \dfrac{L}{2} \left\lVert x_{k+1} - y_{k} \right\rVert ^{2} - \dfrac{1}{2L} \left\lVert \nabla f \left( y_{k} \right) - \nabla f \left( x \right) \right\rVert ^{2} ,
	\end{align*}
	which is nothing else than \eqref{smooth:inq}.
\end{proof}
In the following we denote
\begin{equation}\label{define:Q}
\Q := \dfrac{1}{\sigma} \Id - \beta A^{*} A .
\end{equation}
Assumption \eqref{assume:sigma} guarantees that $\gamma \Q - L \Id \in \sS_{+} \left( \sH \right)$, as
\begin{equation}\label{ineq:Q}
\gamma \Q - L \Id = \left( \dfrac{\gamma}{\sigma} - L \right) \Id - \gamma \beta A^{*} A \succcurlyeq \left( \dfrac{\gamma}{\sigma} - L - \gamma \beta \left\lVert A \right\rVert ^{2} \right) \Id .
\end{equation}

\begin{lem}
	\label{lem:inq}
	Let $\left\lbrace \left( x_{k} , \lambda_{k} \right) \right\rbrace _{k \geq 0}$ be the sequence generated by Algorithm \ref{algo:fast}. Then for every $\left( x , \lambda \right) \in \Fea \times \sG$ and every $k \geq 1$ the following two inequalities hold
	\begin{align}
	& \Lb \left( x_{k+1} , \lambda \right) - \Lb \left( x , \lambda_{k+1} \right) \nonumber \\
	\leq \ & \dfrac{1}{\gamma} \left( 1 - \gamma \right) \left\langle \lambda_{k} - \lambda_{k+1} , Ax_{k+1} - b \right\rangle
	+ \left\langle y_{k} - x_{k+1} , x_{k+1} - x \right\rangle _{\Q}
	+ \dfrac{1}{\rho} \left\langle \mu_{k} - \lambda_{k+1} , \lambda_{k+1} - \lambda \right\rangle \nonumber \\
	& - \dfrac{\beta}{2} \left\lVert Ax_{k+1} - b \right\rVert ^{2}
	+ \dfrac{L}{2} \left\lVert x_{k+1} - y_{k} \right\rVert ^{2}
	- \dfrac{1}{2L} \left\lVert \nabla f \left( y_{k} \right) - \nabla f \left( x \right) \right\rVert ^{2} \nonumber \\
	& - \dfrac{1}{\gamma} \left( t_{k+1} - 1 \right) \left\langle \lambda_{k+1} - \lambda_{k} , Ax_{k+1} - b \right\rangle
	+ \dfrac{1}{\gamma} \left( t_{k+1} - 1 \right) \left\langle \lambda_{k+1} - \lambda , Ax_{k+1} - Ax_{k} \right\rangle , \label{inq-1}
	\end{align}
	and
	\begin{align}
	& \Lb \left( x_{k+1} , \lambda \right) - \Lb \left( x , \lambda_{k+1} \right) \nonumber \\
	\leq \ 	& \Lb \left( x_{k} , \lambda \right) - \Lb \left( x , \lambda_{k} \right)
	+ \dfrac{1}{\gamma} \left( 1 - \gamma \right) \left\langle \lambda_{k} - \lambda_{k+1} , Ax_{k+1} - Ax_{k} \right\rangle
	+ \left\langle y_{k} - x_{k+1} , x_{k+1} - x_{k} \right\rangle _{\Q} \nonumber \\
	& + \dfrac{1}{\rho} \left\langle \mu_{k} - \lambda_{k+1} , \lambda_{k+1} - \lambda_{k} \right\rangle
	- \dfrac{\beta}{2} \left\lVert Ax_{k+1} - Ax_{k} \right\rVert ^{2}
	+ \dfrac{L}{2} \left\lVert x_{k+1} - y_{k} \right\rVert ^{2}
	\nonumber \\
	& - \dfrac{1}{2L} \left\lVert \nabla f \left( y_{k} \right) - \nabla f \left( x_{k} \right) \right\rVert ^{2}
	+ \left\langle \lambda - \lambda_{k+1} , Ax_{k+1} - Ax_{k} \right\rangle
	+ \left\langle \lambda_{k+1} - \lambda_{k} , Ax_{k+1} - b \right\rangle . \label{inq-2}
	\end{align}
\end{lem}
\begin{proof}
Let $x \in \Fea$, which means that $Ax = b$, $\lambda \in \sG$, and  $k \geq 1$ be fixed. We deduce from Lemma \ref{lem:smooth} that
	\begin{align}
	& f \left( x_{k+1} \right) + \left\langle \lambda , Ax_{k+1} - b \right\rangle \nonumber \\
	\leq \ 	& f \left( x \right) + \left\langle \lambda_{k+1} , Ax - b \right\rangle
	+ \dfrac{\beta}{2} \left\lVert Ax - b \right\rVert ^{2}
	+ \left\langle \lambda - \dfrac{1}{\gamma} \nu_{k+1}^{\gamma} , Ax_{k+1} - b \right\rangle \nonumber \\
	& + \dfrac{1}{\gamma} \left( 1 - \gamma \right) \left\langle \lambda_{k} - \lambda_{k+1} , Ax_{k+1} - b \right\rangle
	+ \dfrac{1}{\sigma} \left\langle y_{k} - x_{k+1} , x_{k+1} - x \right\rangle \nonumber \\
	& - \beta \left\langle Ay_{k} - b , Ax_{k+1} - Ax \right\rangle
	+ \dfrac{L}{2} \left\lVert x_{k+1} - y_{k} \right\rVert ^{2}
	- \dfrac{1}{2L} \left\lVert \nabla f \left( y_{k} \right) - \nabla f \left( x \right) \right\rVert ^{2} \nonumber \\
	= \ 	& f \left( x \right) + \left\langle \lambda_{k+1} , Ax - b \right\rangle
	+ \dfrac{\beta}{2} \left\lVert Ax - b \right\rVert ^{2}
	+ \left\langle \lambda - \dfrac{1}{\gamma} \nu_{k+1}^{\gamma} , Ax_{k+1} - b \right\rangle \nonumber \\
	& + \dfrac{1}{\gamma} \left( 1 - \gamma \right) \left\langle \lambda_{k} - \lambda_{k+1} , Ax_{k+1} - b \right\rangle
	+ \left\langle y_{k} - x_{k+1} , x_{k+1} - x \right\rangle _{\Q} \nonumber \\
	& - \beta \left\lVert Ax_{k+1} - b \right\rVert ^{2}
	+ \dfrac{L}{2} \left\lVert x_{k+1} - y_{k} \right\rVert ^{2}
	- \dfrac{1}{2L} \left\lVert \nabla f \left( y_{k} \right) - \nabla f \left( x \right) \right\rVert ^{2}, \label{inq-1:x}
	\end{align}	
where, by using the definition of $\Q$, the last identity follows from
\begin{align*}
& \dfrac{1}{\sigma} \left\langle y_{k} - x_{k+1} , x_{k+1} - x \right\rangle -  \beta \left\langle Ay_{k} - b , Ax_{k+1} - Ax \right\rangle \\
 = \ & \dfrac{1}{\sigma} \left\langle y_{k} - x_{k+1} , x_{k+1} - x \right\rangle -  \beta \left\langle Ay_{k} - Ax_{k+1} , Ax_{k+1} - Ax \right\rangle - \beta \left\lVert Ax_{k+1} - b \right\rVert ^{2} \\
 = \ & \dfrac{1}{\sigma} \left\langle y_{k} - x_{k+1} , x_{k+1} - x \right\rangle -  \beta \left\langle y_{k} - x_{k+1} , A^*A(x_{k+1} - x) \right\rangle - \beta \left\lVert Ax_{k+1} - b \right\rVert ^{2}\\
= \ &  \left\langle y_{k} - x_{k+1} , x_{k+1} - x \right\rangle _{\Q} -  \beta \left\lVert Ax_{k+1} - b \right\rVert ^{2}.
\end{align*}

Recall that from  \eqref{algo:lambda} we have
	\begin{equation*}
	0 = \mu_{k} - \lambda_{k+1} + \dfrac{\rho}{\gamma} \left( Az_{k+1}^{\gamma} - \gamma b \right) ,
	\end{equation*}
	which yields further
	\begin{equation}
	\label{inq-1:lambda}
	0 = \dfrac{1}{\rho} \left\langle \mu_{k} - \lambda_{k+1} , \lambda_{k+1} - \lambda \right\rangle
	+ \dfrac{1}{\gamma} \left\langle \lambda_{k+1} - \lambda , Az_{k+1}^{\gamma} - \gamma b \right\rangle .
	\end{equation}	
Moreover,  from \eqref{algo:nu-gamma} and \eqref{algo:z-gamma} we have
	\begin{align*}
	& \left\langle \lambda - \dfrac{1}{\gamma} \nu_{k+1}^{\gamma} , Ax_{k+1} - b \right\rangle
	+ \dfrac{1}{\gamma} \left\langle \lambda_{k+1} - \lambda , Az_{k+1}^{\gamma} - \gamma b \right\rangle \nonumber \\
= \ 	&  \left\langle \lambda - \lambda_{k+1} - \frac{1}{\gamma}(t_{k+1}-1)(\lambda_{k+1} - \lambda_k) , Ax_{k+1} - b \right\rangle \nonumber \\
& +  \left\langle \lambda_{k+1} - \lambda, Ax_{k+1} + \frac{1}{\gamma}  \left( t_{k+1} - 1 \right)A(x_{k+1} - x_k)- b \right\rangle \nonumber \\
	= \ 	& \left\langle \lambda - \lambda_{k+1} , Ax_{k+1} - b \right\rangle
	- \dfrac{1}{\gamma} \left( t_{k+1} - 1 \right) \left\langle \lambda_{k+1} - \lambda_{k} , Ax_{k+1} - b \right\rangle \nonumber \\
	& + \left\langle \lambda_{k+1} - \lambda , Ax_{k+1} - b \right\rangle
	+ \dfrac{1}{\gamma} \left( t_{k+1} - 1 \right) \left\langle \lambda_{k+1} - \lambda , Ax_{k+1} - Ax_{k} \right\rangle \nonumber \\
	= \		& - \dfrac{1}{\gamma} \left( t_{k+1} - 1 \right) \left\langle \lambda_{k+1} - \lambda_{k} , Ax_{k+1} - b \right\rangle
	+ \dfrac{1}{\gamma} \left( t_{k+1} - 1 \right) \left\langle \lambda_{k+1} - \lambda , Ax_{k+1} - Ax_{k} \right\rangle,
	\end{align*}
therefore, by summing up \eqref{inq-1:x} and \eqref{inq-1:lambda} and after rearranging the terms, the estimate \eqref{inq-1} follows.
	
Next we will prove the second estimate. By take $x := x_{k}$ in inequality \eqref{smooth:inq} we get
	\begin{align}
	& \ f \left( x_{k+1} \right) + \left\langle \lambda , Ax_{k+1} - b \right\rangle \nonumber \\
	\leq & \ f \left( x_{k} \right) + \left\langle \lambda , Ax_{k} - b \right\rangle
	+ \left\langle \lambda - \dfrac{1}{\gamma} \nu_{k+1}^{\gamma} , Ax_{k+1} - Ax_{k} \right\rangle
	+ \dfrac{1}{\gamma} \left( 1 - \gamma \right) \left\langle \lambda_{k} - \lambda_{k+1} , Ax_{k+1} - Ax_{k} \right\rangle \nonumber \\
	& + \dfrac{1}{\sigma} \left\langle y_{k} - x_{k+1} , x_{k+1} - x_{k} \right\rangle
	- \beta \left\langle Ay_{k} - b , Ax_{k+1} - Ax_{k} \right\rangle \nonumber \\
	& + \dfrac{L}{2} \left\lVert x_{k+1} - y_{k} \right\rVert ^{2}
	- \dfrac{1}{2L} \left\lVert \nabla f \left( y_{k} \right) - \nabla f \left( x_{k} \right) \right\rVert ^{2} \nonumber \\
	= & \ f \left( x_{k} \right) + \left\langle \lambda , Ax_{k} - b \right\rangle
	+ \left\langle \lambda - \dfrac{1}{\gamma}  \nu_{k+1}^{\gamma} , Ax_{k+1} - Ax_{k} \right\rangle
	+ \dfrac{1}{\gamma} \left( 1 - \gamma \right) \left\langle \lambda_{k} - \lambda_{k+1} , Ax_{k+1} - Ax_{k} \right\rangle \nonumber \\
	& + \left\langle y_{k} - x_{k+1} , x_{k+1} - x_{k} \right\rangle _{\Q}
	- \beta \left\langle Ax_{k+1} - b , Ax_{k+1} - Ax_{k} \right\rangle \nonumber \\
	& + \dfrac{L}{2} \left\lVert x_{k+1} - y_{k} \right\rVert ^{2}
	- \dfrac{1}{2L} \left\lVert \nabla f \left( y_{k} \right) - \nabla f \left( x_{k} \right) \right\rVert ^{2}, \label{inq-2:x-pre}
	\end{align}	
where, by using again the definition of $\Q$, the last identity follows from
\begin{align*}
& \dfrac{1}{\sigma} \left\langle y_{k} - x_{k+1} , x_{k+1} - x_k \right\rangle -  \beta \left\langle Ay_{k} - b , Ax_{k+1} - Ax_k \right\rangle \\
 = \ & \dfrac{1}{\sigma} \left\langle y_{k} - x_{k+1} , x_{k+1} - x_k \right\rangle -  \beta \left\langle Ay_{k} - Ax_{k+1} , Ax_{k+1} - Ax_k \right\rangle - \beta \left\langle Ax_{k+1} - b , Ax_{k+1} - Ax_k \right\rangle \\
 = \ & \dfrac{1}{\sigma} \left\langle y_{k} - x_{k+1} , x_{k+1} - x_k \right\rangle -  \beta \left\langle y_{k} - x_{k+1} , A^*A(x_{k+1} - x_k) \right\rangle - \beta \left\langle Ax_{k+1} - b , Ax_{k+1} - Ax_k \right\rangle\\
= \ &  \left\langle y_{k} - x_{k+1} , x_{k+1} - x_k \right\rangle _{\Q} -  \beta \left\langle Ax_{k+1} - b , Ax_{k+1} - Ax_k \right\rangle.
\end{align*}

	The identity \eqref{pre:id-eq} gives us
	\begin{equation*}
	- \beta \left\langle Ax_{k+1} - b , Ax_{k+1} - Ax_{k} \right\rangle
	= - \dfrac{\beta}{2} \left\lVert Ax_{k+1} - b \right\rVert ^{2}
	- \dfrac{\beta}{2} \left\lVert Ax_{k+1} - Ax_{k} \right\rVert ^{2}
	+ \dfrac{\beta}{2} \left\lVert Ax_{k} - b \right\rVert ^{2} ,
	\end{equation*}
	hence, by recalling relation \eqref{intro:Fea:eq}, \eqref{inq-2:x-pre} can be equivalently written as
	\begin{align}
	& \ \Lb \left( x_{k+1} , \lambda \right) - \Lb \left( x , \lambda_{k+1} \right) \nonumber \\
	\leq & \ \Lb \left( x_{k} , \lambda \right) - \Lb \left( x , \lambda_{k} \right)
	+ \left\langle \lambda - \dfrac{1}{\gamma} \nu_{k+1}^{\gamma} , Ax_{k+1} - Ax_{k} \right\rangle
	+ \dfrac{1}{\gamma} \left( \gamma - 1 \right) \left\langle \lambda_{k+1} - \lambda_{k} , Ax_{k+1} - Ax_{k} \right\rangle \nonumber \\
	& + \left\langle y_{k} - x_{k+1} , x_{k+1} - x_{k} \right\rangle _{\Q}
	- \dfrac{\beta}{2} \left\lVert Ax_{k+1} - Ax_{k} \right\rVert ^{2}
	+ \dfrac{L}{2} \left\lVert x_{k+1} - y_{k} \right\rVert ^{2} \nonumber \\
       & - \dfrac{1}{2L} \left\lVert \nabla f \left( y_{k} \right) - \nabla f \left( x_{k} \right) \right\rVert ^{2} . \label{inq-2:x}
	\end{align}
	In addition, by taking $\lambda := \lambda_{k}$ in \eqref{inq-1:lambda} gives
	\begin{equation}
	\label{inq-2:lambda}
	0 = \dfrac{1}{\rho} \left\langle \mu_{k} - \lambda_{k+1} , \lambda_{k+1} - \lambda_{k} \right\rangle
	+ \dfrac{1}{\gamma} \left\langle \lambda_{k+1} - \lambda_{k} , Az_{k+1}^{\gamma} - \gamma b \right\rangle .
	\end{equation}
Moreover, we have from \eqref{algo:nu-gamma} and \eqref{algo:z-gamma}
	\begin{align*}
	& \left\langle \lambda - \dfrac{1}{\gamma} \nu_{k+1}^{\gamma} , Ax_{k+1} - Ax_{k} \right\rangle
	+ \dfrac{1}{\gamma} \left\langle \lambda_{k+1} - \lambda_{k} , Az_{k+1}^{\gamma} - \gamma b \right\rangle \nonumber \\
	= \ 	& \left\langle \lambda - \lambda_{k+1} , Ax_{k+1} - Ax_{k} \right\rangle
	- \dfrac{1}{\gamma} \left( t_{k+1} - 1 \right) \left\langle \lambda_{k+1} - \lambda_{k} , Ax_{k+1} - Ax_{k} \right\rangle \nonumber \\
	& + \left\langle \lambda_{k+1} - \lambda_{k} , Ax_{k+1} - b \right\rangle
	+ \dfrac{1}{\gamma} \left( t_{k+1} - 1 \right) \left\langle \lambda_{k+1} - \lambda_{k} , Ax_{k+1} - Ax_{k} \right\rangle \nonumber \\
	= \ 	& \left\langle \lambda - \lambda_{k+1} , Ax_{k+1} - Ax_{k} \right\rangle
	+ \left\langle \lambda_{k+1} - \lambda_{k} , Ax_{k+1} - b \right\rangle,
	\end{align*}
therefore, by summing up \eqref{inq-2:x} and \eqref{inq-2:lambda} and after rearranging terms, the estimate \eqref{inq-2} follows.
\end{proof}

For $(x,\lambda) \in \Fea \times \sG$ and $k \geq 1$ we introduce the following energy function associated with Algorithm \ref{algo:fast}
\begin{align*}
\E_{k} \left( x , \lambda \right)
:= & \ t_{k} \left( t_{k} - 1 + \gamma \right) \left( \Lb \left( x_{k} , \lambda \right) - \Lb \left( x , \lambda_{k} \right) \right)
+ \dfrac{1}{2} \left\lVert z_{k}^{\gamma} - \gamma x \right\rVert _{\Q}^{2}
+ \dfrac{1}{2 \rho} \left\lVert \nu_{k}^{\gamma} - \gamma \lambda \right\rVert ^{2} \nonumber \\
&  + \dfrac{1}{2} \gamma \left( 1 - \gamma \right) \left\lVert x_{k} - x \right\rVert _{\Q}^{2}
+ \dfrac{1}{2 \rho}\gamma \left( 1 - \gamma \right) \left\lVert \lambda_{k} - \lambda \right\rVert ^{2}
+ \dfrac{1 - \gamma}{2 \rho} \left( t_{k} - 1 \right) \left\lVert \lambda_{k} - \lambda_{k-1} \right\rVert ^{2} .
\end{align*}
According to \eqref{intro:saddle}, for every  $\left( x_{*} , \lambda_{*} \right) \in \sol$ and every $k \geq 1$ it holds
\begin{equation*}
\E_{k} \left( x_{*} , \lambda_{*} \right) \geq 0 .
\end{equation*}

The following estimate for the energy function will play a fundamental role in our analysis.
\begin{prop}
	\label{prop:dec}
	Let $\left\lbrace \left( x_{k} , \lambda_{k} \right) \right\rbrace _{k \geq 0}$ be the sequence generated by Algorithm \ref{algo:fast}. Then for every $\left( x , \lambda \right) \in \Fea \times \sG$ and every $k \geq 1$ it holds
	\begin{align}
	\E_{k+1} \left( x , \lambda \right)
	 \leq & \ \E_{k} \left( x , \lambda \right)
	+ \left( t_{k+1}^{2} - t_{k+1} - t_{k}^{2} + \left( 1 - \gamma \right) t_{k} \right) \begin{pmatrix}
	\Lb \left( x_{k} , \lambda \right) - \Lb \left( x , \lambda_{k} \right)
	\end{pmatrix} \nonumber \\
	& - \dfrac{\beta \gamma}{2} t_{k+1} \left\lVert Ax_{k+1} - b \right\rVert ^{2}
	- \dfrac{\beta}{2} t_{k+1} \left( t_{k+1} - 1 \right) \left\lVert Ax_{k+1} - Ax_{k} \right\rVert ^{2} \nonumber \\
	&  - \dfrac{\gamma}{2L} t_{k+1} \left\lVert \nabla f \left( y_{k} \right) - \nabla f \left( x \right) \right\rVert ^{2}
	- \dfrac{1}{2L} t_{k+1} \left( t_{k+1} - 1 \right) \left\lVert \nabla f \left( y_{k} \right) - \nabla f \left( x_{k} \right) \right\rVert ^{2} \nonumber \\
	&  - \left(1- \gamma \right) \left( t_{k+1} - 1 \right) \left\lVert x_{k+1} - x_{k} \right\rVert _{\Q}^{2}
	- \dfrac{1- \gamma}{2 \rho} \left(2 t_{k+1} - 1 \right) \left\lVert \lambda_{k+1} - \lambda_{k} \right\rVert ^{2} \nonumber \\
	& - \dfrac{1}{2} t_{k+1}^{2} \left\lVert x_{k+1} - y_{k} \right\rVert _{\gamma \Q - L \Id}^{2}
	- \dfrac{\gamma}{2 \rho} t_{k+1}^{2} \left\lVert \lambda_{k+1} - \mu_{k} \right\rVert ^{2} . \label{dec:inq}
	\end{align}
\end{prop}
\begin{proof}
	Let $\left( x , \lambda \right) \in \Fea \times \sG$ and $k \geq 1$ be fixed. Multiplying \eqref{inq-2} by $t_{k+1} \left( t_{k+1} - 1 \right) \geq 0$ and  \eqref{inq-1} by $\gamma t_{k+1}$, and adding the resulting inequalities, yields
	\begin{align}
	& \ t_{k+1} \left( t_{k+1} - 1 + \gamma \right) \begin{pmatrix}
	\Lb \left( x_{k+1} , \lambda \right) - \Lb \left( x , \lambda_{k+1} \right)
	\end{pmatrix} \nonumber \\
	\leq & \ t_{k+1} \left( t_{k+1} - 1 \right) \begin{pmatrix}
	\Lb \left( x_{k} , \lambda \right) - \Lb \left( x , \lambda_{k} \right)
	\end{pmatrix} \nonumber \\
	& + \dfrac{1}{\gamma} \left( 1 - \gamma \right) t_{k+1} \left\langle \lambda_{k} - \lambda_{k+1} , \gamma \left( Ax_{k+1} - b \right) + \left( t_{k+1} - 1 \right) \left( Ax_{k+1} - Ax_{k} \right) \right\rangle \nonumber \\
	& + t_{k+1} \left\langle y_{k} - x_{k+1} , \gamma \left( x_{k+1} - x \right) + \left( t_{k+1} - 1 \right) \left( x_{k+1} - x_{k} \right) \right\rangle _{\Q} \nonumber \\
	& + \dfrac{1}{\rho} t_{k+1} \left\langle \mu_{k} - \lambda_{k+1} , \gamma \left( \lambda_{k+1} - \lambda \right) + \left( t_{k+1} - 1 \right) \left( \lambda_{k+1} - \lambda_{k} \right) \right\rangle \nonumber \\
	& - \dfrac{\beta \gamma}{2} t_{k+1} \left\lVert Ax_{k+1} - b \right\rVert ^{2}
	- \dfrac{\beta}{2} t_{k+1} \left( t_{k+1} - 1 \right) \left\lVert Ax_{k+1} - Ax_{k} \right\rVert ^{2} \nonumber \\
	& - \dfrac{\gamma}{2L} t_{k+1} \left\lVert \nabla f \left( y_{k} \right) - \nabla f \left( x \right) \right\rVert ^{2}
	- \dfrac{1}{2L} t_{k+1} \left( t_{k+1} - 1 \right) \left\lVert \nabla f \left( y_{k} \right) - \nabla f \left( x_{k} \right) \right\rVert ^{2} \nonumber \\
	& + \dfrac{L}{2} t_{k+1} \left( t_{k+1} - 1 + \gamma \right) \left\lVert x_{k+1} - y_{k} \right\rVert ^{2} . \label{dec:pre}
	\end{align}
According to \eqref{algo:mu}, \eqref{algo:z-gamma} and \eqref{algo:lambda} we have
	\begin{align}
	& \dfrac{1}{\gamma} \left( 1 - \gamma \right) t_{k+1} \left\langle \lambda_{k} - \lambda_{k+1} , \gamma \left( Ax_{k+1} - b \right) + \left( t_{k+1} - 1 \right) \left( Ax_{k+1} - Ax_{k} \right) \right\rangle \nonumber \\
	= \ 	& \dfrac{1}{\gamma} \left( 1 - \gamma \right) t_{k+1} \left\langle \lambda_{k} - \lambda_{k+1} , Az_{k+1}^{\gamma} - \gamma b \right\rangle
	= \dfrac{1}{\rho} \left( 1 - \gamma \right) t_{k+1} \left\langle \lambda_{k} - \lambda_{k+1} , \lambda_{k+1} - \mu_{k} \right\rangle \nonumber \\
	= \ 	& - \dfrac{1 - \gamma}{2 \rho} t_{k+1} \left\lVert \lambda_{k+1} - \lambda_{k} \right\rVert ^{2}
	- \dfrac{1 - \gamma}{2 \rho} t_{k+1} \left\lVert \lambda_{k+1} - \mu_{k} \right\rVert ^{2}
	+ \dfrac{1 - \gamma}{2 \rho} t_{k+1} \left\lVert \mu_{k} - \lambda_{k} \right\rVert ^{2} \nonumber \\
	\leq \ 	& - \dfrac{1 - \gamma}{2 \rho} t_{k+1} \left\lVert \lambda_{k+1} - \lambda_{k} \right\rVert ^{2} +  \dfrac{1 - \gamma}{2 \rho} \frac{\left( t_{k} - 1 \right)^2}{t_{k+1}} \left\lVert \lambda_{k} - \lambda_{k-1} \right\rVert ^{2} \nonumber \\
\leq \ & - \dfrac{1 - \gamma}{2 \rho} t_{k+1} \left\lVert \lambda_{k+1} - \lambda_{k} \right\rVert ^{2} + \dfrac{1 - \gamma}{2 \rho} \left( t_{k} - 1 \right) \left\lVert \lambda_{k} - \lambda_{k-1} \right\rVert ^{2}, \label{dec:inn:mix}
	\end{align}
where in the last inequality we use that $\{t_k\}_{k \geq 1}$ is nondecreasing and that $t_{k} \geq 1$ for every $k \geq 1$.
	
On the other hand, \eqref{algo:z-gamma}, \eqref{eq:dz:z-x} and \eqref{pre:vm:id-eq} give 
	\begin{align}
	& \ t_{k+1} \left\langle y_{k} - x_{k+1} , \gamma \left( x_{k+1} - x \right) + \left( t_{k+1} - 1 \right) \left( x_{k+1} - x_{k} \right) \right\rangle _{\Q} \nonumber \\
	= & \left\langle z_{k}^{\gamma} - z_{k+1}^{\gamma} , z_{k+1}^{\gamma} - \gamma x \right\rangle _{\Q}
	+ \left( \gamma - 1 \right) \left\langle x_{k+1} - x_{k} , \gamma \left( x_{k+1} - x \right) + \left( t_{k+1} - 1 \right) \left( x_{k+1} - x_{k} \right) \right\rangle _{\Q} \nonumber \\
	= & - \dfrac{1}{2} \left\lVert z_{k+1}^{\gamma} - z_{k}^{\gamma} \right\rVert _{\Q}^{2} - \dfrac{1}{2} \left\lVert z_{k+1}^{\gamma} - \gamma x \right\rVert _{\Q}^{2} + \dfrac{1}{2} \left\lVert z_{k}^{\gamma} - \gamma x \right\rVert _{\Q}^{2} + \dfrac{1}{2} \gamma \left( \gamma - 1 \right) \left\lVert x_{k+1} - x_{k} \right\rVert _{\Q}^{2} \nonumber \\
	& + \dfrac{1}{2} \gamma \left( \gamma - 1 \right) \left\lVert x_{k+1} - x \right\rVert _{\Q}^{2} - \dfrac{1}{2} \gamma \left( \gamma - 1 \right) \left\lVert x_{k} - x \right\rVert _{\Q}^{2} + \left( \gamma - 1 \right) \left( t_{k+1} - 1 \right) \left\lVert x_{k+1} - x_{k} \right\rVert _{\Q}^{2} . \label{dec:inn:x:pre}
	\end{align}
From \eqref{pre:vm:sum-n0}, \eqref{eq:dz:x-x} and \eqref{eq:x-y-z} we have 
	\begin{align*}
	- \dfrac{1}{2} \left\lVert z_{k+1}^{\gamma} - z_{k}^{\gamma} \right\rVert _{\Q}^{2}
	 = & - \dfrac{1}{2} \left\lVert z_{k+1} - z_{k} + \left( \gamma - 1 \right) \left( x_{k+1} - x_{k} \right) \right\rVert _{\Q}^{2} \nonumber \\
	= & - \dfrac{1}{2} \gamma \left\lVert z_{k+1} - z_{k} \right\rVert _{\Q}^{2} - \dfrac{1}{2} \gamma \left( \gamma - 1 \right) \left\lVert x_{k+1} - x_{k} \right\rVert _{\Q}^{2} \nonumber \\
	& + \dfrac{1}{2} \left( \gamma - 1 \right) \left\lVert z_{k+1} - z_{k} - \left( x_{k+1} - x_{k} \right) \right\rVert _{\Q}^{2} \nonumber \\
	\leq & - \dfrac{1}{2} \gamma t_{k+1}^{2} \left\lVert x_{k+1} - y_{k} \right\rVert _{\Q}^{2} - \dfrac{1}{2} \gamma \left( \gamma - 1 \right) \left\lVert x_{k+1} - x_{k} \right\rVert _{\Q}^{2},
	\end{align*}
which we combine with \eqref{dec:inn:x:pre} to obtain
	\begin{align}
	& \ t_{k+1} \left\langle y_{k} - x_{k+1} , \gamma \left( x_{k+1} - x \right) + \left( t_{k+1} - 1 \right) \left( x_{k+1} - x_{k} \right) \right\rangle _{\Q} \nonumber \\
	\leq & - \dfrac{1}{2} \gamma t_{k+1}^{2} \left\lVert x_{k+1} - y_{k} \right\rVert _{\Q}^{2} - \dfrac{1}{2} \left\lVert z_{k+1}^{\gamma} - \gamma x \right\rVert _{\Q}^{2} + \dfrac{1}{2} \left\lVert z_{k}^{\gamma} - \gamma x \right\rVert _{\Q}^{2} - \dfrac{1}{2} \gamma \left( 1 - \gamma \right) \left\lVert x_{k+1} - x \right\rVert _{\Q}^{2} \nonumber \\
	&  + \dfrac{1}{2} \gamma \left( 1 - \gamma \right) \left\lVert x_{k} - x \right\rVert _{\Q}^{2} + \left( \gamma - 1 \right) \left( t_{k+1} - 1 \right) \left\lVert x_{k+1} - x_{k} \right\rVert _{\Q}^{2} . \label{dec:inn:x}
	\end{align}
By using the same technique, we can derive that
	\begin{align}
	& \ \dfrac{1}{\rho} t_{k+1} \left\langle \mu_{k} - \lambda_{k+1} , \gamma \left( \lambda_{k+1} - \lambda \right) + \left( t_{k+1} - 1 \right) \left( \lambda_{k+1} - \lambda_{k} \right) \right\rangle \nonumber \\
	\leq & - \dfrac{\gamma}{2 \rho} t_{k+1}^{2} \left\lVert \lambda_{k+1} - \mu_{k} \right\rVert ^{2} - \dfrac{1}{2 \rho} \left\lVert \nu_{k+1}^{\gamma} - \gamma \lambda \right\rVert ^{2} + \dfrac{1}{2} \left\lVert \nu_{k}^{\gamma} - \gamma \lambda \right\rVert ^{2} - \dfrac{\gamma}{2 \rho} \left( 1 - \gamma \right) \left\lVert \lambda_{k+1} - \lambda \right\rVert ^{2} \nonumber \\
	& + \dfrac{\gamma}{2 \rho} \left( 1 - \gamma \right) \left\lVert \lambda_{k} - \lambda \right\rVert ^{2} + \dfrac{1}{\rho} \left( \gamma - 1 \right) \left( t_{k+1} - 1 \right) \left\lVert \lambda_{k+1} - \lambda_{k} \right\rVert ^{2} . \label{dec:inn:lambda}
	\end{align}
Plugging \eqref{dec:inn:mix}, \eqref{dec:inn:x} and \eqref{dec:inn:lambda} into \eqref{dec:pre}, and taking into consideration the fact that $\gamma \in \left( 0 , 1 \right]$, gives the desired statement.
\end{proof}

Next we record some direct consequences of the above estimate.
\begin{prop}
	\label{prop:sum}
	Let $\left\lbrace \left( x_{k} , \lambda_{k} \right) \right\rbrace _{k \geq 0}$ be the sequence generated by Algorithm \ref{algo:fast} and  $\left( x_{*} , \lambda_{*} \right) \in \sol$. Then the sequence $\left\lbrace \E_{k} \left( x_{*} , \lambda_{*} \right) \right\rbrace _{k \geq 1}$ is nonincreasing and the following statements are true
	\begin{subequations}
		\begin{align*}
		\left( 1- \frac{m}{\gamma} \right) \mysum_{k \geq 1} t_{k} \begin{pmatrix}
		\Lb \left( x_{k} , \lambda_{*} \right) - \Lb \left( x_{*} , \lambda_{k} \right)
		\end{pmatrix} & < + \infty \\
		\mysum_{k \geq 1} t_{k+1} \left( \beta \left\lVert Ax_{k+1} - b \right\rVert ^{2} + \dfrac{1}{L} \left\lVert \nabla f \left( y_{k} \right) - \nabla f \left( x_{*} \right) \right\rVert ^{2} \right) & < + \infty \\
		\mysum_{k \geq 1} t_{k+1} \left( t_{k+1} - 1 \right) \left( \beta \left\lVert Ax_{k+1} - Ax_{k} \right\rVert ^{2} + \dfrac{1}{L} \left\lVert \nabla f \left( y_{k} \right) - \nabla f \left( x_{k} \right) \right\rVert ^{2} \right) & < + \infty \\
		\left( 1 - \gamma \right) \mysum_{k \geq 1} \left( t_{k+1} - 1 \right)  \left\lVert x_{k+1} - x_{k} \right\rVert _{\Q}^{2} & < + \infty \\
(1 - \gamma)\mysum_{k \geq 1} \left( 2t_{k+1} - 1 \right) \left\lVert \lambda_{k+1} - \lambda_{k} \right\rVert ^{2} & < + \infty \\
		\mysum_{k \geq 1} t_{k+1}^{2} \left( \left\lVert x_{k+1} - y_{k} \right\rVert _{\gamma \Q - L \Id}^{2} + \dfrac{\gamma}{\rho} \left\lVert \lambda_{k+1} - \mu_{k} \right\rVert ^{2} \right) & < + \infty .
		\end{align*}
	\end{subequations}
\end{prop}
\begin{proof}
	Since $\left\lbrace t_{k} \right\rbrace _{k \geq 1}$ is an nondecreasing sequence that satisfies \eqref{assume:t-k+} and $0 < m \leq \gamma \leq 1$, we have for every $k \geq 1$
	\begin{equation*}
	t_{k+1}^{2} - t_{k+1} - t_{k}^{2} + \left( 1 - \gamma \right) t_{k}
	\leq \left( m - 1 \right) t_{k+1} + \left( 1 - \gamma \right) t_{k}
	\leq \left( m - \gamma \right) t_{k} \leq 0 .
	\end{equation*}	
	Moreover, as $\left( x_{*} , \lambda_{*} \right) \in \sol$, we must have $x_{*} \in \Fea$ and $\Lb \left( x_{k} , \lambda_{*} \right) - \Lb \left( x_{*} , \lambda_{k} \right) \geq 0$ for every $k \geq 1$ due to \eqref{intro:saddle}.
	Combining these observations, we deduce from the inequality \eqref{dec:inq} applied to $(x,\lambda)=(x_*,\lambda_*)$ that for every $k \geq 1$
	\begin{align}
	\E_{k+1} \left( x_{*} , \lambda_{*} \right)
	\leq & \  \E_{k} \left( x_{*} , \lambda_{*} \right)
	- \left( \gamma - m \right) t_{k} \begin{pmatrix}
	\Lb \left( x_{k} , \lambda_{*} \right) - \Lb \left( x_{*} , \lambda_{k} \right)
	\end{pmatrix} \nonumber \\
	& - \dfrac{\beta \gamma}{2} t_{k+1} \left\lVert Ax_{k+1} - b \right\rVert ^{2}
	- \dfrac{\beta}{2} t_{k+1} \left( t_{k+1} - 1 \right) \left\lVert Ax_{k+1} - Ax_{k} \right\rVert ^{2} \nonumber \\
	& - \dfrac{\gamma}{2L} t_{k+1} \left\lVert \nabla f \left( y_{k} \right) - \nabla f \left( x_{*} \right) \right\rVert ^{2}
	- \dfrac{1}{2L} t_{k+1} \left( t_{k+1} - 1 \right) \left\lVert \nabla f \left( y_{k} \right) - \nabla f \left( x_{k} \right) \right\rVert ^{2} \nonumber \\
	& - \left( 1 - \gamma \right) \left( t_{k+1} - 1 \right) \left\lVert x_{k+1} - x_{k} \right\rVert _{\Q}^{2}
	- \dfrac{1 - \gamma}{2 \rho} \left( 2t_{k+1} - 1 \right) \left\lVert \lambda_{k+1} - \lambda_{k} \right\rVert ^{2} \nonumber \\
	&  - \dfrac{1}{2} t_{k+1}^{2}  \left\lVert x_{k+1} - y_{k} \right\rVert _{\gamma \Q - L \Id}^{2}
	- \dfrac{\gamma}{2 \rho} t_{k+1}^{2} \left\lVert \lambda_{k+1} - \mu_{k} \right\rVert ^{2}. \label{sum:inq}
	\end{align}
By applying Lemma \ref{lem:quasi-Fej} we obtain all conclusions.
\end{proof}

\begin{rmk}
	\label{rmk:bnd}
	Since the sequence $\left\lbrace \E_{k} \left( x_{*} , \lambda_{*} \right) \right\rbrace _{k \geq 1}$ is nonincreasing and for every $k \geq 1$
	\begin{equation}\label{prop-tk}
	\gamma t_{k}^{2} \leq t_{k} \left( t_{k} - 1 + \gamma \right) \Leftrightarrow t_{k} \geq 1 ,
	\end{equation}
	we deduce that
	\begin{align*}
	\dfrac{\beta \gamma}{2} t_{k}^{2} \left\lVert Ax_{k} - b \right\rVert ^{2}
	& \leq \gamma t_{k}^{2} \begin{pmatrix}
	\Lb \left( x_{k} , \lambda_{*} \right) - \Lb \left( x_{*} , \lambda_{k} \right)
	\end{pmatrix} \nonumber \\
	& \leq t_{k} \left( t_{k} - 1 + \gamma \right) \begin{pmatrix}
	\Lb \left( x_{k} , \lambda_{*} \right) - \Lb \left( x_{*} , \lambda_{k} \right)
	\end{pmatrix} \nonumber \\
	& \leq \E_{k} \left( x_{*} , \lambda_{*} \right) \leq \cdots \leq \E_{1} \left( x_{*} , \lambda_{*} \right) .
	\end{align*}
	Consequently, for every $k \geq 1$ we have
	\begin{align}
	\label{bnd:pre}
	& \ \ t_{k}^{2} \Big( {\cal L}\left( x_{k} , \lambda_{*} \right) - {\cal L} \left( x_{*} , \lambda_{k} \right) \Big) \leq t_{k}^{2} \Big ( \Lb \left( x_{k} , \lambda_{*} \right) - \Lb \left( x_{*} , \lambda_{k} \right) \Big) \leq \dfrac{\E_{1} \left( x_{*} , \lambda_{*} \right)}{\gamma} \nonumber \\
\textrm{and, when }  \beta >0, & \ \ t_{k} \left\lVert Ax_{k} - b \right\rVert \leq \sqrt{\dfrac{2 \E_{1} \left( x_{*} , \lambda_{*} \right)}{\beta \gamma}}.
	\end{align}
\end{rmk}

\begin{rmk}
Recall that from Proposition \ref{prop:sum} we have
	\begin{equation}\label{sum}
	\mysum_{k \geq 1} \left( t_{k+1} - 1 \right) \left( \left\lVert x_{k+1} - x_{k} \right\rVert _{\Q}^{2} + \dfrac{1}{2\rho} \left\lVert \lambda_{k+1} - \lambda_{k} \right\rVert ^{2} \right) < + \infty,
	\end{equation}
whenever $\gamma  < 1$. Taking into account the way $\gamma$ has arisen in the context of the dynamical system \eqref{ds:PD-AVD} (see \eqref{ds:gamma}), this corresponds to
\begin{equation*}
	\gamma = \dfrac{1}{\theta \left( \alpha - 1 \right)} < 1
	\Leftrightarrow \dfrac{1}{\alpha - 1} < \theta.
	\end{equation*}
In the continuous case it has been proved (see \cite[Theorem 3.2]{Bot-Nguyen}) that, if $\dfrac{1}{\alpha - 1} < \theta$, then
 	\begin{equation*}
	\int_{t_{0}}^{+ \infty} t \left\lVert \left( \dot{x} \left( t \right) , \dot{\lambda} \left( t \right) \right) \right\rVert ^{2} < +\infty ,
	\end{equation*}
which can be seen as the continuous counterpart of \eqref{sum}. Both statements play a crucial role in the proof of the convergence of the sequence of iterates generated by Algorithm \ref{algo:fast}  and of the trajectory generated by \eqref{ds:PD-AVD}, respectively.
\end{rmk}

The following result, which complements the statements of Proposition \ref{prop:sum}, will also play a crucial role in the proof of the convergence of the sequence of iterates.

\begin{prop}
	\label{prop:dual}
	Let $\left\lbrace \left( x_{k} , \lambda_{k} \right) \right\rbrace _{k \geq 0}$ be the sequence generated by Algorithm \ref{algo:fast} with
$$0 < \sigma < \frac{\gamma}{L + \gamma \beta \|A\|^2},$$
and  $\left( x_{*} , \lambda_{*} \right) \in \sol$. Then the following statements are true
	\begin{subequations}
		\begin{align}
		\left( 1 - \dfrac{m}{\gamma} \right) \mysum_{k \geq 1} t_{k} \left\lVert A^{*} \left( \lambda_{k} - \lambda_{*} \right) \right\rVert ^{2} & < + \infty, \label{dual:sum:1} \\
		t_{k+1} \left( t_{k+1} - 1 \right) ^{2} \mysum_{k \geq 1} \left\lVert A^{*} \left( \lambda_{k+1} - \lambda_{k} \right) \right\rVert ^{2} & < + \infty. \label{dual:sum:2}
		\end{align}
	\end{subequations}
	In addition, there exists $\CAstar > 0$ such that for every $k \geq 1$
	\begin{equation*}
	\left\lVert A^{*} \left( \lambda_{k} - \lambda_{*} \right) \right\rVert \leq \dfrac{\CAstar}{t_{k}} .
	\end{equation*}
\end{prop}
\begin{proof}
From \eqref{Fej:opt-x}, after rearranging some terms, we have for every $k \geq 1$
	\begin{align*}
	A^{*} \left( \dfrac{1}{\gamma} \nu_{k+1}^{\gamma} - \lambda_{*} \right) 
	 = & \ \dfrac{1}{\gamma} \left( 1 - \gamma \right) A^{*} \left( \lambda_{k} - \lambda_{k+1} \right)
	+ \nabla f \left( x_{*} \right) - \nabla f \left( y_{k} \right)
	+ \dfrac{1}{\sigma} \left( y_{k} - x_{k+1} \right) \nonumber \\
	& + \beta A^{*} A \left( x_{k+1} - y_{k} \right) - \beta A^{*} \left( Ax_{k+1} - b \right) .
	\end{align*}
It follows from Proposition \ref{prop:sum}, by using \eqref{ineq:Q} and the fact that $t_k \geq 1$, that for every $k \geq 1$ 
	\begin{align*}
	& \mysum_{k \geq 1} t_{k+1} \left\lVert A^{*} \left( \dfrac{1}{\gamma} \nu_{k+1}^{\gamma} - \lambda_{*} \right) \right\rVert ^{2} \nonumber \\
	\leq \ 	& \dfrac{5}{\gamma^{2}} \left( 1 - \gamma \right) ^{2} \left\lVert A \right\rVert ^{2} \mysum_{k \geq 1} t_{k+1} \left\lVert \lambda_{k+1} - \lambda_{k} \right\rVert ^{2}
	+ 5 \mysum_{k \geq 1} t_{k+1} \left\lVert \nabla f \left( y_{k} \right) - \nabla f \left( x_{*} \right) \right\rVert ^{2} \nonumber \\
	& + \dfrac{5}{\sigma^{2}} \mysum_{k \geq 1} t_{k+1} \left\lVert x_{k+1} - y_{k} \right\rVert ^{2} 
	+ 5 \beta^{2} \left\lVert A^{*} A \right\rVert ^{2} \mysum_{k \geq 1} t_{k+1} \left\lVert x_{k+1} - y_{k} \right\rVert ^{2} \nonumber \\
	& + 5 \beta^{2} \left\lVert A \right\rVert ^{2} \mysum_{k \geq 1} t_{k+1} \left\lVert Ax_{k+1} - b \right\rVert ^{2} < + \infty .
	\end{align*}
According to \eqref{algo:nu-gamma} we have for every $k \geq1$
	\begin{equation*}
	A^{*} \left( \dfrac{1}{\gamma} \nu_{k+1}^{\gamma} - \lambda_{*} \right)
	= \left( 1 + \dfrac{1}{\gamma} \left( t_{k+1} - 1 \right) \right) A^{*} \left( \lambda_{k+1} - \lambda_{*} \right) - \dfrac{1}{\gamma} \left( t_{k+1} - 1 \right) A^{*} \left( \lambda_{k} - \lambda_{*} \right),
	\end{equation*}
hence, by applying the identity \eqref{pre:sum-1}, we get
	\begin{align}
	\left\lVert A^{*} \left( \dfrac{1}{\gamma} \nu_{k+1}^{\gamma} - \lambda_{*} \right) \right\rVert ^{2}
	= & \left( 1 + \dfrac{1}{\gamma} \left( t_{k+1} - 1 \right) \right) \left\lVert A^{*} \left( \lambda_{k+1} - \lambda_{*} \right) \right\rVert ^{2} - \dfrac{1}{\gamma} \left( t_{k+1} - 1 \right) \left\lVert A^{*} \left( \lambda_{k} - \lambda_{*} \right) \right\rVert ^{2} \nonumber \\
	& + \dfrac{1}{\gamma} \left( t_{k+1} - 1 \right) \left( 1 + \dfrac{1}{\gamma} \left( t_{k+1} - 1 \right) \right) \left\lVert A^{*} \left( \lambda_{k+1} - \lambda_{k} \right) \right\rVert ^{2} . \label{dual:norm-1}
	\end{align}
	On the other hand, it follows from condition \eqref{assume:t-k+} and the fact that $\left\lbrace t_{k} \right\rbrace _{k \geq 1}$ is nondecreasing that for every $k \geq 1$
	\begin{align}
	\label{dual:dt}
	\dfrac{1}{\gamma} t_{k+1} \left( t_{k+1} - 1 \right) - t_{k} \left( 1 + \dfrac{1}{\gamma} \left( t_{k} - 1 \right) \right)
	& = \dfrac{1}{\gamma} \left( t_{k+1}^{2} - t_{k+1} - t_{k}^{2} + t_{k} \right) - t_{k} \nonumber \\
	& \leq \dfrac{1}{\gamma} \left( \left( m - 1 \right) t_{k+1} + t_{k} \right) - t_{k} \nonumber \\
	& = \dfrac{m - 1}{\gamma} \left( t_{k+1} - t_{k} \right) + \left( \dfrac{m}{\gamma} - 1 \right) t_{k} \nonumber \\
	& \leq \left( \dfrac{m}{\gamma} - 1 \right) t_{k} ,
	\end{align}
Combining \eqref{dual:norm-1} and \eqref{dual:dt}, it yields for every $k \geq 1$
	\begin{align*}
	& t_{k+1} \left( 1 + \dfrac{1}{\gamma} \left( t_{k+1} - 1 \right) \right) \left\lVert A^{*} \left( \lambda_{k+1} - \lambda_{*} \right) \right\rVert ^{2} \nonumber \\
	= & \ t_{k} \left( 1 + \dfrac{1}{\gamma} \left( t_{k} - 1 \right) \right) \left\lVert A^{*} \left( \lambda_{k} - \lambda_{*} \right) \right\rVert ^{2}
	+ \left( \dfrac{1}{\gamma} t_{k+1} \left( t_{k+1} - 1 \right) - t_{k} \left( 1 + \dfrac{1}{\gamma} \left( t_{k} - 1 \right) \right) \right) \left\lVert A^{*} \left( \lambda_{k} - \lambda_{*} \right) \right\rVert ^{2} \nonumber \\
	& - \dfrac{1}{\gamma} t_{k+1} \left( t_{k+1} - 1 \right) \left( 1 + \dfrac{1}{\gamma} \left( t_{k+1} - 1 \right) \right) \left\lVert A^{*} \left( \lambda_{k+1} - \lambda_{k} \right) \right\rVert ^{2}
	+ t_{k+1}\left\lVert A^{*} \left( \dfrac{1}{\gamma} \nu_{k+1}^{\gamma} - \lambda_{*} \right) \right\rVert ^{2} \nonumber \\
	\leq \ 	& t_{k} \left( 1 + \dfrac{1}{\gamma} \left( t_{k} - 1 \right) \right) \left\lVert A^{*} \left( \lambda_{k} - \lambda_{*} \right) \right\rVert ^{2}
	- \left( 1 - \dfrac{m}{\gamma} \right) t_{k} \left\lVert A^{*} \left( \lambda_{k} - \lambda_{*} \right) \right\rVert ^{2} \nonumber \\
	& - \dfrac{1}{\gamma^{2}} t_{k+1} \left( t_{k+1} - 1 \right) ^{2} \left\lVert A^{*} \left( \lambda_{k+1} - \lambda_{k} \right) \right\rVert ^{2}
	+ t_{k+1} \left\lVert A^{*} \left( \dfrac{1}{\gamma} \nu_{k+1}^{\gamma} - \lambda_{*} \right) \right\rVert ^{2} .
	\end{align*}
	We are in the setting of inequality \eqref{quasi-Fej:inq} with
	\begin{align*}
	a_{k} & := t_{k} \left( 1 + \dfrac{1}{\gamma} \left( t_{k} - 1 \right) \right) \left\lVert A^{*} \left( \lambda_{k} - \lambda_{*} \right) \right\rVert ^{2} \geq 0 , \nonumber \\
	b_{k} & := \left( 1 - \dfrac{m}{\gamma} \right) t_{k} \left\lVert A^{*} \left( \lambda_{k} - \lambda_{*} \right) \right\rVert ^{2} + \dfrac{1}{\gamma^{2}} t_{k+1} \left( t_{k+1} - 1 \right) ^{2} \left\lVert A^{*} \left( \lambda_{k+1} - \lambda_{k} \right) \right\rVert ^{2} \geq 0 , \nonumber \\
	d_{k} & := t_{k+1} \left\lVert A^{*} \left( \dfrac{1}{\gamma} \nu_{k+1}^{\gamma} - \lambda_{*} \right) \right\rVert ^{2} \geq 0,
	\end{align*}
for every $k \geq 1$. According to Lemma \ref{lem:quasi-Fej}, \eqref{dual:sum:1} and \eqref{dual:sum:2} are fulfilled and the sequence $\left\lbrace t_{k} \left( 1 + \dfrac{1}{\gamma} \left( t_{k} - 1 \right) \right) \left\lVert A^{*} \left( \lambda_{k} - \lambda_{*} \right) \right\rVert ^{2} \right\rbrace _{k \geq 1}$ is convergent, therefore it is bounded. Consequently, there exists $\CAstar > 0$ such that for every $k \geq 1$
	\begin{equation*}
	t_{k}^{2} \left\lVert A^{*} \left( \lambda_{k} - \lambda_{*} \right) \right\rVert ^{2}
	\leq t_{k} \left( 1 + \dfrac{1}{\gamma} \left( t_{k} - 1 \right) \right) \left\lVert A^{*} \left( \lambda_{k} - \lambda_{*} \right) \right\rVert ^{2} \leq \CAstar,
	\end{equation*}
which provides the conclusion.
\end{proof}

\subsection{On the boundedness of the sequences}

In this section we will discuss the boundedness of the sequence of primal-dual iterates $\left\lbrace \left( x_{k} , \lambda_{k} \right) \right\rbrace _{k \geq 0}$ and also of other related sequences which play a role in the convergence analysis.

To this end we define on $\sH \times \sG$ the inner product
\begin{equation*}
\left\langle u , u' \right\rangle _{\W}
= \left\langle \left( x , \lambda \right) , \left( x' , \lambda' \right) \right\rangle _{\W}
= \left\langle x , x' \right\rangle _{\Q} + \dfrac{1}{\rho} \left\langle \lambda , \lambda' \right\rangle  \quad \forall u := \left( x , \lambda \right) , u' := \left( x' , \lambda' \right) \in \sH \times \sG,
\end{equation*}
where ${\Q}$ is the operator defined in \eqref{define:Q} which we proved to be positive definite under assumption \eqref{assume:sigma}. The norm induced by this scalar product is
\begin{equation*}
\left\lVert u \right\rVert _{\W}
= \left\lVert \left( x , \lambda \right) \right\rVert _{\W}
= \sqrt{\left\lVert x \right\rVert _{\Q}^{2} + \dfrac{1}{\rho} \left\lVert \lambda \right\rVert ^{2}} \quad \forall u := \left( x , \lambda \right).
\end{equation*}

The condition on the sequence $\left\lbrace t_k \right\rbrace _{k \geq 1}$ which we will assume in the next proposition in order to guarantee boundedness for the sequences generated by Algorithm \ref{algo:fast} has been proposed in \cite{Bauschke-Bui-Wang}. Later we will see that it is satisfied by the three classical inertial parameters rules by Nesterov, Chambolle-Dossal and Attouch-Cabot.

\begin{prop}
	\label{prop:bnd}
	Let $\left\lbrace \left( x_{k} , \lambda_{k} \right) \right\rbrace _{k \geq 0}$ be the sequence generated by Algorithm \ref{algo:fast}.
	Suppose that
	\begin{equation}
	\label{assume:sup}
	\kappa := \inf\limits_{k \geq 1} \dfrac{t_{k}}{k} > 0 .
	\end{equation}
Then the sequences $\left\lbrace \left( x_{k} , \lambda_{k} \right) \right\rbrace _{k \geq 0}$, $\left\lbrace \left( z_{k}^{\gamma} , \nu_{k}^{\gamma} \right) \right\rbrace _{k \geq 1}$ and $\left\lbrace \left( t_{k+1} \left( x_{k+1} - x_{k} \right) , t_{k+1} \left( \lambda_{k+1} - \lambda_{k} \right) \right) \right\rbrace _{k \geq 0}$ are bounded.
If, in addition $\beta > 0$, then the sequence $\left\lbrace t_{k+1} \left( t_{k+1} - 1 \right) \left( Ax_{k+1} - Ax_{k} \right) \right\rbrace _{k \geq 0}$ is also bounded.
\end{prop}
\begin{proof}
Let $\left( x_{*} , \lambda_{*} \right) \in \sol$  be fixed. For brevity we will write
\begin{equation*}
u_{*} := \left( x_{*} , \lambda_{*} \right) \in \sol \ \mbox{and} \ u_{k} := \left( x_{k} , \lambda_{k} \right) \in \sH \times \sG \quad \forall k \geq 0.
\end{equation*}		
By applying \eqref{pre:vm:sum-n0}, we have from \eqref{algo:z-gamma} that for every $k \geq 1$
	\begin{align*}
	\left\lVert z_{k}^{\gamma} - \gamma x_{*} \right\rVert _{\Q}^{2}
	  = & \ \left\lVert \left( t_{k} - 1 + \gamma \right) \left( x_{k} - x_{*} \right) - \left( t_{k} - 1 \right) \left( x_{k-1} - x_{*} \right) \right\rVert _{\Q}^{2} \nonumber \\
	 = & \ \gamma \left( t_{k} - 1 + \gamma \right) \left\lVert x_{k} - x_{*} \right\rVert _{\Q}^{2} - \gamma \left( t_{k} - 1 \right) \left\lVert x_{k-1} - x_{*} \right\rVert _{\Q}^{2} \nonumber \\
	&  + \left( t_{k} - 1 + \gamma \right) \left( t_{k} - 1 \right) \left\lVert x_{k} - x_{k-1} \right\rVert _{\Q}^{2} .
	\end{align*}
By applying \eqref{pre:sum-n0}, we have from \eqref{algo:nu-gamma} that for every $k \geq 1$
	\begin{align*}
	\left\lVert \nu_{k}^{\gamma} - \gamma \lambda_{*} \right\rVert ^{2}
	= & \ \left\lVert \left( t_{k} - 1 + \gamma \right) \left( \lambda_{k} - \lambda_{*} \right) - \left( t_{k} - 1 \right) \left( \lambda_{k-1} - \lambda_{*} \right) \right\rVert ^{2} \nonumber \\
	 = & \ \gamma \left( t_{k} - 1 + \gamma \right) \left\lVert \lambda_{k} - \lambda_{*} \right\rVert ^{2} - \gamma \left( t_{k} - 1 \right) \left\lVert \lambda_{k-1} - \lambda_{*} \right\rVert ^{2} \nonumber \\
	& + \left( t_{k} - 1 + \gamma \right) \left( t_{k} - 1 \right) \left\lVert \lambda_{k} - \lambda_{k-1} \right\rVert ^{2} .
	\end{align*}
This means the energy function at $(x_*,\lambda_*)$ can be written for every $k \geq 1$ as
	\begin{align}
	\E_{k} \left( x_* , \lambda_* \right)
	= & \ t_{k} \left( t_{k} - 1 + \gamma \right) \left( \Lb \left( x_{k} , \lambda_* \right) - \Lb \left( x_*, \lambda_{k} \right) \right) \nonumber \\
	& + \dfrac{\gamma}{2} t_{k} \left\lVert u_{k} - u_* \right\rVert _{\W}^{2} - \dfrac{\gamma}{2} \left( t_{k} - 1 \right) \left\lVert u_{k-1} - u_* \right\rVert _{\W}^{2} \nonumber \\
	& + \dfrac{1}{2} \left( t_{k} - 1 + \gamma \right) \left( t_{k} - 1 \right) \left\lVert u_{k} - u_{k-1} \right\rVert _{\W}^{2}
	+ \dfrac{1 - \gamma}{2 \rho} \left( t_{k} - 1 \right) \left\lVert \lambda_{k} - \lambda_{k-1} \right\rVert ^{2} . \label{bnd:E}
	\end{align}
	According to Proposition \ref{prop:sum}, the sequence $\left\lbrace \E_{k} \left( x_{*} , \lambda_{*} \right) \right\rbrace _{k \geq 1}$ is nonincreasing, therefore for every $k \geq 1$
	\begin{align*}
	& \dfrac{\gamma}{2} t_{k} \left\lVert u_{k} - u_{*} \right\rVert _{\W}^{2} - \dfrac{\gamma}{2} \left( t_{k} - 1 \right) \left\lVert u_{k-1} - u_* \right\rVert _{\W}^{2}
	+ \dfrac{1}{2} \left( t_{k} - 1 + \gamma \right) \left( t_{k} - 1 \right) \left\lVert u_{k} - u_{k-1} \right\rVert _{\W}^{2} \nonumber \\
	\leq  & \ \dfrac{1}{2} \left\lVert z_{k}^{\gamma} - \gamma x_{*} \right\rVert _{\Q}^{2} + \dfrac{1}{2 \rho} \left\lVert \nu_{k}^{\gamma} - \gamma \lambda_{*} \right\rVert ^{2}
	\leq \E_{k} \left( x_{*} , \lambda_{*} \right) \leq \cdots \leq \E_{1} \left( x_{*} , \lambda_{*} \right) < + \infty .
	\end{align*}
	From here we conclude that the sequence $\left\lbrace \left( z_{k}^{\gamma} , \nu_{k}^{\gamma} \right) \right\rbrace _{k \geq 1}$ is bounded. In addition, for every $k \geq 1$ it holds
	\begin{align*}
	\dfrac{\gamma}{2} t_{k} \left\lVert u_{k} - u_{*} \right\rVert _{\W}^{2}
	\leq \dfrac{\gamma}{2} \left( t_{k} - 1 \right) \left\lVert u_{k-1} - u_{*} \right\rVert _{\W}^{2} + \E_{1} \left( x_{*} , \lambda_{*} \right)
	\leq \dfrac{\gamma}{2} t_{k-1} \left\lVert u_{k-1} - u_{*} \right\rVert _{\W}^{2} + \E_{1} \left( x_{*} , \lambda_{*} \right) ,
	\end{align*}
	where the last inequality is due to \eqref{bound:t-k}, with the convention $t_{0} := 0$. After telescoping, we get
	\begin{equation*}
	\dfrac{\gamma}{2} t_{k} \left\lVert u_{k} - u_{*} \right\rVert _{\W}^{2}
	\leq k \E_{1} \left( x_{*} , \lambda_{*} \right) \quad \forall k \geq 1.
	\end{equation*}
	Then thanks to \eqref{assume:sup} we obtain
	\begin{equation*}
	\left\lVert u_{k} - u_{*} \right\rVert _{\W}^{2} \leq \dfrac{2k}{\gamma t_{k}} \E_{1} \left( x_{*} , \lambda_{*} \right) \leq \dfrac{2}{\gamma \kappa} \E_{1} \left( x_{*} , \lambda_{*} \right) < + \infty ,
	\end{equation*}
	which means that $\left\lbrace u_{k} := \left( x_{k} , \lambda_{k} \right) \right\rbrace _{k \geq 0}$ is bounded. That $\left\lbrace \left( t_{k+1} \left( x_{k+1} - x_{k} \right) , t_{k+1} \left( \lambda_{k+1} - \lambda_{k} \right) \right) \right\rbrace _{k \geq 0}$ is bounded follows from the fact that for all $k \geq 1$
	\begin{align*}
	t_{k} \left( x_{k} - x_{k-1} \right) 				& = z_{k}^{\gamma} - (\gamma -1) x_{k} - x_{k-1} , \nonumber \\
	t_{k} \left( \lambda_{k} - \lambda_{k-1} \right) 	& = \nu_{k}^{\gamma} - (\gamma -1) \lambda_{k} - \lambda_{k-1} 
	\end{align*}
	
	Finally, recall that from \eqref{eq:nu-gamma-lambda}, \eqref{algo:lambda} and \eqref{algo:z-gamma}, we have for every $k \geq 1$
	\begin{align*}
	\nu_{k+1}^{\gamma} - \nu_{k}^{\gamma} + \left( 1 - \gamma \right) \left( \lambda_{k+1} - \lambda_{k} \right)
	& = t_{k+1} \left( \lambda_{k+1} - \mu_{k} \right) = \dfrac{\rho}{\gamma} t_{k+1}  \left( Az_{k+1}^{\gamma} - \gamma b \right) \nonumber \\
	& = \dfrac{\rho}{\gamma} \left( \gamma t_{k+1} \left( Ax_{k+1} - b \right) + t_{k+1} \left( t_{k+1} - 1 \right) \left( Ax_{k+1} - Ax_{k} \right) \right) .
	\end{align*}
	The last statement of the proposition follows from here and \eqref{bnd:pre}.
\end{proof}

In the following, we will see that the two most prominent choices for the sequence $\{t_k\}_{k \geq 1}$ from the literature, namely, the ones following the rules by Nesterov and by Chambolle-Dossal satisfy not only \eqref{assume:t-k+}, but also \eqref{assume:sup}. 

\begin{ex} {\bf (Nesterov rule)}
	\label{ex:Nes}
The classical construction proposed Nesterov in \cite{Nesterov:83} for $\left\lbrace t_{k} \right\rbrace _{k \geq 1}$ satisfies the following rule
	\begin{equation}
	\label{t-k:Nes}
	t_{1} := 1 \qquad \textrm{ and } \qquad t_{k+1} := \dfrac{1 + \sqrt{1 + 4t_{k}^{2}}}{2} \quad \forall k \geq 1 .
	\end{equation}
The sequence $\left\lbrace t_{k} \right\rbrace _{k \geq 1}$ is strictly increasing and verifies relation \eqref{assume:t-k+} for $m:=1$ with equality. In addition (see, for instance, \cite[Lemma 4.3]{FISTA}), it holds $t_{k} \geq \dfrac{k + 1}{2}$ for every $k \geq 1$, which means that \eqref{assume:sup} is satisfied for $\kappa \geq \dfrac{1}{2}$.
\end{ex}
\begin{ex} {\bf (Chambolle-Dossal rule)}
	\label{ex:C-D}
The construction proposed by Chambolle and Dossal in \cite{Chambolle-Dossal} (see also \cite{Attouch-Peypouquet})
for $\left\lbrace t_{k} \right\rbrace _{k \geq 1}$ satisfies for $\alpha \geq 3$ the following rule
 \begin{equation}
	\label{t-k:C-D}
	t_{k} := 1 + \dfrac{k-1}{\alpha - 1} = \dfrac{k + \alpha - 2}{\alpha - 1} \quad \forall k \geq 1.
	\end{equation}
First we show that this sequence fulfills \eqref{assume:t-k+} with $m := \dfrac{2}{\alpha - 1}  \leq 1$. Indeed, for every $k \geq 1$ we have
	\begin{align}
	t_{k+1}^{2} - m t_{k+1} - t_{k}^{2}
	& = \left( t_{k+1} - t_{k} \right) \left( t_{k+1} + t_{k} \right) - m t_{k+1} 
	= \dfrac{1}{\alpha - 1} \left( 2 + \dfrac{2k - 1}{\alpha - 1} \right) -  \dfrac{2}{\alpha - 1} \dfrac{k + \alpha - 1}{\alpha - 1} \nonumber \\
	& =- \dfrac{1}{\left( \alpha - 1 \right) ^{2}}  < 0. \label{C-D:dts}
	\end{align}
	Furthermore, one can see that for every $k \geq 1$ it holds
	\begin{equation*}
	\dfrac{t_{k}}{k} = \dfrac{1}{\alpha - 1} + \dfrac{\alpha - 2}{k \left( \alpha - 1 \right)} ,
	\end{equation*}
	which proves that \eqref{assume:sup}  is verified for $\kappa =\dfrac{1}{\alpha - 1} $. 

Finally, we observe that, by taking into consideration the choice of $\gamma$ in \eqref{ds:gamma} in the context of
the dynamical system \eqref{ds:PD-AVD} and assumption \eqref{assume:sigma} in Algorithm \ref{algo:fast}, it holds
\begin{equation}
	\label{C-D:m-gamma}
	m = \dfrac{2}{\alpha - 1} \leq \gamma = \dfrac{1}{\theta \left( \alpha - 1 \right)} \Leftrightarrow \theta \leq \dfrac{1}{2} .
	\end{equation}
This connects the choice of the parameter $m$ in Algorithm \ref{algo:fast} with the one of the parameter $\theta$ in \eqref{ds:PD-AVD}. 
\end{ex}

\subsection{Fast convergence rates for the primal-dual gap, the feasibility measure and the objective function value}
\label{subsec:fast-rate}

We have seen in Remark \ref{rmk:bnd} that, for the general choice of the sequence $\{t_k\}_{k \geq  1}$ in  \eqref{assume:t-k+}, the convergence rate of the primal-dual gap is of order $\bO \left( 1 / t_{k}^{2} \right)$ as $k \rightarrow +\infty$. In addition, if $\beta >0$, then the convergence rate of the feasibility measure is of order $\bO \left( 1 / t_{k} \right)$ as $k \rightarrow +\infty$. In this section we will prove that convergence rates of the feasibility measure and of the objective function value are $\bO \left( 1 / t_{k}^{2} \right)$ as $k \rightarrow +\infty$ when the sequence $\{t_k\}_{k \geq  1}$ is chosen by following the rules by Nesterov, Chambolle-Dossal and also Attouch-Cabot. 

In view of \eqref{assume:sup}, this will lead for the primal-dual sequence $\{(x_k,\lambda_k)\}_{k \geq 0}$ generated by Algorithm \ref{algo:fast} and a given primal-dual solution $(x_{*},\lambda_*)$ to the following fast convergence rates
$$\Lag \left( x_{k} , \lambda_{*} \right) - \Lag \left( x_{*} , \lambda_{k} \right) = \bO \left( \dfrac{1}{k^{2}} \right)  \ \mbox{as} \ k  \rightarrow +\infty,$$
$$\left\lVert A x_{k} - b \right\rVert = \bO \left( \dfrac{1}{k^{2}} \right) \mbox{ and } \left\lvert f \left(x_{k} \right) - f_{*} \right\rvert = \bO \left( \dfrac{1}{k^{2}} \right)  \ \mbox{as} \ k  \rightarrow +\infty.$$

We start with the following lemma which holds in the very general setting of Algorithm \ref{algo:fast}.

\begin{lem}
	\label{lem:sup}
	Let $\left\lbrace \left( x_{k} , \lambda_{k} \right) \right\rbrace _{k \geq 0}$ be the sequence generated by Algorithm \ref{algo:fast} and $\left( x_{*} , \lambda_{*} \right) \in \sol$. Then the quantity
	\begin{equation*}
	\Csup := \sup\limits_{\mu \in \sB \left( \lambda_{*} ; 1 \right)}  \E_{1} \left( x_{*} , \mu \right) < + \infty. 
	\end{equation*}
\end{lem}

\begin{proof}
	Let $\left( x_{*} , \lambda_{*} \right) \in \sol$ and $\mu \in \sB \left( \lambda_{*} ; 1 \right)$.
	The Cauchy-Schwarz inequality gives
	\begin{align*}
	\Lb \left( x_{1} , \mu \right) - \Lb \left( x_{*} , \lambda_{1} \right)
	= \ 	& f \left( x_{1} \right) - f \left( x_{*} \right) + \left\langle \mu , Ax_{1} - b \right\rangle + \dfrac{\beta}{2} \left\lVert Ax_{1} - b \right\rVert ^{2} \nonumber \\
	\leq \ 	& f \left( x_{1} \right) - f \left( x_{*} \right) + \left\lVert \mu \right\rVert \left\lVert Ax_{1} - b \right\rVert + \dfrac{\beta}{2} \left\lVert Ax_{1} - b \right\rVert ^{2} \nonumber \\
	\leq \ 	& \CLag := f \left( x_{1} \right) - f \left( x_{*} \right) + \left( 1 + \left\lVert \lambda_{*} \right\rVert \right) \left\lVert Ax_{1} - b \right\rVert + \dfrac{\beta}{2} \left\lVert Ax_{1} - b \right\rVert ^{2} .
	\end{align*}
On the other hand, as $\nu_{1}^{\gamma} = \gamma \lambda_{1}$ and $\mu \in \sB \left( \lambda_{*} ; 1 \right)$, it holds
	\begin{align*}
	& \dfrac{1}{2 \rho} \left\lVert \nu_{1}^{\gamma} - \gamma \mu \right\rVert ^{2}
	+ \dfrac{\gamma}{2 \rho} \left( 1 - \gamma \right) \left\lVert \lambda_{1} - \mu \right\rVert ^{2} \nonumber \\
	\leq \ 	& \dfrac{1}{\rho} \left( \left\lVert \nu_{1}^{\gamma} - \gamma \lambda_{*} \right\rVert ^{2} + \gamma^{2} \left\lVert \mu - \lambda_{*} \right\rVert ^{2} \right)
	+ \dfrac{\gamma}{\rho} \left( 1 - \gamma \right) \left( \left\lVert \lambda_{1} - \lambda_{*} \right\rVert ^{2} + \left\lVert \mu - \lambda_{*} \right\rVert ^{2} \right) \nonumber \\
	\leq \ 	& \Cite := \dfrac{\gamma^{2}}{\rho} \left\lVert \lambda_{1} - \lambda_{*} \right\rVert ^{2} + \dfrac{\gamma}{\rho} \left( 1 - \gamma \right) \left\lVert \lambda_{1} - \lambda_{*} \right\rVert ^{2} + \dfrac{\gamma}{\rho} = \dfrac{\gamma}{\rho} \left( \left\lVert \lambda_{1} - \lambda_{*} \right\rVert ^{2} + 1 \right) .
	\end{align*}
Combining these estimates, as $z_{1}^{\gamma} = \gamma x_{1}$, we have
	\begin{align*}
	\E_{1} \left( x_{*} , \mu \right)
	= & \ t_{1} \left( t_{1} - 1 + \gamma \right) \left( \Lb \left( x_{1} , \mu \right) - \Lb \left( x_{*} , \lambda_{1} \right) \right)
	+ \dfrac{1}{2} \left\lVert z_{1}^{\gamma} - \gamma x_{*} \right\rVert _{\Q}^{2}
	+ \dfrac{1}{2 \rho} \left\lVert \nu_{1}^{\gamma} - \gamma \mu \right\rVert ^{2} \nonumber \\
	& + \dfrac{1}{2} \gamma \left( 1 - \gamma \right) \left\lVert x_{1} - x_{*} \right\rVert _{\Q}^{2}
	+ \dfrac{\gamma}{2 \rho} \left( 1 - \gamma \right) \left\lVert \lambda_{1} - \mu \right\rVert ^{2}
	+ \dfrac{1 - \gamma}{2 \rho} \left( t_{1} - 1 \right) \left\lVert \lambda_{1} - \lambda_{0} \right\rVert ^{2} \nonumber \\
	\leq & \ \gamma \CLag + \Cite + \dfrac{\gamma}{2} \left\lVert x_{1} - x_{*} \right\rVert _{\Q}^{2} < + \infty ,
	\end{align*}
which proves the statement.
\end{proof}

\subsubsection{The Nesterov (\cite{Nesterov:83}) rule}

We have seen that by choosing $\left\lbrace t_{k} \right\rbrace _{k \geq 1}$ as in \eqref{t-k:Nes},  \eqref{assume:t-k+} is fulfilled as equality for $m = 1$, which also yields  $\gamma = 1$ due to \eqref{assume:sigma}. Consequently, from Proposition \ref{prop:dec} it follows that for every $\left( x , \lambda \right) \in \Fea \times \sG$ and every $k \geq 1$ it holds
\begin{equation}
\label{Nes:dec}
\E_{k+1} \left( x , \lambda \right) \leq \E_{k} \left( x , \lambda \right) ,
\end{equation}
which means that the sequence $\left\lbrace \E_{k} \left( x , \lambda \right) \right\rbrace _{k \geq 1}$ is nonincreasing. This statement is stronger than the one in Proposition \ref{prop:sum}, where we have proved that the sequence of function values of the energy function taken at a primal-dual optimal solution is nonincreasing, and will play an important role in the following.

\begin{thm}
	\label{thm:rate:Nes}
	Let $\left\lbrace \left( x_{k} , \lambda_{k} \right) \right\rbrace _{k \geq 0}$ be the sequence generated by Algorithm \ref{algo:fast}, with the sequence $\left\lbrace t_{k} \right\rbrace _{k \geq 1}$ chosen to satisfy Nesterov rule \eqref{t-k:Nes}, and $\left( x_{*} , \lambda_{*} \right) \in \sol$. Then for every $k \geq 1$ it holds
	\begin{equation}
	\label{Nes:inq:Lag}
	0 \leq \Lag \left( x_{k} , \lambda_{*} \right) - \Lag \left( x_{*} , \lambda_{k} \right) + \left\lVert Ax_{k} - b \right\rVert \leq \dfrac{\Csup}{t_{k}^{2}}
	\end{equation}
and
	\begin{equation}
	\label{Nes:inq:fun}
	- \dfrac{\left\lVert \lambda_{*} \right\rVert \Csup}{t_{k}^{2}}
	\leq f \left( x_{k} \right) - f \left( x_{*} \right)
	\leq \dfrac{\left( 1 + \left\lVert \lambda_{*} \right\rVert \right) \Csup}{t_{k}^{2}} .
	\end{equation}
\end{thm}
\begin{proof}
As mentioned earlier in \eqref{Nes:dec}, for every $\left( x , \lambda \right) \in \Fea \times \sG$ and every $k \geq 1$ we have (take into account that $\gamma=1$)
	\begin{equation}
	\label{Nes:pre}
	t_{k}^{2} \left( f \left( x_{k} \right) - f \left( x \right) + \left\langle \lambda , Ax_{k} - b \right\rangle \right) \leq \E_{k} \left( x , \lambda \right) \leq \cdots \leq \E_{1} \left( x , \lambda \right) .
	\end{equation}
We fix $n \geq 1$ and define
	\begin{equation*}
	r_{n} := \begin{dcases}
	\lambda_{*}, 		& \textrm{ if } Ax_{n} - b = 0 \\
	\lambda_{*} + \dfrac{Ax_{n} - b}{\left\lVert Ax_{n} - b \right\rVert},		& \textrm{ if } Ax_{n} - b \neq 0
	\end{dcases} .
	\end{equation*}	
Then $x_{*} \in \Fea$ and $r_{n} \in \sB \left( \lambda_{*} ; 1 \right)$. Hence, $\left( x_{*} , r_{n} \right) \in \Fea \times \sB \left( \lambda_{*} ; 1 \right)$, therefore, according to \eqref{Nes:pre} and Lemma \ref{lem:sup},
	\begin{equation}
	\label{Nes:choose}
	t_{n}^{2} \left( f \left( x_{n} \right) - f \left( x_{*} \right) + \left\langle r_{n} , Ax_{n} - b \right\rangle \right) \leq \E_{1} \left( x_{*} , r_{n} \right) \leq \sup\limits_{\mu \in \sB \left( \lambda_{*} ; 1 \right)}  \E_{1} \left( x_{*} , \mu \right) = \Csup.
	\end{equation}
If $Ax_{n} - b \neq 0$, then
	\begin{align*}
	f \left( x_{n} \right) - f \left( x_{*} \right) + \left\langle r_{n} , Ax_{n} - b \right\rangle
	& = f \left( x_{n} \right) - f \left( x_{*} \right) + \left\langle \lambda_{*} , Ax_{n} - b \right\rangle + \left\lVert Ax_{n} - b \right\rVert \nonumber \\
	& = \Lag \left( x_{n} , \lambda_{*} \right) - \Lag \left( x_{*} , \lambda_{n} \right) + \left\lVert Ax_{n} - b \right\rVert .
	\end{align*}
On the other hand, if $Ax_{n} - b = 0$, we have
	\begin{align*}
	f \left( x_{n} \right) - f \left( x_{*} \right) + \left\langle r_{n} , Ax_{n} - b \right\rangle
	& = f \left( x_{n} \right) - f \left( x_{*} \right) + \left\langle \lambda_{*} , Ax_{n} - b \right\rangle
	= \Lag \left( x_{n} , \lambda_{*} \right) - \Lag \left( x_{*} , \lambda_{n} \right) \nonumber \\
	& = \Lag \left( x_{n} , \lambda_{*} \right) - \Lag \left( x_{*} , \lambda_{n} \right) + \left\lVert Ax_{n} - b \right\rVert,
	\end{align*}
thus, in both scenarios, \eqref{Nes:choose} becomes
	\begin{equation*}
	0 \leq t_{n}^{2} \left( \Lag \left( x_{n} , \lambda_{*} \right) - \Lag \left( x_{*} , \lambda_{n} \right) + \left\lVert Ax_{n} - b \right\rVert \right) \leq \Csup.
	\end{equation*}
	Since $n \geq 1$ has been arbitrarily chosen, we obtain \eqref{Nes:inq:Lag}.
	
As $\Lag \left( x_{k} , \lambda_{*} \right) - \Lag \left( x_{*} , \lambda_{k} \right) \geq 0$, a direct consequent of \eqref{Nes:inq:Lag} is that for every $k \geq 1$
	\begin{equation*}
	0 \leq \left\lVert Ax_{k} - b \right\rVert \leq \dfrac{\Csup}{t_{k}^{2}} .
	\end{equation*}
	From \eqref{Nes:inq:Lag} and the Cauchy-Schwarz inequality, we deduce from here that for every $k \geq 1$
	\begin{align}
	f \left( x_{k} \right) - f \left( x_{*} \right)
	\leq \dfrac{\Csup}{t_{k}^{2}} - \left\langle \lambda_{*} , Ax_{k} - b \right\rangle
	& \leq \dfrac{\Csup}{t_{k}^{2}} + \left\lVert \lambda_{*} \right\rVert \left\lVert Ax_{k} - b \right\rVert \nonumber \\
	& \leq \dfrac{\left( 1 + \left\lVert \lambda_{*} \right\rVert \right) \Csup}{t_{k}^{2}} . \label{Nes:C-S}
	\end{align}
	On the other hand, the convexity of $f$ together with \eqref{intro:opt-Lag} guarantee that for every $k \geq 1$
	\begin{align}
	f \left( x_{k} \right) - f \left( x_{*} \right)
	& \geq \left\langle \nabla f \left( x_{*} \right) , x_{k} - x_{*} \right\rangle
	= - \left\langle A^{*} \lambda_{*} , x_{k} - x_{*} \right\rangle \nonumber \\
	& = - \left\langle \lambda_{*} , Ax_{k} - b \right\rangle
	\geq - \left\lVert \lambda_{*} \right\rVert \left\lVert Ax_{k} - b \right\rVert
	\geq - \dfrac{\left\lVert \lambda_{*} \right\rVert \Csup}{t_{k}^{2}} . \label{Nes:conv}
	\end{align}
	By combining \eqref{Nes:C-S} and \eqref{Nes:conv}, we obtain \eqref{Nes:inq:fun}.
\end{proof}

\subsubsection{The Chambolle-Dossal (\cite{Chambolle-Dossal}) rule}

In this section we prove fast convergence rates for the primal-dual gap, the feasibility measure and the objective function value for the sequence of inertial parameters $\left\lbrace t_{k} \right\rbrace _{k \geq 1}$ following for $\alpha \geq 3$ the Chambole-Dossal rule \eqref{t-k:C-D}. We have seen in Example \ref{ex:C-D} that in this case $\left\lbrace t_{k} \right\rbrace _{k \geq 1}$ fulfills \eqref{assume:t-k+} for $m:=\frac{2}{\alpha-1}$ and \eqref{assume:sup} for $\kappa:=\frac{1}{\alpha-1}$.

For the beginning we observe that for $\frac{2}{\alpha-1} = m \leq \gamma \leq 1$ and every $k \geq 1$ it holds (see \eqref{C-D:dts})
\begin{align}
t_{k} \left( t_{k} - 1 + \gamma \right) - t_{k+1} \left( t_{k+1} - 1 \right)
& = t_{k}^{2} - t_{k+1}^{2} + \left( 1 - \gamma \right) \left( t_{k+1} - t_{k} \right) + \gamma t_{k+1} \nonumber \\
& = - \dfrac{2}{\alpha - 1} t_{k+1} + \dfrac{1}{\left( \alpha - 1 \right) ^{2}} + \dfrac{1 - \gamma}{\alpha - 1} + \gamma t_{k+1} \nonumber\\
& = \dfrac{1}{\alpha-1}(\gamma(\alpha-1)-2)t_{k+1} +  \dfrac{1}{(\alpha-1)^2}(1 - \gamma(\alpha-1)) + \dfrac{1}{\alpha - 1} \nonumber \\
& = \dfrac{1}{\left( \alpha - 1 \right) ^{2}} \Big( \left(\gamma(\alpha-1) - 2 \right) k + \left(\gamma(\alpha-1)- 1 \right) \left( \alpha - 2 \right) \Big) \nonumber \\
& = \dfrac{1}{\left( \alpha - 1 \right) ^{2}} \Big( \left( \gamma(\alpha-1)- 2 \right) \left( k + \alpha - 2 \right) + \alpha - 2 \Big). \label{C-D:dty}
\end{align}
Next we are going to consider two separate cases depending on the relation between $m:=\frac{2}{\alpha-1}$ and $\gamma$.
First we will assume that they are equal, which will then also cover the case $\alpha =3$.

\begin{thm}
	Let $\left\lbrace \left( x_{k} , \lambda_{k} \right) \right\rbrace _{k \geq 0}$ be the sequence generated by Algorithm \ref{algo:fast} with the sequence $\left\lbrace t_{k} \right\rbrace _{k \geq 1}$ chosen to satisfy Chambolle-Dossal rule \eqref{t-k:C-D},  $m:=\frac{2}{\alpha-1} = \gamma \leq 1$, $\beta >0$, and $\left( x_{*} , \lambda_{*} \right) \in \sol$. Then for every $k \geq 2$ it holds
	\begin{equation}\label{C-D1:inq:Lag}
	0 \leq \Lag \left( x_{k} , \lambda_{*} \right) - \Lag \left( x_{*} , \lambda_{k} \right) + \left\lVert Ax_{k} - b \right\rVert \leq \dfrac{C_4}{t_{k}^2}
	\end{equation}
and
	\begin{equation}\label{C-D1:inq:fun}
	- \dfrac{\left\lVert \lambda_{*} \right\rVert \Ctel}{t_{k}^{2}}
	\leq f \left( x_{k} \right) - f \left( x_{*} \right)
	\leq \dfrac{\left( 1 + \left\lVert \lambda_{*} \right\rVert \right) \Ctel}{t_{k}^{2}},
	\end{equation}
	where
	\begin{equation*}
	\Ctel := \dfrac{\Csup}{\gamma} + \dfrac{2 (\alpha - 2)}{\gamma^2 \kappa^{2}\left( \alpha - 1 \right) ^{2}} \left (\Csup + \dfrac{\alpha - 2}{\kappa \left( \alpha - 1 \right) ^{2}} \sqrt{\dfrac{2 \E_{1} \left( x_{*} , \lambda_{*} \right)}{\beta \gamma}} \right) \mysum_{i \geq 1} \dfrac{1}{i^{3/2}}  \in \sR_{+} .
	\end{equation*}
\end{thm}
\begin{proof}
We fix  $n \geq 2$ and define
	\begin{equation*}
	r_{n} := \begin{dcases}
	\lambda_{*}, 		& \textrm{ if } Ax_{n} - b = 0 \\
	\lambda_{*} + \dfrac{Ax_{n} - b}{\left\lVert Ax_{n} - b \right\rVert},		& \textrm{ if } Ax_{n} - b \neq 0
	\end{dcases} .
	\end{equation*}	
Then $x_{*} \in \Fea$ and $r_{n} \in \sB \left( \lambda_{*} ; 1 \right)$. Since $\gamma(\alpha-1) =2$, according to \eqref{C-D:dty}, we have for every $k \geq 1$
	\begin{align*}
	& \left( t_{k+1}^{2} - t_{k+1} - t_{k}^{2} + \left( 1 - \gamma \right) t_{k} \right) \begin{pmatrix}
	\Lb \left( x_{k} , r_{n} \right) - \Lb \left( x_{*} , \lambda_{k} \right) 
	\end{pmatrix} \nonumber \\
	= & - \dfrac{\alpha - 2}{\left( \alpha - 1 \right) ^{2}} \begin{pmatrix}
	\Lb \left( x_{k} , \lambda_{*} \right) - \Lb \left( x_{*} , \lambda_{k} \right) 
	+ \left\langle r_{n} - \lambda_{*} , Ax_{k} - b \right\rangle
	\end{pmatrix} \nonumber \\
	\leq 	& - \dfrac{\alpha - 2}{\left( \alpha - 1 \right) ^{2}} \left\langle r_{n} - \lambda_{*} , Ax_{k} - b \right\rangle .
	\end{align*}
By taking $\left( x , \lambda \right) := \left( x_{*} , r_{n} \right) \in \Fea \times \sB \left( \lambda_{*} ; 1 \right)$ in \eqref{dec:inq}, we obtain for every $k \geq 1$
	\begin{subequations}
		\begin{align}
		\E_{k+1} \left( x_{*} , r_{n} \right) 
		& \leq \E_{k} \left( x_{*} , r_{n} \right) 
		+ \left( t_{k+1}^{2} - t_{k+1} - t_{k}^{2} + \left( 1 - \gamma \right) t_{k} \right) \begin{pmatrix}
		\Lb \left( x_{k} , r_{n} \right) - \Lb \left( x_{*} , \lambda_{k} \right) 
		\end{pmatrix} \nonumber \\
		& \leq \E_{k} \left( x_{*} , r_{n} \right) 
		- \dfrac{\alpha - 2}{\left( \alpha - 1 \right) ^{2}} \left\langle r_{n} - \lambda_{*} , Ax_{k} - b \right\rangle \nonumber \\
		& \leq \E_{k} \left( x_{*} , r_{n} \right) 
		+ \dfrac{\alpha - 2}{\left( \alpha - 1 \right) ^{2}} \left\lVert r_{n} - \lambda_{*} \right\rVert \left\lVert Ax_{k} - b \right\rVert \label{rate:C-D:C-S} \\
		& \leq \E_{k} \left( x_{*} , r_{n} \right) 
		+ \dfrac{\alpha - 2}{\left( \alpha - 1 \right) ^{2}} \sqrt{\dfrac{2 \E_{1} \left( x_{*} , \lambda_{*} \right)}{\beta \gamma}} \dfrac{1}{t_{k}} \label{rate:C-D:t-k} \\
		& \leq \E_{k} \left( x_{*} , r_{n} \right) 
		+ \dfrac{\alpha - 2}{\kappa \left( \alpha - 1 \right) ^{2}} \sqrt{\dfrac{2 \E_{1} \left( x_{*} , \lambda_{*} \right)}{\beta \gamma}} \dfrac{1}{k} \label{rate:C-D:1-k} ,
		\end{align}
	\end{subequations}
where \eqref{rate:C-D:t-k} follows from \eqref{bnd:pre} and\eqref{rate:C-D:1-k} is due to \eqref{assume:sup}. By a telescoping sum argument and Lemma \ref{lem:sup} we conclude that for every $k \geq 1$
	\begin{align*}
	\E_{k+1} \left( x_{*} , r_{n} \right) 
	& \leq \E_{1} \left( x_{*} , r_{n} \right) + \dfrac{\alpha - 2}{\kappa \left( \alpha - 1 \right) ^{2}} \sqrt{\dfrac{2 \E_{1} \left( x_{*} , \lambda_{*} \right)}{\beta \gamma}} \mysum_{i = 1}^{k} \dfrac{1}{i} \nonumber \\
	& \leq \Csup + \dfrac{\alpha - 2}{\kappa \left( \alpha - 1 \right) ^{2}} \sqrt{\dfrac{2 \E_{1} \left( x_{*} , \lambda_{*} \right)}{\beta \gamma}} \Big( \log \left( k \right) + 1 \Big )
	\leq \Clog \Big(\log(k) + 1 \Big),
	\end{align*}
	where
	\begin{equation*}
	\Clog := \Csup + \dfrac{\alpha - 2}{\kappa \left( \alpha - 1 \right) ^{2}} \sqrt{\dfrac{2 \E_{1} \left( x_{*} , \lambda_{*} \right)}{\beta \gamma}} > 0 .
	\end{equation*}
By choosing $k := n-1$, it yields
	\begin{equation*}
	t_{n} \left( t_{n} - 1 + \gamma \right) \left( f \left( x_{n} \right) - f \left( x_{*} \right) + \left\langle r_{n} , Ax_{n} - b \right\rangle \right) 
	\leq \E_{n} \left( x_{*} , r_{n} \right) \leq \Clog \Big(\log(n-1) + 1 \Big) .
	\end{equation*}
We have seen in the proof of Theorem \ref{thm:rate:Nes} that 
\begin{equation}\label{eq:L-f}
f \left( x_{n} \right) - f \left( x_{*} \right) + \left\langle r_{n} , Ax_{n} - b \right \rangle =   \Lag \left( x_{n} , \lambda_{*} \right) - \Lag \left( x_{*} , \lambda_{n} \right) +  \left\lVert Ax_{n} - b \right\rVert ,
\end{equation}
thus, by taking into account \eqref{assume:sup}, we obtain
	\begin{align*}
	\gamma \kappa^{2} n^{2} \left\lVert Ax_{n} - b \right\rVert 
	\leq \gamma t_{n}^{2} \left\lVert Ax_{n} - b \right\rVert 
	& \leq t_{n} \left( t_{n} - 1 + \gamma \right) \left( \Lag \left( x_{n} , \lambda_{*} \right) - \Lag \left( x_{*} , \lambda_{n} \right) + \left\lVert Ax_{n} - b \right\rVert \right) \nonumber \\
	& \leq \E_{n} \left( x_{*} , r_{n} \right) \leq \Clog \Big(\log(n-1) + 1 \Big),
	\end{align*}
therefore, since $2 + \log(n-1) \leq 2(n-1)^{1/2}$,
	\begin{equation*}
	\left\lVert Ax_{n} - b \right\rVert \leq \dfrac{\Clog \Big(\log(n-1) + 1 \Big)}{\gamma \kappa^2 n^{2}} \leq \dfrac{2\Clog}{\gamma \kappa^{2} n^{3/2}} .
	\end{equation*}	
Taking into account also Lemma \ref{lem:sup} and the definition of $\Clog$, we have that for every $k \geq 1$
	\begin{equation*}
	\left\lVert Ax_{k} - b \right\rVert \leq \dfrac{2\Clog}{\gamma \kappa^{2} k^{3/2}}.
	\end{equation*}	
Using this estimate in \eqref{rate:C-D:C-S}, we obtain for every $k \geq 1$
	\begin{equation*}
	\E_{k+1} \left( x_{*} , r_{n} \right) \leq \E_{k} \left( x_{*} , r_{n} \right) + \dfrac{\alpha - 2}{\left( \alpha - 1 \right) ^{2}} \dfrac{2\Clog}{\gamma \kappa^{2} k^{3/2}}.
	\end{equation*}
By using once again a the telescoping sum argument, we conclude that for every $k \geq 1$
	\begin{align*}
	\E_{k+1} \left( x_{*} , r_{n} \right) 
	& \leq \E_{1} \left( x_{*} , r_{n} \right) + \dfrac{2 \Clog(\alpha - 2)}{\gamma \kappa^{2}\left( \alpha - 1 \right) ^{2}}  \mysum_{i = 1}^{k} \dfrac{1}{i^{3/2}} \nonumber \\
	& \leq  \Csup + \dfrac{2 \Clog(\alpha - 2)}{\gamma \kappa^{2}\left( \alpha - 1 \right) ^{2}}  \mysum_{i \geq 1} \dfrac{1}{i^{3/2}} < + \infty .
	\end{align*}
From here, \eqref{C-D1:inq:Lag} follows by choosing $k:=n-1$, and by using that $\gamma t_n^2 \leq t_n(t_n-1+\gamma)$ and \eqref{eq:L-f}. Statement  \eqref{C-D1:inq:fun} follows from \eqref{C-D1:inq:Lag} by repeating the arguments at the end of the proof of Theorem \ref{thm:rate:Nes}.
\end{proof}

Now we come to the second case, namely, when $m:=\frac{2}{\alpha-1} < \gamma \leq 1$, which implicitly requires that $\alpha >3$. For the proof of the fast convergence rates we will make use of the following result which can be found in \cite[Lemma 2]{Li-Lin} (see, also, \cite[Lemma 3.18]{Lin-Li-Fang}).
\begin{lem}
	\label{lem:rate-bnd}
	Let $\left\lbrace \zeta_{k} \right\rbrace _{k \geq 1} \subseteq {\cal G}$ be a sequence such that there exist $\delta >1$ and $M >0$ with the property that for every $K \geq 1$ 
	\begin{equation*}
	\left\lVert \left( \left( \delta - 1 \right) K + \delta \right) \zeta_{K+1} + \mysum_{k = 1}^{K} \zeta_{k} \right\rVert \leq M.
	\end{equation*}
	Then for every $K \geq 1$ it holds
	\begin{equation*}
	\left\lVert \mysum_{k = 1}^{K} \zeta_{k} \right\rVert \leq M.
	\end{equation*}
\end{lem}

\begin{thm}
Let $\left\lbrace \left( x_{k} , \lambda_{k} \right) \right\rbrace _{k \geq 0}$ be the sequence generated by Algorithm \ref{algo:fast} with the sequence $\left\lbrace t_{k} \right\rbrace _{k \geq 1}$ chosen to satisfy Chambolle-Dossal rule \eqref{t-k:C-D},  $m:=\frac{2}{\alpha-1} < \gamma \leq 1$, $\beta >0$, and $\left( x_{*} , \lambda_{*} \right) \in \sol$. Then for every $k \geq 1$ it holds
	\begin{equation}\label{C-D:inq:Lag}
	0 \leq \Lag \left( x_{k} , \lambda_{*} \right) - \Lag \left( x_{*} , \lambda_{k} \right) \leq \dfrac{\E_{1} \left( x_{*} , \lambda_{*} \right)}{\gamma t_{k}^2},
	\end{equation}
	\begin{equation}
	\label{C-D:inq:feas}
	0 \leq \left\lVert Ax_{k} - b \right\rVert \leq \dfrac{\Cbnd}{t_{k}^{2}} ,
	\end{equation}
and
\begin{equation}
	\label{C-D:inq:fun}
	- \dfrac{\left\lVert \lambda_{*} \right\rVert \Cbnd}{t_{k}^{2}}
	\leq f \left( x_{k} \right) - f \left( x_{*} \right) 
	\leq \dfrac{1}{t_{k}^{2}} \left( \dfrac{\E_{1} \left( x_{*} , \lambda_{*} \right)}{\gamma} + \left\lVert \lambda_{*} \right\rVert \Cbnd \right),
	\end{equation}
where
	\begin{align*}
	\Cbnd 		& := 2 \left( 1 + \varphi_{m} \right) ^{2} \left( 2 \left( \alpha - 1 \right) ^{2} \dfrac{\gamma}{\rho} \sup\limits_{k \geq 1} \left\lVert \nu_{k} \right\rVert
	+ (\alpha - 1)^2 \gamma \left\lVert Ax_{1} - b \right\rVert
	+ \dfrac{1}{\kappa} \left( \left\lvert \omega_{0} \right\rvert + \left\lvert \omega_{1} \right\rvert \right) \sqrt{\dfrac{2 \E_{1} \left( x_{*} , \lambda_{*} \right)}{\beta \gamma}} \right),
	\end{align*}
with
$$\delta:= 1+ \frac{1}{\gamma(\alpha-1) - 2} >1,$$ 
$$\omega_{0} := \delta \left( \alpha - 2 \right) - 2 \left( \alpha - 1 \right) \quad \mbox{and} \quad \omega_{1} := \left( \delta - 1 \right) \left( \alpha - 2 \right) - 1.$$
\end{thm}
\begin{proof}
Relation \eqref{C-D:inq:Lag} follows from \eqref{bnd:pre}. We fix $K \geq 1$. For every $1 \leq k \leq K$, according to \eqref{algo:z-gamma}, we have
	\begin{align*}
	t_{k+1} \left( Az_{k+1}^{\gamma} - \gamma b \right)
	 = & \ t_{k+1} \left( t_{k+1} - 1 + \gamma \right) \left( Ax_{k+1} - b \right) - t_{k+1} \left( t_{k+1} - 1 \right) \left( Ax_{k} - b \right) \nonumber \\
	 = & \ t_{k+1} \left( t_{k+1} - 1 + \gamma \right) \left( Ax_{k+1} - b \right) - t_{k} \left( t_{k} - 1 + \gamma \right) \left( Ax_{k} - b \right) \nonumber \\
	& + \bigl( t_{k} \left( t_{k} - 1 + \gamma \right) - t_{k+1} \left( t_{k+1} - 1 \right) \bigr) \left( Ax_{k} - b \right) . 
	\end{align*}	
Taking into consideration \eqref{eq:nu-gamma-lambda}, \eqref{algo:lambda} and \eqref{C-D:dty}, by a telescoping argument it yields
	\begin{align}
	\left( \alpha - 1 \right) ^{2} \dfrac{\gamma}{\rho} \left( \nu_{K + 1} - \nu_{1} \right)
	= & \left( \alpha - 1 \right) ^{2} \dfrac{\gamma}{\rho} \mysum_{k = 1}^{K} \left( \nu_{k+1} - \nu_{k} \right) 
	= \left( \alpha - 1 \right) ^{2} \mysum_{k = 1}^{K} t_{k+1} \left( Az_{k+1}^{\gamma} - \gamma b \right) \nonumber \\
 = & \left( \alpha - 1 \right) ^{2} t_{K+1} \left( t_{K+1} - 1 + \gamma \right) \left( Ax_{K+1} - b \right) - \left( \alpha - 1 \right) ^{2} \gamma \left( Ax_{1} - b \right) \nonumber \\
	& + \left( \alpha - 1 \right) ^{2}  \mysum_{k = 1}^{K} \bigl( t_{k} \left( t_{k} - 1 + \gamma \right) - t_{k+1} \left( t_{k+1} - 1 \right) \bigr) \left( Ax_{k} - b \right) \nonumber \\
	= & \left( K + \alpha - 1 \right) \left( K + \gamma(\alpha-1) \right) \left( Ax_{K+1} - b \right) - \left( \alpha - 1 \right) ^{2} \gamma \left( Ax_{1} - b \right) \nonumber \\
	& + \mysum_{k = 1}^{K} \Big ( \left( \gamma(\alpha-1) - 2 \right) \left( k + \alpha - 2 \right) + \alpha - 2 \Big) \left( Ax_{k} - b \right) . \label{C-D:dnu}
	\end{align}
We define
$$\delta:= 1+ \frac{1}{\gamma(\alpha-1) - 2} >1,$$ 
$$\omega_{0} := \delta \left( \alpha - 2 \right) - 2 \left( \alpha - 1 \right) \quad \mbox{and} \quad \omega_{1} := \left( \delta - 1 \right) \left( \alpha - 2 \right) - 1,$$
and
$$\zeta_{k} := \Big( \left( \gamma(\alpha-1)- 2 \right) \left( k + \alpha - 2 \right) + \alpha - 2 \Big) \left( Ax_{k} - b \right) \ \mbox{for} \ k=1, ..., K.$$
It holds
\begin{align}
	& \left( K + \alpha - 1 \right) \left( K + \gamma(\alpha-1)\right) \left( Ax_{K+1} - b \right) \nonumber \\
	=  & \left( \delta - 1 \right) K \Big( \left( \gamma(\alpha-1)- 2 \right) \left( K + \alpha - 1 \right) + \alpha - 2 \Big) \left( Ax_{K+1} - b \right) \nonumber \\
	& + \delta \Big( \left( \gamma(\alpha-1) - 2 \right) \left( K + \alpha - 1 \right) + \alpha - 2 \Big) \left( Ax_{K+1} - b \right) 
	- \left( \omega_{1} K + \omega_{0} \right) \left( Ax_{K+1} - b \right) \nonumber \\
	= & \left( \left( \delta - 1 \right) K + \delta \right) \zeta_{K+1} - \left( \omega_{1} K + \omega_{0} \right) \left( Ax_{K+1} - b \right) . \label{C-D:set}
	\end{align}
Furthermore, it follows from \eqref{assume:sup} and \eqref{bnd:pre} that
	\begin{align}
	\left\lVert \left( \omega_{1} K + \omega_{0} \right) \left( Ax_{K+1} - b \right) \right\rVert
	& \leq \left( \left\lvert \omega_{0} \right\rvert + \left\lvert \omega_{1} \right\rvert \right) \left( K + 1 \right) \left\lVert Ax_{K+1} - b \right\rVert \nonumber \\
	& \leq \dfrac{1}{\kappa} \left( \left\lvert \omega_{0} \right\rvert + \left\lvert \omega_{1} \right\rvert \right) t_{K+1} \left\lVert Ax_{K+1} - b \right\rVert \nonumber \\
	& \leq \dfrac{1}{\kappa} \left( \left\lvert \omega_{0} \right\rvert + \left\lvert \omega_{1} \right\rvert \right) \sqrt{\dfrac{2 \E_{1} \left( x_{*} , \lambda_{*} \right)}{\beta \gamma}}. \label{C-D:fea}
	\end{align}
Combining the relations \eqref{C-D:dnu}, \eqref{C-D:set} and \eqref{C-D:fea}, we get via the triangle inequality
	\begin{align}
	& \left\lVert \left( \left( \delta - 1 \right) K + \delta \right) \zeta_{K+1} + \mysum_{k = 1}^{K} \zeta_{k} \right\rVert \nonumber \\
	= \ 	& \left\lVert \left( \alpha - 1 \right) ^{2} \dfrac{\gamma}{\rho} \left( \nu_{K + 1} - \nu_{1} \right) + (\alpha - 1)^2 \gamma \left( Ax_{1} - b \right) + \left( \omega_{1} K + \omega_{0} \right) \left( Ax_{K+1} - b \right) \right\rVert \nonumber \\
	\leq \ 	& \left( \alpha - 1 \right) ^{2} \dfrac{\gamma}{\rho} \left\lVert \nu_{K + 1} - \nu_{1} \right\rVert
	+ (\alpha - 1)^2 \gamma \left\lVert Ax_{1} - b \right\rVert
	+ \left\lVert \left( \omega_{1} K + \omega_{0} \right) \left( Ax_{K+1} - b \right) \right\rVert \nonumber \\
	\leq \ 	& \Cnu := 2 \left( \alpha - 1 \right) ^{2} \dfrac{\gamma}{\rho} \sup\limits_{k \geq 1} \left\lVert \nu_{k} \right\rVert
	+ (\alpha - 1)^2 \gamma \left\lVert Ax_{1} - b \right\rVert
	+ \dfrac{1}{\kappa} \left( \left\lvert \omega_{0} \right\rvert + \left\lvert \omega_{1} \right\rvert \right) \sqrt{\dfrac{2 \E_{1} \left( x_{*} , \lambda_{*} \right)}{\beta \gamma}} < + \infty , \label{C-D:bnd}
	\end{align}
	where we also recall that, due to Proposition \ref{prop:bnd}, it holds $\sup\limits_{k \geq 1} \left\lVert \nu_{k} \right\rVert < + \infty$. 

Inequality \eqref{C-D:bnd} holds for every $K \geq 1$ (notice that $\Cnu$ is independent of $K$), consequently, we can apply Lemma \ref{lem:rate-bnd} to conclude that $\left\lVert \mysum_{k = 1}^{K} \zeta_{k} \right\rVert \leq \Cnu$ for every $K \geq 1$. By using again the triangle inequality and \eqref{C-D:bnd}, we obtain for every $K \geq 1$ that
	\begin{equation}
	\label{C-D:app}
	\left( \delta - 1 \right) K \left\lVert \zeta_{K+1} \right\rVert \leq \left\lVert \left( \left( \delta - 1 \right) K + \delta \right) \zeta_{K+1} \right\rVert \leq 2 \Cnu .
	\end{equation}
Using the inequality \eqref{bound:k} in Lemma \ref{lem:lim}, we see that for every $K \geq 1$ it holds
	\begin{equation}
	\label{C-D:est}
	\dfrac{t_{K+1}^{2}}{\left( 1 + \varphi_{m} \right) ^{2}} \left\lVert Ax_{K+1} - b \right\rVert 
	\leq K^{2} \left\lVert Ax_{K+1} - b \right\rVert
	\leq \left( \delta - 1 \right) K \left\lVert \zeta_{K+1} \right\rVert .
	\end{equation}
By combining \eqref{C-D:app} and \eqref{C-D:est}, we obtain \eqref{C-D:inq:feas}.

Statement  \eqref{C-D:inq:fun} follows from \eqref{C-D:inq:Lag}  and \eqref{C-D:inq:feas} by repeating the arguments at the end of the proof of Theorem \ref{thm:rate:Nes}.
\end{proof}

\subsubsection{The Attouch-Cabot (\cite{Attouch-Cabot:18}) rule}

Another inertial parameter rule used in the literature in the context of fast numerical algorithms is the one proposed by Attouch and Cabot in \cite{Attouch-Cabot:18}, which reads for $\alpha \geq 3$
\begin{equation*}
t_{k} := \dfrac{k-1}{\alpha - 1} \quad \forall k \geq 1.
\end{equation*}
This sequence is monotonically increasing and it fulfills \eqref{assume:t-k+} with $m := \dfrac{2}{\alpha - 1}  \leq 1$ as, for every $k \geq 1$, it holds
	\begin{align*}
	t_{k+1}^{2} - m t_{k+1} - t_{k}^{2}
	& = \left( t_{k+1} - t_{k} \right) \left( t_{k+1} + t_{k} \right) - m t_{k+1} 
	= \dfrac{1}{\alpha - 1} \dfrac{2k - 1}{\alpha - 1}  -  \dfrac{2}{\alpha - 1} \dfrac{k}{\alpha - 1} \nonumber \\
	& =- \dfrac{1}{\left( \alpha - 1 \right) ^{2}}  < 0.
	\end{align*}
This shows that the sequence $\left\lbrace t_{k} \right\rbrace _{k \geq 1}$ has very much in common with the Chambolle-Dossal parameter rule. The only significant difference is that is starts at $0$ and that $t_k \geq 1$ holds only for $k \geq k_1:=\lfloor \alpha \rfloor +1$. Consequently, the fast convergence rate results for the primal-dual gap, the feasibility measure and the objective function value are valid also for the Attouch-Cabot rule. This can be easily seen by slightly adapting the proofs made in the setting of the Chambolle-Dossal rule by taking into consideration that some of the estimates hold only for $k \geq k_{1}$. This exercise is left to the reader.

\section{Convergence of the iterates}

In this section we will turn our attention to the convergence of the sequence of primal-dual iterates generated by Algorithm \ref{algo:fast} to a primal-dual solution of \eqref{intro:pb}. First, we will prove that the first assumption in the Opial Lemma is verified and to this end we will need the following technical lemma.

\begin{lem}
	\label{lem:sum}
	Let $\left\lbrace \theta_{k} \right\rbrace _{k \geq 1},\left\lbrace a_{k} \right\rbrace _{k \geq 1}, \left\lbrace t_{k} \right\rbrace _{k \geq 1}$ be real sequences such that $\left\lbrace a_{k} \right\rbrace _{k \geq 1}$  is bounded from below and $\left\lbrace t_{k} \right\rbrace _{k \geq 1}$ is nondecreasing and bounded from below by $1$, and $\left\lbrace d_{k} \right\rbrace _{k \geq 1}$ be a nonnegative sequence such that for every $k \geq 1$
	\begin{subequations}
		\begin{align}
		a_{k+1}					& \leq a_k + \theta_{k+1}, \label{sum:inq:a} \\
		t_{k+1} \theta_{k+1}	& \leq \left( t_{k} - 1 \right) \theta_{k} + d_{k}. \label{sum:inq:theta}
		\end{align}
	\end{subequations}
	If  $\sum_{k \geq 1} d_{k} < + \infty$, then the sequence $\left\lbrace a_{k} \right\rbrace _{k \geq 1}$ is convergent.
\end{lem}
\begin{proof}
It follows from \eqref{sum:inq:theta} that for every $k \geq 1$
	\begin{equation}
	\label{lim:quasi-Fej}
	t_{k+1} \theta_{k+1} \leq \left( t_{k} - 1 \right) \theta_{k} + d_{k} \leq \left( t_{k} - 1 \right) \left[ \theta_{k} \right] _{+} + d_{k},
	\end{equation}
where $[\cdot]_+$ denotes the positive part. Since the right-hand side of this inequality is nonnegative, it yields that for every $k \geq 1$
	\begin{equation*}
	\left[ \theta_{k} \right] _{+}  \leq t_{k} \left[ \theta_{k} \right] _{+} - t_{k+1} \left[ \theta_{k+1} \right] _{+} + d_{k} .
	\end{equation*}	
which, by telescoping cancellation,  gives $\sum_{k \geq 1} \left[ \theta_{k} \right] _{+} < + \infty$.

According to \eqref{sum:inq:a}, we have that for every $k \geq 1$ it holds
	\begin{equation*}
	a_{k+1} \leq a_{k} + \theta_{k+1} \leq a_{k} + \left[ \theta_{k+1} \right] _{+}.
	\end{equation*}
By using Lemma \ref{lem:quasi-Fej} we obtain from here that the sequence $\left\lbrace a_{k} \right\rbrace _{k \geq 1}$ is convergent.
\end{proof}

\begin{prop}
	\label{prop:lim}
	Let $\left\lbrace \left( x_{k} , \lambda_{k} \right) \right\rbrace _{k \geq 0}$ be the sequence generated by Algorithm \ref{algo:fast} with $0 < m < \gamma < 1$. Then for every $ \left( x_{*} , \lambda_{*} \right) \in \sol$ the limit $\lim\limits_{k \to + \infty} \left\lVert \left( x_{k} , \lambda_{k} \right)  - \left( x_{*} , \lambda_{*} \right) \right\rVert _{\W}$ exists.
\end{prop}
\begin{proof}
Let $\left( x_{*} , \lambda_{*} \right) \in \sol$  be fixed. For brevity we will write
\begin{equation*}
u_{*} := \left( x_{*} , \lambda_{*} \right) \in \sol \ \mbox{and} \ u_{k} := \left( x_{k} , \lambda_{k} \right) \in \sH \times \sG \quad \forall k \geq 0.
\end{equation*}	
It follows from \eqref{sum:inq} that $\E_{k+1} \left( x_{*} , \lambda_{*} \right) \leq \E_{k} \left( x_{*} , \lambda_{*} \right)$ for every $k \geq 1$. In view of \eqref{bnd:E}, after rearranging some terms, we get for every $k \geq 1$
	\begin{align}
	& t_{k+1} \left( t_{k+1} - 1 + \gamma \right) \left( \Lb \left( x_{k+1} , \lambda_{*} \right) - \Lb \left( x_{*} , \lambda_{k+1} \right) + \dfrac{1}{2} \left\lVert u_{k+1} - u_{k} \right\rVert _{\W}^{2} \right) \nonumber \\
	& + \dfrac{\gamma}{2} t_{k+1} \left( \left\lVert u_{k+1} - u_{*} \right\rVert _{\W}^{2} - \left\lVert u_{k} - u_{*} \right\rVert _{\W}^{2} \right)
	+ \dfrac{1 - \gamma}{2 \rho} t_{k+1} \left\lVert \lambda_{k+1} - \lambda_{k} \right\rVert ^{2} \nonumber \\
	\leq \ & \left( t_{k} - 1 \right) \left( t_{k} - 1 + \gamma \right) \left( \Lb \left( x_{k} , \lambda_{*} \right) - \Lb \left( x_{*} , \lambda_{k} \right) + \dfrac{1}{2} \left\lVert u_{k} - u_{k-1} \right\rVert _{\W}^{2} \right) \nonumber \\
	& + \dfrac{\gamma}{2} \left( t_{k} - 1 \right) \left( \left\lVert u_{k} - u_{*} \right\rVert _{\W}^{2} - \left\lVert u_{k-1} - u_{*} \right\rVert _{\W}^{2} \right)
	+ \dfrac{1 - \gamma}{2 \rho} \left( t_{k} - 1 \right) \left\lVert \lambda_{k} - \lambda_{k-1} \right\rVert ^{2} \nonumber \\
	& + \left( t_{k} - 1 + \gamma \right) \left( \Lb \left( x_{k} , \lambda_{*} \right) - \Lb \left( x_{*} , \lambda_{k} \right) \right)
	+ \dfrac{1}{2} \left( t_{k+1} - 1 + \gamma \right) \left\lVert u_{k+1} - u_{k} \right\rVert _{\W}^{2} \nonumber \\
	& + \dfrac{1 - \gamma}{2 \rho} \left\lVert \lambda_{k+1} - \lambda_{k} \right\rVert ^{2} . \label{lim:inq}
	\end{align}
Set $a_{0}:= \dfrac{\gamma}{2} \left\lVert u_{0} - u_{*} \right\rVert _{\W}^{2} \geq 0$ and for every $k \geq 1$
	\begin{align*}
a_{k}	 := & \dfrac{\gamma}{2} \left\lVert u_{k} - u_{*} \right\rVert _{\W}^{2} \geq 0, \nonumber \\
	\theta_{k}	 := & \left( t_{k} - 1 + \gamma \right) \left( \Lb \left( x_{k} , \lambda_{*} \right) - \Lb \left( x_{*} , \lambda_{k} \right) + \dfrac{1}{2} \left\lVert u_{k} - u_{k-1} \right\rVert _{\W}^{2} \right) \nonumber \\
	& +  \left( a_{k} - a_{k-1} \right)
	+ \dfrac{1 - \gamma}{2 \rho} \left\lVert \lambda_{k} - \lambda_{k-1} \right\rVert ^{2} , \nonumber \\
	d_{k}  := & \left( t_{k} - 1 + \gamma \right) \left( \Lb \left( x_{k} , \lambda_{*} \right) - \Lb \left( x_{*} , \lambda_{k} \right) \right)
	+ \dfrac{1}{2} \left( t_{k+1} - 1 + \gamma \right) \left\lVert u_{k+1} - u_{k} \right\rVert _{\W}^{2} \nonumber \\
	& + \dfrac{1 - \gamma}{2 \rho} \left\lVert \lambda_{k+1} - \lambda_{k} \right\rVert ^{2} \geq 0.
	\end{align*}
We notice that for every $k \geq 1$ the estimate \eqref{lim:inq} becomes \eqref{sum:inq:theta}, while \eqref{sum:inq:a} obviously holds. As $0 < m < \gamma < 1$, it follows from Proposition \ref{prop:sum} that $\sum_{k \geq 1} d_{k} < + \infty$.

Hence, we can apply Lemma \ref{lem:sum} to conclude that $\left\lbrace \left\lVert \left( x_{k} , \lambda_{k} \right)  - \left( x_{*} , \lambda_{*} \right) \right\rVert _{\W} \right\rbrace _{k \geq 1}$ is convergent.	
\end{proof}
The following result is the discrete counterpart of \cite[Theorem 4.7]{Bot-Nguyen} (see \eqref{ratesgrad}). Its proof is a direct consequence of Proposition \ref{prop:sum} and Proposition \ref{prop:dual}.
\begin{thm}
	\label{thm:KKT}
	Let $\left\lbrace \left( x_{k} , \lambda_{k} \right) \right\rbrace _{k \geq 0}$ be the sequence generated by Algorithm \ref{algo:fast} with the sequence $\{t_k\}_{k \geq 1}$ chosen to satisfy \eqref{assume:sup}, $0 < m < \gamma \leq 1$, $0 < \sigma < \dfrac{\gamma}{L + \gamma \beta \|A\|^2}$, $\beta >0$, and $ \left( x_{*} , \lambda_{*} \right) \in \sol$. Then it holds
	\begin{equation*}
	\left\lVert \nabla f \left( x_{k} \right) - \nabla f \left( x_{*} \right) \right\rVert = o \left( \dfrac{1}{\sqrt{k}} \right)
	\quad \textrm{ and } \quad
	\left\lVert A^* \lambda_k - A^* \lambda_* \right\rVert = o \left( \dfrac{1}{\sqrt{k}} \right)
	\textrm{ as } k \to + \infty .
	\end{equation*}
	consequently,
	\begin{equation*}
	\left\lVert \nabla_{x} \Lag \left( x_{k} , \lambda_{k} \right) \right\rVert = \left\lVert \nabla f \left( x_{k} \right) + A^* \lambda_k \right\rVert = o \left( \dfrac{1}{\sqrt{k}} \right)
	\textrm{ as } k \to + \infty,
	\end{equation*}
and
\begin{equation*}
	\left\lVert \nabla_{\lambda} \Lag \left( x_{k} , \lambda_{k} \right) \right\rVert = \left\lVert Ax_k - b \right\rVert  = o \left( \dfrac{1}{\sqrt{k}} \right)
	\textrm{ as } k \to + \infty .
	\end{equation*}
As seen in Section \ref{subsec:fast-rate}, if, in addition, $\{t_k\}_{k \geq 1}$ is chosen to satisfy Chambolle-Dossal or Attouch-Cabot rule and $m:= \frac{2}{\alpha-1}$, then
	\begin{equation*}
	\left\lVert \nabla_{\lambda} \Lag \left( x_{k} , \lambda_{k} \right) \right\rVert = \left\lVert Ax_k - b \right\rVert  = \bO \left( \dfrac{1}{k^{2}} \right)
	\textrm{ as } k \to + \infty .
	\end{equation*}
\end{thm}

Now we can prove the main theorem of this section establishing the convergence of the sequence of iterates generated by Algorithm \ref{algo:fast}.

\begin{thm}\label{maintheorem}
Let $\left\lbrace \left( x_{k} , \lambda_{k} \right) \right\rbrace _{k \geq 0}$ be the sequence generated by Algorithm \ref{algo:fast} with the sequence $\{t_k\}_{k \geq 1}$ chosen to satisfy \eqref{assume:sup}, $0 < m < \gamma < 1$, $0 < \sigma < \dfrac{\gamma}{L + \gamma \beta \|A\|^2}$ and $\beta >0$. Then the sequence $\left\lbrace \left( x_{k} , \lambda_{k} \right) \right\rbrace _{k \geq 0}$ converges weakly to a primal-dual optimal solution of \eqref{intro:pb}.
\end{thm}
\begin{proof}
From Proposition \ref{prop:lim} it follows that the limit $\lim\limits_{k \to + \infty} \left\lVert (x_k, \lambda_k) - (x_*,\lambda_*) \right\rVert$ exists for every $ \left( x_{*} , \lambda_{*} \right) \in \sol$. This proves the first condition of  Lemma \ref{lem:Opial}.

In order to prove condition (ii), let $\left( \tx , \tlambda \right) \in \sH \times \sG$ be an arbitrary weak sequential cluster point of $\left\lbrace \left( x_{k} , \lambda_{k} \right) \right\rbrace _{k \geq 0}$. This means that there exists a subsequence $\left\lbrace \left( x_{k_{n}} , \lambda_{k_{n}} \right) \right\rbrace _{n \geq 0}$ which converges weakly to $\left( \tx , \tlambda \right)$ as $n \to + \infty$.
According to Theorem \ref{thm:KKT} we have $\nabla f \left( x_{k} \right) + A^{*} \lambda_{k} \to 0$ and $Ax_k - b \to 0$ as $k \to + \infty$, hence,
	\begin{equation*}
	\nabla f \left( x_{k_{n}} \right) + A^{*} \lambda_{k_{n}} \to 0  \quad \mbox{and} \quad Ax_{k_{n}} - b \to 0 \qquad \textrm{ as } \qquad n \to + \infty.
	\end{equation*}	
Since the graph of the operator $\TL$ is sequentially closed in $\left( \sH \times \sG \right) ^{\mathrm{weak}} \times \left( \sH \times \sG \right) ^{\mathrm{strong}}$ (cf. \cite[Proposition 20.38]{Bauschke-Combettes:book}), it follows from here that
	\begin{equation*}
	\begin{dcases}
	\nabla f \left( \tx \right) + A^{*} \tlambda 	& = 0 \nonumber \\
	A \tx - b										& = 0 
	\end{dcases} .
	\end{equation*}
	In other words, $\left( \tx , \tlambda \right) \in \sol$ and the proof is complete.
\end{proof}

\begin{rmk}\label{finalremark}
If the sequence $\{t_k\}_{k \geq 1}$ is chosen to satisfy the Chambolle-Dossal or the Attouch-Cabot rule with 
$$\alpha >3, \quad  m:=\dfrac{1}{\alpha-2} < \gamma < 1, \quad 0 < \sigma < \dfrac{\gamma}{L + \gamma \beta \|A\|^2} \quad \mbox{and} \quad \beta >0,$$ then Theorem \ref{maintheorem} guarantees that the sequence $\left\lbrace \left( x_{k} , \lambda_{k} \right) \right\rbrace _{k \geq 0}$ converges weakly to a primal-dual optimal solution of \eqref{intro:pb}. This statement is in addition to the fast convergence rates of order $\bO \left( 1/k^{2} \right)$ for the primal-dual gap, the feasibility measure, and the objective function value.

If the sequence $\{t_k\}_{k \geq 1}$ is chosen to satisfy the Nesterov rule, then, as we have seen, the fast convergence rate results also hold, however, since in this setting $m=\gamma=1$, one cannot apply Theorem \ref{maintheorem} to obtain the convergence of the iterates. This is consistent with the unconstrained case for which it is also not known if the sequence of iterates generated by the fast gradient method with inertial parameters following the Nesterov rule converges.
\end{rmk}

{\bf Acknowledgements.} The authors are thankful to the two anonymous reviewers for their remarks and suggestions which have improved the quality of the paper.

\end{document}